\documentclass[11pt]{article}
\usepackage[numbers,sort&compress]{natbib}
\usepackage{appendix,diagbox,booktabs,float}
\usepackage{enumerate,makecell}
\usepackage{amscd}
\usepackage{amsmath}
\usepackage{latexsym}
\usepackage{amsfonts}
\usepackage{amssymb}
\usepackage{amsthm}
\usepackage{bm}
\usepackage{verbatim}
\usepackage{mathrsfs}
\usepackage{enumerate}
\usepackage{hyperref}

\theoremstyle{plain}
\theoremstyle{definition}\newtheorem{theorem}{Theorem}[section]
\theoremstyle{plain}\newtheorem{lemma}[theorem]{Lemma}
\theoremstyle{plain}\newtheorem{coro}[theorem]{Corollary}
\theoremstyle{plain}
\theoremstyle{remark}
\usepackage{xcolor}

\newcommand{\Div}{\mathrm{div}\,}
\newcommand{\B}{\Big}

\newcommand{\be}{\begin{equation}}
\newcommand{\ee}{\end{equation}}
 \newcommand{\ba}{\begin{aligned}}
 \newcommand{\ea}{\end{aligned}}

  \newcommand{\f}{\frac}
    
  \newcommand{\ben}{\begin{enumerate}}
   \newcommand{\een}{\end{enumerate}}

\newcommand{\ti}{\nabla}

\newcommand{\Rmnum}[1]{\expandafter\@slowromancap\romannumeral #1@}

\allowdisplaybreaks

\numberwithin{equation}{section}
\begin{document}
\title{
On space-time  derivative estimates for the magnetohydrodynamic   equations
  }
\author{Yanqing Wang\footnote{School of Mathematics and   Information Science, Zhengzhou University of Light Industry, Zhengzhou, Henan  450002,  P. R. China Email: wangyanqing20056@gmail.com},~
    Wei Wei\footnote{Corresponding author. School of Mathematics and Center for Nonlinear Studies,  Northwest University, Xi'an, Shaanxi 710127,  P. R. China.  Email: ww5998198@126.com },~  Gang Wu \footnote{School of Mathematical Sciences, University of Chinese Academy of Sciences, Beijing 100049, P. R.
China Email: wugang2011@ucas.ac.cn} ~and 
   Yulin Ye\footnote{School of Mathematics and Statistics,
		Henan University,
		Kaifeng, 475004,
		P. R. China. Email: ylye@vip.henu.edu.cn}      }
\date{}
\maketitle
\begin{abstract}
 In this paper,  we present the space-time derivative estimates of solutions to the
 MHD equations.  It is a generalization of a priori bounds  for the classical
Navier-Stokes system  by Duff in
\cite[Acta Math. 164, 1990]{[Duff]}  and gives an affirmative answer to a question proposed by Zheligovsky
in \cite[Mathematics. 9, 2021]{[Zheligovsky]}. In addition, we show that  $  u, B   \in L^{\f{q}{q-3}}(0,T;L^{q} (\mathbb{R}^{3}))$ with $  6\leq q\leq\infty $  and $ D ^{k}u,  D ^{k}B   \in L^{\f{q}{ q(k+1)-3}}(0,T;L^{q} (\mathbb{R}^{3})) $ for $k\geq1, 2\leq q\leq\infty$  in this system.

  \end{abstract}
\section{Introduction}
\label{intro}
\setcounter{section}{1}\setcounter{equation}{0}
The incompressible   magnetohydrodynamic (MHD for short) equations describe the evolution of electrically conducting fluids and arise in a wide range of applications in plasmas such as liquid metals and
electrolytes (see \cite{[Davidson]}).
MHD system
consisting of the Navier-Stokes equations and Maxwell's equations   is written as
\be\left\{\ba\label{MHD}
&u_{t} -\Delta u+   u\cdot\nabla u -B\cdot\ti B +\nabla \Pi= 0, \ \ \ &\text{in}\ \mathbb{R}^{3}\times (0,T),\\
&B_{t} -\Delta B +u\cdot\nabla B-B\cdot\nabla u = 0, \ \ \ \ \ \ \ \ \ \ \ \ \ \  &\text{in}\ \mathbb{R}^{3}\times (0,T),\\
&\Div u=\Div B=0,\ \ \ \ \ \ \ \ \ \ \ \ \ \ \ \ \ \ \ \ \ \ &\text{in}\ \mathbb{R}^{3}\times (0,T),
\ea\right.\ee
with the initial conditions
\be\label{inital data}
u(x,0)=u_0(x),\ \ \ B(x,0)=B_0(x).
\ee
Here $u$ represents the velocity field,   $B
$ stands for the magnetic field and $\Pi=\pi +\f12 |b|^2$ denotes the magnetic pressure with $\pi$ being the fluid pressure, respectively.

The basic energy balance law in system \eqref{MHD} reads
\be \ba\label{EI}
& \f12\|u (T)\|^{2}_{L^{2}(\mathbb{R}^3)}
+\f12\|B (T)\|^{2}_{L^{2}(\mathbb{R}^3)}
+\int_0^T\|\nabla u (s)\|^{2}_{L^{2}(\mathbb{R}^3)}ds+\int_0^T\|\nabla B (s)\|^{2}_{L^{2}(\mathbb{R}^3)}ds\\=& \f12\|u_{0} \|^{2}_{L^{2}(\mathbb{R}^3)}
+\f12\|B_{0}  \|^{2}_{L^{2}(\mathbb{R}^3)}.
\ea\ee
 The Leray-Hopf type
weak solutions of the MHD equations \eqref{MHD} were constructed by Sermange and   Temam in \cite{[ST]}. It is well known that weak solutions of the  MHD equations \eqref{MHD} satisfy
\be
u, B\in L^{\infty}(0,T;L^{2}(\mathbb{R}^{3}))\cap L^{2}(0,T;\dot{H}^{1}(\mathbb{R}^{3})).
\ee
Just as the tri-dimensional Navier-Stokes equations, the full regularity of weak solutions to the MHD equations  \eqref{MHD} remains open. A natural question of investigation is to exploit  a priori
 estimates   of   solutions to the MHD equations \eqref{MHD}. In this direction,
Kim \cite{[Kim]} derived  a priori bounds
of weak solutions to the MHD equations \eqref{MHD} on the  periodic domain and proved that
\be\ba
&\partial^{s}_{x}  {u}, \partial^{s}_{x}B \in L^{\f{2}{ 2s-1}}  (0,T;L^{2}(\mathbb{ T}^{3})); \\
&{u}, B \in L^{1}(0,T;L^{\infty}(\mathbb{T}^{3})).
 \ea\ee
 For the  MHD equations \eqref{MHD} on the whole space, it was shown by Wu in \cite{[Wu]} that
\be
  {u}, ~B \in L^{1}(0,T;L^{\infty}(\mathbb{R}^{3})).
  \ee
Recently, Zheligovsky \cite{[Zheligovsky]} addressed
a priori bounds for arbitrary-order space derivatives of the first-order
time derivative of
the MHD solutions to \eqref{MHD} with periodic boundary
conditions. There,
it was proved that
  \be  \partial_{t}\partial^{s}_{x}{u}, ~\partial_{t}\partial^{s}_{x} B \in L^{\f{2}{ 2s+3}}  (0,T;L^{2}(\mathbb{ T}^{3})).  \ee
In addition,
Zheligovsky   made a   statement that might be interpreted as Zheligovsky's
question (see \cite{[Zheligovsky]}):
\begin{quote}
According to the present paradigm, the action of viscosity and diffusivity hampers the
development of small-scale structures generated by the nonlinearity. Thus, bounding the
diffusive and nonlinear terms jointly may be expected to result in more accurate bounds
for a larger ``number of derivatives" (i.e., for a higher-index Sobolev space norm). We have
not achieved this when estimating the time derivatives of the MHD solutions: our bounds
for the nonlinear advective terms are for the same index norms, as for the dissipative
terms. Indirectly, this confirms that cancellation may be possible with the sum residing in a
higher-index Sobolev space; our estimations are then too conservative.\end{quote}

 The first objective of our paper
is to give a positive answer to Zheligovsky's issue.
Before formulating our results,  we recall derivative estimates of the Navier-Stokes equations as the sub-system of the MHD equations.
Foias, Guillope and Temam \cite{[FGT]} showed that weak solutions
of the 3D Navier-Stokes equations satisfying
   $\partial^{s}_{x}  {u} \in L^{\f{2}{ 2s-1}}  (0,T;L^{2}(\mathbb{T}^{3})).$   The regularity $  {u} \in L^{1}(0,T; L^{\infty}(\mathbb{T}^{3})) $ due to Tartar
was also stated in \cite{[FGT]}.   M\'alek,    Padula  and   R\r{ u}\v{z}i\v{c}ka proved that a weak solution $u$  is in energy class $\nabla  {u} \in L^{\f{q}{2q-3}}(0,T;L^{q}(\mathbb{T}^{3})),~2 \leq q<\infty, $ in \cite{[MPR]}.
Duff  \cite{[Duff]}    studied the space-time derivative estimates and  demonstrated  that   $ \partial_{ t}^{r}\partial_{x}^{s}  {u}  \in L^{\f{2}{4r+2s-1}}(0,T;L^{2}(\mathbb{R}^{3}))$.  Recently, it was shown in our another paper \cite{[WW]} that  $  u   \in L^{\f{q}{q-3}}(0,T;L^{q} (\mathbb{R}^{3}))$ with $  6\leq q\leq\infty $  and $ D ^{k}u   \in L^{\f{q}{ q(k+1)-3}}(0,T;L^{q} (\mathbb{R}^{3})) $ for $k\geq1, 2\leq q\leq\infty $  in the Navier-Stokes equations.

In addition,
 Constantin  \cite{[Constantin]}
 considered second derivative estimates of suitable weak solutions to the
3D Navier-Stokes equations and obtained that $ \nabla^{2}   {u} \in   L^{q}(0,T;L^{q}(\mathbb{T}^{3}))  $ for $q<\f43$. Constantin's derivative estimates  were refined to Lorentz spaces   $ \nabla^{2}   {u} \in   L^{\f43,\infty}(0,T;L^{\f43,\infty}(\Omega))$ by Lions in \cite{[Lions]}.
In a series of works \cite{[Vasseur],[CV],[VY]}, Vasseur and his coauthors generalized Constantin's and Lions's derivative estimates. In particular,
Constantin type derivative estimates on the whole spaces were established by
Vasseur in \cite{[Vasseur]}, where the  blow-up method and the galilean invariance of the transport part of the equation were exploited.
  Choi and Vasseur \cite{[CV]} extended Lions's  derivative estimates to
 the
fractional  case    $\nabla^{\beta} u \in L^{\f{4}{\beta+1},\infty}(0,T;L^{\f{4}{\beta+1},\infty}(\mathbb{R}^{3})),  ~\beta>1.$ Furthermore, Constantin's and Lions's  bound was improved to $ \nabla^{2}   {u} \in   L^{\f43,\ell}(0,T;L^{\f43,\ell}(\mathbb{R}^{3}))  $ for $\ell>\f43$ by
 Vassuer and Yang in \cite{[VY]}. There, the vorticity equation and   De Giorgi iteration were adapted. Yang \cite{[Yang]} recently generalized Constantin's result  to
the integral norms with different exponents in space and
time directions and demonstrated that $ \nabla^{n}   {u} \in   L^{p,\infty}(0,T;L^{q,\infty}(\mathbb{T}^{3})), 1/p+3/q=n+1, n\geq1, 0<p\leq q\leq\infty$. We refer the interested reader to  \cite{[Chae],[BG],[KO]} for further recent
derivative estimates  of  solutions for the Navier-Stokes equations and related models.

Now, we state our first result involving the time derivative estimates of solutions to the
MHD equations \eqref{MHD}.
\begin{theorem}\label{the1.2}Let the pair  $(u, B) $ be a smooth solution of    the   MHD system \eqref{MHD}. Then it is valid that for any nonnegative integer $k$,
 \be\label{timederivative}  u _{t}^{(k )}, B _{t}^{(k )}  \in L^{\f{2 }{4k-1}}(0,T;L^{2} (\mathbb{R}^{3})).\ee
\end{theorem}
Furthermore, we can show that
  \begin{theorem}\label{the1.1}
Assume that  the pair  $(u, B) $ is a smooth solution of    the  MHD equations  \eqref{MHD}. Then for each $T>0$, there holds
\be\label{1.1}\ba
   D ^{n }u^{(m)}_{t},     D ^{n }B^{(m)}_{t}\in L^{\f{2 }{4m +2n-1}}(0,T; L^{2}(\mathbb{R}^{3})),  \ea\ee
where $m$ and $n$ are  any nonnegative integer.
\end{theorem}
The aforementioned
 derivative estimates for the MHD   equations \eqref{MHD}
  answer  the question posed by Zheligovsky in \cite{[Zheligovsky]}.
As pointed out by
Zheligovsky \cite{[Zheligovsky]},
cancellation of nonlinear terms in the energy estimate
plays an important role in the proof of above theorems (see Lemma \ref{lema3.2} and \ref{lemma4.1}).
The deduction of the time derivation is more complicated than that of the spatial derivation.
Roughly speaking,
we can get the energy inequality in terms of $D ^{ s}\cdot $ and $D^{s+1}\cdot$ by applying the operator $D ^{2s}$ to the equations in one step (see e.g. \eqref{3.2v1}). However,
 it seems that it break down for the time derivative.
 We need two steps to get the energy inequalities based on
$\cdot _{t}^{(r)}$   and $\cdot ^{(r+1)}_{t}$ (see e.g. \eqref{3.13v2} and \eqref{3.24v1}).
Compared with previous works \cite{[WW],[Duff]}, coupling the velocity with the magnetic field is required to be dealt with. To the knowledge of authors, it seems to be the first extension of Duff's space-time derivative estimates to a system consisting of the Navier-Stokes equations and  the evolution of other physical quantity.

Moreover, the  cancellation of interaction  terms between the velocity and the magnetic field  seems to break down for the $L^{p}$ ($p\neq2$) type energy estimate in the   MHD   equations \eqref{MHD}. It is difficult to apply the
nonlinear $p$-Laplace operator to construct a test function used by M\'alek,    Padula  and   R\r{ u}\v{z}i\v{c}ka in \cite{[MPR]} to get $\nabla  {u}, \nabla  {B} \in L^{\f{q}{2q-3}}(0,T;L^{q}(\mathbb{T}^{3})),~2 \leq q<\infty. $ Here, in the spirit of our recent work \cite{[WW]}, as a by-product of Theorem \ref{the1.1},   we have the following general bound.
\begin{coro}\label{coro} Let the pair  $(u, B) $ be a smooth solution of    the   MHD system \eqref{MHD}. Then
$  u, B   \in L^{\f{q}{q-3}}(0,T;L^{q} (\mathbb{R}^{3}))$ with $  6\leq q\leq\infty $  and $ D ^{k}u,  D ^{k}B   \in L^{\f{q}{ q(k+1)-3}}(0,T;L^{q} (\mathbb{R}^{3})) $ for $k\geq 1, 2\leq q\leq\infty $.
\end{coro}
We extend M\'alek,    Padula  and   R\r{ u}\v{z}i\v{c}ka's  corresponding  estimate as well as its limiting case for solutions of the Navier-Stokes equations to those of the MHD equations \eqref{MHD}.

It is worth pointing out that the difference between the operators  $ D  $ and  $(-\Delta)^{1/2}$ is a Riesz operator. Hence, the classical Calder\'on-Zygmund Theorem allows us to replace $ D  $  by  $(-\Delta)^{1/2}$ in Theorem \ref{the1.1}  and Corollary \ref{coro}.

The present paper is built up as follows. Section 2 is devoted to some notations and
auxiliary lemmas. In Section 3, we study time derivative  estimates of solutions to the MHD equations. Section 4 is concerned with space-time
derivative estimates of solutions to the MHD equations.
\section{Preliminaries}
\label{pre}
\setcounter{section}{2}\setcounter{equation}{0}
The classical Sobolev norm $\|\cdot\|_{H^{s}}$  is defined as   $\|f\|^{2} _{{H}^{s}}= \int_{\mathbb{R}^{n}} (1+|\xi|)^{2s}|\hat{f}(\xi)|^{2}d\xi$, $s\in \mathbb{R}$.
  We denote by  $ \dot{H}^{s}$ homogenous Sobolev spaces with the norm $\|f\|^{2} _{\dot{H}^{s}}= \int_{\mathbb{R}^{n}} |\xi|^{2s}|\hat{f}(\xi)|^{2}d\xi$.
 For $q\in [1,\,\infty]$, the notation $L^{q}(0,\,T;\,X)$ stands for the set of measurable functions on the interval $(0,\,T)$ with values in $X$ and $\|f(t,\cdot)\|_{X}$ belonging to $L^{q}(0,\,T)$.
  To simplify the presentation,  we write
 $$  \|f\|_{L^p(\mathbb{R}^{3})}=\|f\|_{L^p}, \|f\|_{\dot{H}^{s}(\mathbb{R}^{3})}=\|f\|_{\dot{H}^{s}}.$$
 A vector $\alpha=(\alpha_{1},\alpha_{2},\cdots,\alpha_{n})$ is called a multi-index of order $|\alpha|=  \alpha_{1}+\alpha_{2}+\cdots+\alpha_{n},$ where $\alpha_{1}, \alpha_{2},\cdots,\alpha_{n}$ are nonnegative integer. Let
$k$ be a nonnegative integer,
$$
D^{k}u(x)=\{D^{\alpha}u(x): |\alpha|=k\},
$$
with
$$
D^{\alpha}f(x)=\f{\partial^{|\alpha|}f(x)}{\partial x_{1}^{\alpha_{1}}\partial x_{2}^{\alpha_{2}}\cdots\partial x_{n}^{\alpha_{n}}}=\partial _{x_{1}}^{\alpha_{1}}\partial_{x_{2}}^{\alpha_{2}}\cdots\partial_{x_{n}}^{\alpha_{n}}f(x).
$$
As a special case, we have $Df=\nabla f$.
\begin{lemma}(\cite{[Duff],[WW]})\label{hofflemma}
Let $F(t)\geq0$  and $F(t)\in L^{\f{1}{p}}(0,T)$. Suppose that  $G(t)\geq0$  and
\be\label{2.1}
\f{d}{dt}F(t)+G(t)\leq CF(t)^{\f{q}{p}}.
\ee
Then for $q>p+1$, there holds $G(t)\in L^{\f{1}{q}}(0,T)$.

\end{lemma}
\begin{proof}
Rearranging the inequality \eqref{2.1}, we have
\be
\f{\f{d}{dt}F(t)}{(F(t)+1)^{\f{q-1}{p }}}+\f{G(t)}{(F(t)+1)^{\f{q-1}{p }}}\leq C[F(t)+1]^{\f{1}{p}}.
\ee
A simple  calculation  leads to $$\int_{0}^{T}\f{G (t)}{(1+F(t))^{\f{q-1}{p }}}dt<\infty.$$
Then we deduce from the H\"older inequality that
$$\ba
\int_{0}^{T}G^{\f{1}{q}}(t)dt=&\int_{0}^{T}(1+F(t))^{\f{q-1}{pq}}
\f{G^{\f{1}{q}}(t)}{(1+F(t))^{\f{q-1}{pq}}}dt\\
\leq& \B[\int_{0}^{T}(1+F(t))^{\f1p}dt \B]^{\f{q-1}{ q}}\B[\int_{0}^{T}\f{G (t)}{(1+F(t))^{\f{q-1}{p }}}dt\B]^{\f1q}.
\ea$$
This lemma is  thus proved.
\end{proof}
\begin{lemma} Let  $s>\frac{3}{2}$ and $u\in \dot{H}^{1} \cap \dot{H}^{s} $. Then there holds
\be\label{keyinequality}\|\nabla u \|_{L^{3}}\leq C\| u \|^{\f{2(s+1)  -5 }{2 (s-1)}}_{\dot{H}^{1}}\| u \|_{\dot{H}^{s}}^{\f{1}{2 (s-1)}}.
\ee
\end{lemma}\begin{proof}
In view of the Bernstein inequality, we see that
\be\ba
\|\nabla u \|_{L^{3}}\leq& C\sum_{j<N}\|\dot{\Delta}_{j}\nabla u \|_{L^{ 3 } } +C\sum_{j\geq N}\|\dot{\Delta}_{j}\nabla u \|_{L^{ 3 }}\\
\leq& C\sum_{j<N}2^{j[1+3(\f12-\f{1}{3})-\alpha
]}\|\dot{\Delta}_{j}\nabla u \|_{L^{2}}+C\sum_{j\geq N}2^{j[1+3(\f12-\f{1}{3})-s
]}2^{js }\|\dot{\Delta}_{j}\nabla u \|_{L^{2}},
\ea\ee
where $\dot{\Delta}_{j}$ is the standard homogeneous dyadic block.\\
We compute
\be\ba
\|\nabla u \|_{L^{3}}
\leq  C 2^{N(
\f52-2
)}\|  u \|_{\dot{H}^{1}}+C2^{N[
\f52-(s+1)
]}\|  u \|_{\dot{H}^{s}}.
\ea\ee
To bound the right hand side of this inequality, we choose $N$ such that
$$ 2^{N(
\f52-2
)}\|  u \|_{\dot{H}^{1}}\approx 2^{N[
\f52-(s+1)
]}\|  u \|_{\dot{H}^{s}}.$$
Hence, we get \eqref{keyinequality}. The proof of this lemma is completed.
\end{proof}
Denote the operator $\Lambda=(-\Delta)^{\frac{1}{2}}$. We end this section with a homogeneous version of Kato-Ponce commutator estimate and its corollary, a fractional Leibniz rule.
\begin{lemma}(\cite[Corollary 5.2, p. 68]{[Li]})\label{odelemma}
Let $s>0$ and $1< p< \infty$. Then for any two Schwartz functions $f,g$ on $\mathbb{R}^{d}$, there holds
\be\label{katoponce}
\|\Lambda^{s}(fg)-f\Lambda^{s} g\|_{L^{p } }\leq
C\|\nabla f\|_{L^{p_{1}  } }\|\Lambda^{s-1}g\|_{L^{p_{2}} }
+C\|\Lambda^{s} f\|_{L^{p_{3}} }\|  g\|_{L^{p_{4} } },
\ee
where $1<p_{1}, p_{2}, p_{3}, p_{4}\leq\infty $ and $\f{1}{p }=\f{1}{p_{1}} +\f{1}{p_{2}}= \f{1}{p_{3}} +\f{1}{p_{4}}$.
\end{lemma}
\begin{lemma}
Let $s>0$ and $\Lambda^{s}f\in L^{p_{1} }(\mathbb{R}^{d}), g \in L^{p_{2} }(\mathbb{R}^{d}), f\in L^{p_{3} }(\mathbb{R}^{d}), \Lambda^{s}g \in L^{p_{4} }(\mathbb{R}^{d})$.  Assume  that $1<p_{1},p_{2},p_{3},p_{4}\leq\infty$ and $1<p<\infty$. Then there holds
\be\label{fraleibnilaw}
\|\Lambda^{s}(fg) \|_{L^{p } }\leq
C\|\Lambda^{s}f\|_{L^{p_{1} } }
\|g\|_{L^{p_{2} } }
+C\|  f\|_{L^{p_{3} } }\|\Lambda^{s}g\|_{L^{p_{4} } },
\ee
where
\be\label{fraleibnilawcon}\f{1}{p}=\f{1}{p_{1}}+\f{1}{p_{2}}
=\f{1}{p_{3}}+\f{1}{p_{4}}.
\ee\end{lemma}
\section{Time derivative estimates of solutions to the MHD
equations}
\label{estimate}
\setcounter{section}{3}\setcounter{equation}{0}

The purpose of this section is to deduce time derivative estimates of solutions to the MHD
equations. We begin with some energy estimate based on $u _{t}, B _{t}, D   u _{t}, D   B_{t}$ in the MHD
equations. Then, we present energy estimate of $u _{t}^{(r)}, B _{t}^{(r)}$, $D u _{t}^{(r)}, DB _{t}^{(r)}$ in this system, where $r$
is any   nonnegative integer. With these  inequality in hand, we can get the desired estimates  by induction.

\begin{lemma}Let the pair  $(u, B)$ satisfy   the  MHD  equations
\eqref{MHD}. Then, for any positive $t$  and $\eta$, there exists a positive constant $C$ such that
\be\label{3.2v1}
\f{d}{dt}(\| D  u (t)\|^{2}_{L^{2}}+\| D  B (t)\|^{2}_{L^{2}})
+ \f32\| D ^{2}u (t)\|_{L^{2}}^{2}+ \f32\| D ^{2}B (t)\|_{L^{2}}^{2}
\leq C(\| D  u (t)\|_{L^{2}}^{2}+\| D  B (t)\|_{L^{2}}^{2})^{3},
\ee
\be\ba\label{3.6v1}
&\f{d}{dt}(\| D  u (t)\|^{2}_{L^{2}}+\| D  B (t)\|^{2}_{L^{2}})
+ \f{1}{2}\| D ^{2}u (t)\|^{2}_{L^{2}}+\f{1}{2} \| u _{t}\|_{L^{2}}^{2}+ \f{1}{2} \|B_{t}\|^{2}_{L^{2}}+ \f{1}{2} \| D ^{2}B (t)\|^{2}_{L^{2}}\\
\leq&  C(\| D  u (t)\|_{L^{2}}^{2}+\| D  B (t)\|_{L^{2}}^{2})^{3},
\ea
\ee
\be\ba\label{3.13v2}
\f{d}{dt}(\| u _{t}\|^{2}_{L^{2}}+\|B _{t}\|^{2}_{L^{2}})+\f12\| D  u _{t}\|_{L^{2}}^{2}+&\f12\| D  B _{t}\|_{L^{2}}^{2}
\\\leq& C\|u _{t}\|^{\f{10}{3 }}_{L^{2}}+C\|B _{t}\|^{\f{10}{3 }}_{L^{2}}+C\| D  u \|^{  10  }_{L^{2}}  +C\| D  B \|^{  10  }_{L^{2}}\\&+4\eta(\| D  u \|_{L^{2}}^{2}+\| D  B \|_{L^{2}}^{2})^{2}
(\| D ^{2}u \|^{2}_{L^{2}}+\| D ^{2}B \|^{2}_{L^{2}}),\ea\ee
\be\ba\label{3.24v1}
&\f{1}{2}\f{d}{dt}(\| D   u _{t}\|^{2}_{L^{2}}+\| D   B _{t}\|^{2}_{L^{2}})+\f12\| D ^{2 }u _{t}\|_{L^{2}}^{2}+\f12\| D ^{2 }B _{t}\|_{L^{2}}^{2}+ \f12(\| u _{tt}\|_{L^{2}}^{2}+\| B _{tt}\|_{L^{2}}^{2})
\\
\leq& C(\| D   u _{t}\|^{2}_{L^{2}}+\| D   B _{t}\|^{2}_{L^{2}})^{\f{7}{5}}
+C(\| D  u (t)\|^{2}_{L^{2}}+\| D  B (t)\|^{2}_{L^{2}})^{7}
\\&+8\eta(\| D ^{2}u (t)\|_{L^{2}}^{2}+ \| D ^{2}B (t)\|_{L^{2}}^{2})(\| D  B \|_{L^{2}}^{2}+\| D  u\|_{L^{2}}^{2})^{4}.
 \ea
\ee
\end{lemma}
\begin{proof}
(1) Multiplying  the velocity field  and the magnetic field equations in the MHD equations \eqref{MHD}
  by $ D ^{2}u$ and $ D ^{2}B$ respectively, we infer that
 \be\label{uH1}
\f{1}{2}\f{d}{dt}\| D  u (t)\|^{2}_{L^{2}}
+ \| D ^{2}u (t)\|^{2}_{L^{2}}= -\int_{\mathbb{R}^{3}}
u (t)\cdot\nabla u (t)\cdot D ^{2}u (t)dx+\int_{\mathbb{R}^{3}}
B (t)\cdot\nabla B (t)\cdot D ^{2}u (t)dx, \ee
 \be\label{BH1}
\f{1}{2}\f{d}{dt}\| D  B (t)\|^{2}_{L^{2}}
+ \| D ^{2}B (t)\|^{2}_{L^{2}}= -\int_{\mathbb{R}^{3}}
u (t)\cdot\nabla B (t)\cdot D ^{2}B (t)dx+\int_{\mathbb{R}^{3}}
B (t)\cdot\nabla u (t)\cdot D ^{2}B (t)dx. \ee
 Summarizing \eqref{uH1} and \eqref{BH1}, we
directly obtain
 \be\label{H1}\ba
&\f{1}{2}\f{d}{dt}(\| D  u (t)\|^{2}_{L^{2}}+\| D  B (t)\|^{2}_{L^{2}})
+ \| D ^{2}u (t)\|^{2}_{L^{2}}+ \| D ^{2}B (t)\|^{2}_{L^{2}}
\\=& -\int_{\mathbb{R}^{3}}
u (t)\cdot\nabla u (t)\cdot D ^{2}u (t)dx+\int_{\mathbb{R}^{3}}
B (t)\cdot\nabla B (t)\cdot D ^{2}u (t)dx\\
 &-\int_{\mathbb{R}^{3}}
u (t)\cdot\nabla B (t)\cdot D ^{2}B (t)dx+\int_{\mathbb{R}^{3}}
B (t)\cdot\nabla u (t)\cdot D ^{2}B (t)dx.
\ea\ee
In view of  the H\"older inequality, one gets
$$\ba
\B|\int_{\mathbb{R}^{3}}u (t) \cdot\nabla u (t)\cdot D ^{2}u (t)dx\B|
&\leq\|u (t)\|_{L^{6}}\|\nabla u (t)\|
_{L^{3}}
\| D ^{2}u (t)\|_{L^{2}}.
\ea$$
 Thanks to the Sobolev embedding and the
Gagliardo-Nirenberg inequality \eqref{keyinequality}, we infer that
$$
\|u (t)\|_{L^{6} }\leq
C\| D  u (t)\|_{L^{2}},
$$
and
$$
\|\nabla u (t)\|_{L^{3} }\leq
C\| D  u (t)\|_{L^{2}}^{\f{1}{2}}
\| D ^{2}u (t)\|^{\f{1}{2}}_{L^{2}}.
$$
As a consequence, we obtain
$$\ba
\B|\int_{\mathbb{R}^{3}}u (t) \cdot\nabla u (t)\cdot D ^{2}u (t)dx\B|
 &\leq
C\| D  u (t)\|_{L^{2}}
\| D  u (t)\|_{L^{2}}^{\f{1}{2}}
\| D ^{2}u (t)\|^{\f{3}{2}}_{L^{2}}
\\
&\leq
\f{1}{32} \| D ^{2}u (t)\|_{L^{2}}^{2}
+C
\| D  u (t)\|_{L^{2}}^{6},\ea$$
 where we  used    Young's inequality.\\
Proceeding as above, we discover that
$$\ba
\B|\int_{\mathbb{R}^{3}}
B (t)\cdot\nabla B (t)\cdot D ^{2}u (t)dx \B| \leq&\|B (t)\|_{L^{6}}\|\nabla B (t)\|
_{L^{3}}
\| D ^{2}u (t)\|_{L^{2}}\\
\leq&\| D  B (t)\|_{L^{2}}^{\f{3}{2}}
\| D ^{2}B (t)\|^{\f{1}{2}}_{L^{2}}\| D ^{2}u (t)\|_{L^{2}}\\
\leq&
\f{1}{32} (\| D ^{2}u (t)\|_{L^{2}}^{2}+\| D ^{2}B (t)\|_{L^{2}}^{2})
+C
\| D  B (t)\|_{L^{2}}^{6},
\ea$$
and
$$\ba
&\B|-\int_{\mathbb{R}^{3}}
u (t)\cdot\nabla B (t)\cdot D ^{2}B (t)dx+\int
B (t)\cdot\nabla u (t)\cdot D ^{2}B (t)dx\B|\\
\leq&
\f{1}{32} (\| D ^{2}u (t)\|_{L^{2}}^{2}+\| D ^{2}B (t)\|_{L^{2}}^{2})
+C
(\| D  B (t)\|_{L^{2}}^{6}+\| D  u (t)\|_{L^{2}}^{6}).
\ea$$
Substituting  this   into estimate
\eqref{H1}, we observe that
\be\label{3.2}
\f12\f{d}{dt}(\| D  u (t)\|^{2}_{L^{2}}+\| D  B (t)\|^{2}_{L^{2}})
+ \f34\| D ^{2}u (t)\|_{L^{2}}^{2}+ \f34\| D ^{2}B (t)\|_{L^{2}}^{2}
\leq C(\| D  u (t)\|_{L^{2}}^{2}+\| D  B (t)\|_{L^{2}}^{2})^{3}.
\ee

(2) Testing the MHD equations
$\eqref{MHD}_{1}$  with $u _{t}$ and $\eqref{MHD}_{2}$  with $B _{t}$, we
  integrate
  in the space variable and sum to obtain
\be\label{grawall}\ba
&\f12\f{d}{dt}(\| D  u (t)\|^{2}_{L^{2}}+\| D  B (t)\|^{2}_{L^{2}})
+ \| u _{t}\|_{L^{2}}^{2}+ \| B _{t}\|_{L^{2}}^{2}
\\=&-\int_{\mathbb{R}^{3}}u\cdot\nabla u\cdot u _{t}dx+\int_{\mathbb{R}^{3}} B\cdot\nabla B\cdot u _{t}dx -\int_{\mathbb{R}^{3}}u\cdot\nabla B\cdot B _{t}dx+\int_{\mathbb{R}^{3}} B\cdot\nabla u\cdot B _{t}dx.
\ea\ee
With the help of the H\"older inequality, Sobolev inequality   and interpolation inequality \eqref{keyinequality}, we know that
$$\ba
 \B|\int_{\mathbb{R}^{3}}u\cdot\nabla u\cdot u _{t}dx\B|\leq&\|u _{t}\|_{L^{2}}\|u  \|_{L^{6}}\|\nabla u \|_{L^{3}}\\\leq& C \|u _{t}\|_{L^{2}} \| D  u  \|_{L^{2}} \| D  u (t)\|_{L^{2}}^{\f{1}{2}}
\| D ^{2}u (t)\|^{\f{1}{2}}_{L^{2}}
\\\leq&   \varepsilon\|u _{t}\|_{L^{2}}^{2} +  \| D  u (t)\|_{L^{2}}^{\f{3 }{ 2}}
\| D ^{2}u (t)\| _{L^{2}}\\\leq&    \varepsilon\|u _{t}\|_{L^{2}}^{2} + C\| D  u (t)\|_{L^{2}}^{6}+
\delta\| D ^{2}u (t)\|^{2} _{L^{2}},\ea
$$
where the Young inequality was used and $\varepsilon, \delta$ will be determined later.

Along the exact same lines as above, we have
$$
 \B|\int_{\mathbb{R}^{3}}B\cdot\nabla B\cdot u _{t}dx \B|\leq  \varepsilon\|u _{t}\|_{L^{2}}^{2} + C\| D  B (t)\|_{L^{2}}^{6}+
\delta\| D ^{2}B (t)\|^{2} _{L^{2}},
$$
and
$$\ba
& \B|-\int_{\mathbb{R}^{3}}u\cdot\nabla B\cdot B _{t}dx+\int_{\mathbb{R}^{3}} B\cdot\nabla u\cdot B _{t}dx \B|\\
\leq&  \varepsilon\|B _{t}\|_{L^{2}}^{2} + C(\| D  B (t)\|_{L^{2}}^{6}+\| D  u (t)\|_{L^{2}}^{6})+
 \delta(\| D ^{2}B (t)\|^{2} _{L^{2}}+\| D ^{2}B (t)\|^{2} _{L^{2}}).
\ea$$
Plugging these estimates into \eqref{grawall}, we arrive at
\be\ba
&\f12\f{d}{dt}(\| D  u (t)\|^{2}_{L^{2}}+\| D  B (t)\|^{2}_{L^{2}})
+ \| u _{t}\|_{L^{2}}^{2}+ \| B _{t}\|_{L^{2}}^{2}\\
\leq&2\varepsilon(\|u _{t}\|_{L^{2}}^{2} +\|B _{t}\|_{L^{2}}^{2} ) + C(\| D  u (t)\|_{L^{2}}^{6}+\| D  B (t)\|_{L^{2}}^{6}) +2\delta(
 \| D ^{2}u (t)\|^{2} _{L^{2}}+
  \| D ^{2}B (t)\|^{2} _{L^{2}}),
\ea\ee
which together with \eqref{3.2} implies that
\be\ba
&\f{d}{dt}(\| D  u (t)\|^{2}_{L^{2}}+\| D  B (t)\|^{2}_{L^{2}})+ \f34 \| D ^{2}u (t)\|^{2}_{L^{2}}
+  \f34\| D ^{2}B (t)\|^{2}_{L^{2}}+ \| u _{t}\|_{L^{2}}^{2}+ \| B_{t}\|_{L^{2}}^{2}\\
\leq &2\varepsilon(\|u _{t}\|_{L^{2}}^{2} +\|B _{t}\|_{L^{2}}^{2} ) + C(\| D  u (t)\|_{L^{2}}^{2}+\| D  B (t)\|_{L^{2}}^{2})^{3}+2\delta(
 \| D ^{2}u (t)\|^{2} _{L^{2}}+
  \| D ^{2}B (t)\|^{2} _{L^{2}}).
\ea
\ee
Taking $\varepsilon$ and $\delta$ small, we get
\be\ba\label{3.6}
&\f{d}{dt}(\| D  u (t)\|^{2}_{L^{2}}+\| D  B (t)\|^{2}_{L^{2}})
+ \f{1}{2}\| D ^{2}u (t)\|^{2}_{L^{2}}+\f{1}{2} \| u _{t}\|_{L^{2}}^{2}+ \f{1}{2} \|B_{t}\|^{2}_{L^{2}}+ \f{1}{2} \| D ^{2}B (t)\|^{2}_{L^{2}}\\
\leq&  C(\| D  u (t)\|_{L^{2}}^{2}+\| D  B (t)\|_{L^{2}}^{2})^{3}.
\ea
\ee

(3)
Differentiating  the MHD equations
\eqref{MHD} with respect to the time variable, we compute
\be\ba\label{3.7} u^{(2) }_{t}-\Delta u_{t}+u_{t}\cdot\nabla u+u\cdot\nabla u_{t}-B_{t}\cdot\nabla B-B\cdot\nabla B_{t}+\nabla \Pi_{t}=0, \\
B^{(2) }_{t}-\Delta B_{t}+u_{t}\cdot\nabla B+u\cdot\nabla B_{t}-B_{t}\cdot\nabla u-B\cdot\nabla u_{t}=0.
\ea
\ee
Taking the inner product of   $\eqref{3.7}_{1}$ with $u_{t}$  and $\eqref{3.7}_{2}$ with $B_{t}$,  integrating over the spatial variable,
one deduces  that
$$\ba
&\f{1}{2}\f{d}{dt}\| u _{t}\|^{2}_{L^{2}}+\| D  u _{t}\|_{L^{2}}^{2}=-\int_{\mathbb{R}^{3}}u_{t}\cdot\nabla u\cdot u_{t}dx+\int_{\mathbb{R}^{3}}B_{t}\cdot\nabla B\cdot u_{t}dx+\int_{\mathbb{R}^{3}}B\cdot\nabla B_{t}\cdot u_{t}dx,\\
&\f{1}{2}\f{d}{dt}\| B _{t}\|^{2}_{L^{2}}+\| D  B _{t}\|_{L^{2}}^{2}=-\int_{\mathbb{R}^{3}}u_{t}\cdot\nabla B\cdot B_{t}dx+\int_{\mathbb{R}^{3}}B_{t}\cdot\nabla u\cdot B_{t}dx+\int_{\mathbb{R}^{3}}B\cdot\nabla u_{t}\cdot B_{t}dx.
\ea
$$
By means of   of integration by parts, we get
$$\int_{\mathbb{R}^{3}}B\cdot\nabla B_{t}\cdot u_{t}dx=-\int_{\mathbb{R}^{3}} B\cdot\nabla u_{t}\cdot B_{t}dx.$$
As a consequence, we get
\be\ba\label{3.8}
&\f{1}{2}\f{d}{dt}(\| u _{t}\|^{2}_{L^{2}}+\| B _{t}\|^{2}_{L^{2}})+\| D  u _{t}\|_{L^{2}}^{2}+\| D  B _{t}\|_{L^{2}}^{2}\\=&-\int_{\mathbb{R}^{3}}u_{t}\cdot\nabla u\cdot u_{t}dx+\int_{\mathbb{R}^{3}}B_{t}\cdot\nabla B\cdot u_{t}dx-\int_{\mathbb{R}^{3}}u_{t}\cdot\nabla B\cdot B_{t}dx+\int_{\mathbb{R}^{3}}B_{t}\cdot\nabla u\cdot B_{t}dx.
\ea
\ee
In view of the H\"older inequality, Sobolev embedding  and interpolation inequality \eqref{keyinequality}, one derives from the Young  inequality that
$$\ba
\B|-\int_{\mathbb{R}^{3}}u_{t}\cdot\nabla B\cdot u_{t}dx\B|\leq&
\|u _{t}\|_{L^{2}}\|u_{t}  \|_{L^{6}}\|\nabla B \|_{L^{3}}\\
\leq& C\|u _{t}\|_{L^{2}}\|  D  u_{t}  \|_{L^{2}}\| D  B \|_{L^{2}}^{\f{1}{2}}
\| D ^{2}B \|^{\f{1}{2}}_{L^{2}}\\
 \leq& \varepsilon\| D  u_{t}  \|_{L^{2}}^{2}+C\|u _{t}\|^{2}_{L^{2}} \| D  B \|_{L^{2}}
\| D ^{2}B \| _{L^{2}},
\ea
$$
where $\varepsilon$
will be fixed later.\\
Thanks to Young inequality, we observe that
$$\ba
\|u _{t}\|^{2}_{L^{2}} \| D  B \|_{L^{2}}
\| D ^{2}B \| _{L^{2}}
 \leq C\|u _{t}\|^{\f{10}{3 }}_{L^{2}}+C\| D  B \|^{  10  }_{L^{2}}+\eta\| D  B \|_{L^{2}}^{4}
\| D ^{2}B \|^{2}_{L^{2}},
\ea$$
which turns out that
\be\ba\label{3.11}
\B|-\int_{\mathbb{R}^{3}}u_{t}\cdot\nabla B\cdot u_{t}dx\B|\leq&
   C\|u _{t}\|^{\f{10}{3 }}_{L^{2}}+C\| D  B \|^{  10  }_{L^{2}}+\eta\| D  B \|_{L^{2}}^{4}
\| D ^{2}B \|^{2}_{L^{2}}+ \varepsilon\| D  u_{t}  \|_{L^{2}}^{2}.
\ea\ee
Following the same path, we arrive at
\be\ba\label{3.12}
\B|-\int_{\mathbb{R}^{3}}u_{t}\cdot\nabla u\cdot u_{t}dx\B|\leq&C\|u _{t}\|^{\f{10}{3 }}_{L^{2}}+C\| D  u \|^{  10  }_{L^{2}}+\eta\| D  u \|_{L^{2}}^{4}
\| D ^{2}u \|^{2}_{L^{2}}+\varepsilon\| D  u_{t}  \|_{L^{2}}^{2},\\
\B|\int_{\mathbb{R}^{3}}B_{t}\cdot\nabla B\cdot u_{t}dx\B|\leq&C\|B _{t}\|^{\f{10}{3 }}_{L^{2}}+C\| D  B \|^{  10  }_{L^{2}}+\eta\| D  B \|_{L^{2}}^{4}
\| D ^{2}B \|^{2}_{L^{2}}+\varepsilon\| D  u_{t}  \|_{L^{2}}^{2},\\
\B|\int_{\mathbb{R}^{3}}B_{t}\cdot\nabla u\cdot B_{t}dx\B|\leq&C\|B _{t}\|^{\f{10}{3 }}_{L^{2}}+C\| D  u \|^{  10  }_{L^{2}}+\eta\| D  u \|_{L^{2}}^{4}
\| D ^{2}u \|^{2}_{L^{2}}+\varepsilon\| D  B_{t}  \|_{L^{2}}^{2}.
\ea\ee
Plugging \eqref{3.11}-\eqref{3.12} into \eqref{3.8}, we discover that
$$\ba
&\f{1}{2}\f{d}{dt}(\| u _{t}\|^{2}_{L^{2}}+\|B _{t}\|^{2}_{L^{2}})+\| D  u _{t}\|_{L^{2}}^{2}+\| D  B _{t}\|_{L^{2}}^{2}
\\\leq&C\|u _{t}\|^{\f{10}{3 }}_{L^{2}}+C\|B _{t}\|^{\f{10}{3 }}_{L^{2}}+C\| D  u \|^{  10  }_{L^{2}} +C\| D  B \|^{  10  }_{L^{2}}\\&+\eta\| D  u \|_{L^{2}}^{4}
\| D ^{2}u \|^{2}_{L^{2}}+\eta\| D  B \|_{L^{2}}^{4}
\| D ^{2}B\|^{2}_{L^{2}}+\varepsilon\| D  u_{t}  \|_{L^{2}}^{2}+\varepsilon\| D  B_{t}  \|_{L^{2}}^{2},
\ea
$$
Taking $\varepsilon$ sufficiently small, we infer that
\be\ba\label{3.10v1}
& \f{d}{dt}(\| u _{t}\|^{2}_{L^{2}}+\|B _{t}\|^{2}_{L^{2}})+\f12\| D  u _{t}\|_{L^{2}}^{2}+\f12\| D  B _{t}\|_{L^{2}}^{2}
\\\leq&C\|u _{t}\|^{\f{10}{3 }}_{L^{2}}+C\|B _{t}\|^{\f{10}{3 }}_{L^{2}}+C\| D  u \|^{  10  }_{L^{2}} +C\| D  B \|^{  10  }_{L^{2}}\\&+4\eta(\| D  u \|_{L^{2}}^{2}+\| D  B \|_{L^{2}}^{2})^{2}
(\| D ^{2}u \|^{2}_{L^{2}}+\| D ^{2}B \|^{2}_{L^{2}}),
\ea
\ee
where the Young inequality was used.

(4)
Multiplying the velocity field and the  magnetic field equations in \eqref{3.7} with $ D ^{2 }u_{t}$ and $ D ^{2 }u_{t}$, respectively, we integrate
over $\mathbb{R}^{3}$ and sum to get
\be\ba\label{3.18}
&\f{1}{2}\f{d}{dt}(\| D   u _{t}\|^{2}_{L^{2}}+\| D   B _{t}\|^{2}_{L^{2}})+\| D ^{2 }u _{t}\|_{L^{2}}^{2}+\| D ^{2 }B _{t}\|_{L^{2}}^{2}\\=&-\int_{\mathbb{R}^{3}} u_{t}\cdot\nabla u\cdot  D ^{2 }u_{t}dx-\int_{\mathbb{R}^{3}}u\cdot\nabla u_{t}\cdot D ^{2 }u_{t}dx\\&+\int_{\mathbb{R}^{3}} B_{t}\cdot\nabla B\cdot  D ^{2 }u_{t}dx+\int_{\mathbb{R}^{3}} B\cdot\nabla B_{t}\cdot  D ^{2 }u_{t}dx\\
=&-\int_{\mathbb{R}^{3}} u_{t}\cdot\nabla B\cdot  D ^{2 }B_{t}dx-\int_{\mathbb{R}^{3}}u\cdot\nabla B_{t}\cdot D ^{2 }B_{t}dx\\&+\int_{\mathbb{R}^{3}} B_{t}\cdot\nabla u\cdot  D ^{2 }B_{t}dx+\int_{\mathbb{R}^{3}} B\cdot\nabla u_{t}\cdot  D ^{2 }B_{t}dx.
\ea
\ee
In the light   of the H\"older inequality and interpolation inequality \eqref{keyinequality}, one has
\be\ba\label{3.17}
& \B|\int_{\mathbb{R}^{3}}u_{t}\cdot\nabla u+u\cdot\nabla u_{t},  D ^{2 }u_{t} dx\B|\\
\leq& \|u_{t}\|_{L^{6}} \|\nabla u \|_{L^{3}} \| D ^{2 }u_{t}\|_{L^{2}}+\|u\|_{L^{6}} \|\nabla u_{t} \|_{L^{3}}\| D ^{2 }u_{t}\|_{L^{2}}\\
\leq& C\| D  u_{t}\|_{L^{2}} \| D  u \|_{L^{2}}^{\f{1}{ 2}}
\| D ^{2}u \|^{\f{1}{ 2}}_{L^{2}} \| D ^{2 }u_{t}\|_{L^{2}}\\&+C\| D  u\|_{L^{2}} \| D  u_{t} \|_{L^{2}}^{\f{1}{ 2}}
\| D ^{2}u_{t} \|^{\f{1}{ 2}}_{L^{2}} \| D ^{2 }u_{t}\|_{L^{2}}.
\ea
\ee
By means  of the Young inequality, we remark  that
\be\ba\label{3.15}
&\| D  u_{t}\|_{L^{2}} \| D  u \|_{L^{2}}^{\f{1}{ 2}}
\| D ^{2}u \|^{\f{1}{ 2}}_{L^{2}} \| D ^{2 }u_{t}\|_{L^{2}}\\
\leq&  C\| D  u_{t}\|_{L^{2}}^{\f{14}{5}}
+C\| D  u \|_{L^{2}}^{14}
+\eta\| D  u \|^{8}_{L^{2}} \| D ^{2 }u \|_{L^{2}}^{2}+\varepsilon\| D ^{2 }u_{t}\|_{L^{2}}^{2}\ea
\ee
and
\be\ba
&\| D  u\|_{L^{2}} \| D  u_{t} \|_{L^{2}}^{\f{1}{ 2}}
\| D ^{2}u_{t} \|^{\f{1}{ 2}}_{L^{2}} \| D ^{2 }u_{t}\|_{L^{2}}\\
\leq&C\| D  u_{t} \|_{L^{2}}^{2}\| D  u\|_{L^{2}}^{ 4  } +\varepsilon\| D ^{2 }u_{t}\|_{L^{2}}^{2}
\\
\leq & C\| D  u_{t}\|_{L^{2}}^{\f{14}{5}}
+C\| D  u \|_{L^{2}}^{14}+\varepsilon\| D ^{2 }u_{t}\|_{L^{2}}^{2},
\ea
\ee
where $\varepsilon$ and $\eta$ are to be determined later.

A slight modification of the above derivation implies that
\be\ba\label{3.16}
&\B|\int_{\mathbb{R}^{3}} B_{t}\cdot\nabla B\cdot  D ^{2 }u_{t}dx\B|\leq   C\| D  B_{t}\|_{L^{2}}^{\f{14}{5}}
+C\| D  B \|_{L^{2}}^{14}
+\eta\| D  B \|^{8}_{L^{2}} \| D ^{2 }B \|_{L^{2}}^{2}+\varepsilon\| D ^{2 }u_{t}
\|_{L^{2}}^{2},\\
&\B|\int_{\mathbb{R}^{3}} u_{t}\cdot\nabla B\cdot  D ^{2 }B_{t}dx\B|
\leq   C\| D  u_{t}\|_{L^{2}}^{\f{14}{5}}
+C\| D  B \|_{L^{2}}^{14}
+\eta\| D  B \|^{8}_{L^{2}} \| D ^{2 }B \|_{L^{2}}^{2}+\varepsilon\| D ^{2 }B_{t}\|_{L^{2}}^{2},\\
&
\B|\int_{\mathbb{R}^{3}} B_{t}\cdot\nabla u\cdot  D ^{2 }B_{t}dx\B|\leq C\| D  B_{t}\|_{L^{2}}^{\f{14}{5}}
+C\| D  u \|_{L^{2}}^{14}
+\eta\| D  u \|^{8}_{L^{2}} \| D ^{2 }u \|_{L^{2}}^{2}+\varepsilon\| D ^{2 }
B_{t}\|_{L^{2}}^{2},\\
&
\B|\int_{\mathbb{R}^{3}} B\cdot\nabla B_{t}\cdot  D ^{2 }u_{t}dx\B|
\leq  C\| D  B_{t}\|_{L^{2}}^{\f{14}{5}}
+C\| D  B \|_{L^{2}}^{14}+\varepsilon(\| D ^{2 }u_{t}\|_{L^{2}}^{2}+\| D ^{2 }B_{t}\|_{L^{2}}^{2}),\\
&\B|-\int_{\mathbb{R}^{3}}u\cdot\nabla B_{t}\cdot  D ^{2 }B_{t}dx\B|
\leq   C\| D  B_{t}\|_{L^{2}}^{\f{14}{5}}
+C\| D  u \|_{L^{2}}^{14}+\varepsilon \| D ^{2 }B_{t}\|_{L^{2}}^{2}, \\
&\B|\int_{\mathbb{R}^{3}} B\cdot\nabla u_{t}\cdot  D ^{2 }B_{t}dx\B|\leq  C\| D  u_{t}\|_{L^{2}}^{\f{14}{5}}
+C\| D  B \|_{L^{2}}^{14}+\varepsilon(\| D ^{2 }u_{t}\|_{L^{2}}^{2}+\| D ^{2 }B_{t}\|_{L^{2}}^{2}).
\ea\ee
Inserting \eqref{3.15}-\eqref{3.16} into \eqref{3.18}  and taking $\varepsilon$ sufficiently small, we remark that
\be\ba\label{3.20}
&\f{1}{2}\f{d}{dt}(\| D   u _{t}\|^{2}_{L^{2}}+\| D   B _{t}\|^{2}_{L^{2}})+\f12\| D ^{2 }u _{t}\|_{L^{2}}^{2}+\f12\| D ^{2 }B _{t}\|_{L^{2}}^{2}
\\
\leq& C(\| D   u _{t}\|^{2}_{L^{2}}+\| D   B _{t}\|^{2}_{L^{2}})^{\f{7}{5}}
+C(\| D  u (t)\|^{2}_{L^{2}}+\| D  B (t)\|^{2}_{L^{2}})^{7}
\\&+4\eta(\| D ^{2}u (t)\|_{L^{2}}^{2}+ \| D ^{2}B (t)\|_{L^{2}}^{2})(\| D  B \|_{L^{2}}^{2}+\| D  u\|_{L^{2}}^{2})^{4}.
\ea
\ee
We
multiply  the velocity field and the  magnetic field equations in  \eqref{3.7} of
by $u_{tt}$ and $B_{tt}$, respectively, integrate over the spatial variable and sum to obtain
\be\ba\label{3.28}
&\f12\f{d}{dt}(\| D  u _{t}\|^{2}_{L^{2}}+\| D  B _{t}\|^{2}_{L^{2}})
+ \| u _{tt}\|_{L^{2}}^{2}+ \| B _{tt}\|_{L^{2}}^{2}
\\=&-\int_{\mathbb{R}^{3}}u_{t}\cdot\nabla u\cdot u_{tt}dx-\int_{\mathbb{R}^{3}}u\cdot\nabla u_{t}\cdot u_{tt}dx+\int_{\mathbb{R}^{3}}B_{t}\cdot\nabla B\cdot u_{tt}dx+\int_{\mathbb{R}^{3}}B\cdot\nabla B_{t} \cdot u_{tt}dx.\\
& -\int_{\mathbb{R}^{3}}u_{t}\cdot\nabla B\cdot B_{tt}dx-\int_{\mathbb{R}^{3}} u\cdot\nabla B_{t}\cdot B_{tt}dx+\int_{\mathbb{R}^{3}}B_{t}\cdot\nabla u\cdot B_{tt}dx+\int_{\mathbb{R}^{3}}B\cdot\nabla u_{t} \cdot B_{tt}dx.
\ea\ee
Along the line of  \eqref{3.17}, one has
$$\ba
&\B|\int_{\mathbb{R}^{3}}(u_{t}\cdot\nabla u+u\cdot\nabla u_{t}) \cdot u_{tt}dx\B|
\\\leq& \|u_{t}\|_{L^{6}} \|\nabla u \|_{L^{3}} \|u_{tt}\|_{L^{2}}+\|u\|_{L^{6}} \|\nabla u_{t} \|_{L^{3}}\|u_{tt}\|_{L^{2}}\\
\leq& C\| D  u_{t}\|_{L^{2}} \| D  u \|_{L^{2}}^{\f{1}{ 2}}
\| D ^{2}u \|^{\f{1}{ 2}}_{L^{2}} \|u_{tt}\|_{L^{2}}+C\| D  u\|_{L^{2}} \| D  u_{t} \|_{L^{2}}^{\f{1}{ 2}}
\| D ^{2}u_{t} \|^{\f{1}{ 2}}_{L^{2}} \|u_{tt}\|_{L^{2}}.\ea
$$
 Replacing $\| D ^{2 }u_{t}\|_{L^{2}}$ by $u_{tt}$ in \eqref{3.15} and \eqref{3.16}, we find
$$\ba
&\| D  u_{t}\|_{L^{2}} \| D  u \|_{L^{2}}^{\f{1}{ 2}}
\| D ^{2}u \|^{\f{1}{ 2}}_{L^{2}} \|  u_{tt}\|_{L^{2}}\\
\leq&  C\| D  u_{t}\|_{L^{2}}^{\f{14}{5}}
+C\| D  u \|_{L^{2}}^{14}
+\eta\| D  u \|^{8}_{L^{2}} \| D ^{2 }u \|_{L^{2}}^{2}+\varepsilon\| u_{tt}\|_{L^{2}}^{2},\ea
$$
and
$$\ba
&\| D  u\|_{L^{2}} \| D  u_{t} \|_{L^{2}}^{\f{1}{ 2}}
\| D ^{2}u_{t} \|^{\f{1}{ 2}}_{L^{2}} \|  u_{tt}\|_{L^{2}}\\
 \leq & C\| D  u_{t}\|_{L^{2}}^{\f{14}{5}}
+C\| D  u \|_{L^{2}}^{14}+\varepsilon\| u_{tt}\|_{L^{2}}^{2}+\varepsilon\| D ^{2}u_{t}\|_{L^{2}}^{2},
\ea
$$
which turns out that
$$\ba
&\B|\int_{\mathbb{R}^{3}}u_{t}\cdot\nabla u+u\cdot\nabla u_{t}, u_{tt}dx\B|
\\\leq&  C\| D  u_{t}\|_{L^{2}}^{\f{14}{5}}
+C\| D  u \|_{L^{2}}^{14}
+\eta\| D  u \|^{8}_{L^{2}} \| D ^{2 }u \|_{L^{2}}^{2}+2\varepsilon\| u_{tt}\|_{L^{2}}^{2}.
\ea
$$
Here, $\eta$ and $\varepsilon $  will be chosen later.\\
Likewise,
$$\ba
&\B|\int_{\mathbb{R}^{3}}B_{t}\cdot\nabla B\cdot u_{tt}dx+\int_{\mathbb{R}^{3}}B\cdot\nabla B_{t} \cdot u_{tt}dx\B|\\
\leq&  C\| D  B_{t}\|_{L^{2}}^{\f{14}{5}}
+C\| D  B \|_{L^{2}}^{14}
+\eta\| D  B \|^{8}_{L^{2}} \| D ^{2 }B \|_{L^{2}}^{2}+\varepsilon\| u_{tt}\|_{L^{2}}^{2}\\&+
C\| D  B_{t}\|_{L^{2}}^{\f{14}{5}}
+C\| D  B \|_{L^{2}}^{14}+\varepsilon\| u_{tt}\|_{L^{2}}^{2}+\varepsilon\| D ^{2}B_{t}\|_{L^{2}}^{2},\ea
$$
$$\ba
&\B|-\int_{\mathbb{R}^{3}}u_{t}\cdot\nabla B\cdot B_{tt}dx-\int_{\mathbb{R}^{3}}u\cdot\nabla B_{t}\cdot B_{tt}dx\B|\\
\leq&  C\| D  u_{t}\|_{L^{2}}^{\f{14}{5}}
+C\| D  B \|_{L^{2}}^{14}
+\eta\| D  B \|^{8}_{L^{2}} \| D ^{2 }B \|_{L^{2}}^{2}+\varepsilon\| B_{tt}\|_{L^{2}}^{2}\\ &+
C\| D  B_{t}\|_{L^{2}}^{\f{14}{5}}
+C\| D  u \|_{L^{2}}^{14}+\varepsilon\| B_{tt}\|_{L^{2}}^{2}+\varepsilon\| D ^{2}B_{t}\|_{L^{2}}^{2},
\ea$$
and
$$\ba
&\B|\int_{\mathbb{R}^{3}}B_{t}\cdot\nabla u\cdot B_{tt}dx+\int_{\mathbb{R}^{3}}B\cdot\nabla u_{t} \cdot B_{tt}dx\B|\\
\leq&  C\| D  u_{t}\|_{L^{2}}^{\f{14}{5}}
+C\| D  u \|_{L^{2}}^{14}
+\eta\| D  u \|^{8}_{L^{2}} \| D ^{2 }u \|_{L^{2}}^{2}+\varepsilon\| B_{tt}\|_{L^{2}}^{2}\\ &+
C\| D  u_{t}\|_{L^{2}}^{\f{14}{5}}
+C\| D  B \|_{L^{2}}^{14}+\varepsilon\| B_{tt}\|_{L^{2}}^{2}+\varepsilon\| D ^{2}u_{t}\|_{L^{2}}^{2}.
\ea$$
Plugging  these inequalities into \eqref{3.28}, by choosing $\varepsilon$ sufficiently small, we notice  that
$$\ba\label{3.21v1}
&\f12\f{d}{dt}(\| D  u _{t}\|^{2}_{L^{2}}+\| D  B _{t}\|^{2}_{L^{2}})
+ \f12(\| u _{tt}\|_{L^{2}}^{2}+\| B _{tt}\|_{L^{2}}^{2})\\
\leq&  C(\| D   u _{t}\|^{2}_{L^{2}}+\| D   B _{t}\|^{2}_{L^{2}})^{\f{7}{5}}
+C(\| D  u (t)\|^{2}_{L^{2}}+\| D  B (t)\|^{2}_{L^{2}})^{7}
\\&+
6\eta(\| D ^{2}u (t)\|_{L^{2}}^{2}+ \| D ^{2}B (t)\|_{L^{2}}^{2})(\| D  B \|_{L^{2}}^{2}+\| D  u\|_{L^{2}}^{2})^{4},
\ea
$$
which together with
\eqref{3.20} leads to
\be\ba
& \f{d}{dt}(\| D   u _{t}\|^{2}_{L^{2}}+\| D   B _{t}\|^{2}_{L^{2}})+\f12\| D ^{2 }u _{t}\|_{L^{2}}^{2}+\f12\| D ^{2 }B _{t}\|_{L^{2}}^{2}+ \f12(\| u _{tt}\|_{L^{2}}^{2}+\| B _{tt}\|_{L^{2}}^{2})
\\
\leq& C(\| D   u _{t}\|^{2}_{L^{2}}+\| D   B _{t}\|^{2}_{L^{2}})^{\f{7}{5}}
+C(\| D  u (t)\|^{2}_{L^{2}}+\| D  B (t)\|^{2}_{L^{2}})^{7}
\\&+10\eta(\| D ^{2}u (t)\|_{L^{2}}^{2}+ \| D ^{2}B (t)\|_{L^{2}}^{2})(\| D  B \|_{L^{2}}^{2}+\| D  u\|_{L^{2}}^{2})^{4}.
\ea
\ee
This completes the proof of this lemma.
\end{proof}

We get $H^{s}$ estimate of solutions by zhijiedi, We could get $H^{s} $ and $H^{s+1}\in L^{2}L^{2}$  However,

$u _{t}^{(r)}$ $\nabla  B _{t}^{(r)}$ rather than the desiral  $u _{t}^{(r+1)}$.

$u _{t}^{(r+1)}$ require the addition energy inequality $\nabla  u _{t}^{(r)}$

Unlike the i space time
\begin{lemma}\label{lema3.2}Let the pair  $(u, B)$ satisfy   the  MHD  equations
\eqref{MHD}. Then for any positive $t$ and $\eta$, there exists a positive constant $C$ such that
\be\ba\label{3.31v1}
&\f{d}{dt}(\| u _{t}^{(r)}\|^{2}_{L^{2}}+ \| B _{t}^{(r)}\|^{2}_{L^{2}})+\f{1}{2}\| \nabla  u _{t}^{(r)}\|_{L^{2}}^{2}+\f{1}{2}\|  \nabla  B _{t}^{(r)}\|_{L^{2}}^{2} \\
\leq&C\sum_{j>0}(\|u_{t}^{(j)}\|_{L^{2}}^{2}+\|B_{t}^{(j)}\|_{L^{2}}^{2})
^{\f{ (4r +1)}{4j -1}}
+C\sum_{j>0}(\| D  u_{t}^{(r-j)} \|_{L^{2}}^{2}+\| D  u_{t}^{(r-j)} \|_{L^{2}}^{2})^{\f{(4r +1)}{4(r-j)+1}}
\\&+8\eta\sum_{j>0}(\| D  u_{t}^{(r-j)} \|_{L^{2}}+\| D  B_{t}^{(r-j)} \|_{L^{2}})^{\f{2 (2j-1)}{  4(r-j)+1}}
(\| D ^{2}u_{t}^{(r-j)} \|^{2}_{L^{2}}+\| D ^{2}B_{t}^{(r-j)} \|^{2}_{L^{2}}),
\ea
\ee
and
\be\ba\label{3.38v1}
&\f{d}{dt}(\| \nabla  u _{t}^{(r)}\|^{2}_{L^{2}}+\|   \nabla   B _{t}^{(r)}\|^{2}_{L^{2}})+\f{1}{2}\| D \nabla u _{t}^{(r)}\|_{L^{2}}^{2}+\f{1}{2}\| D  \nabla B _{t}^{(r)}\|_{L^{2}}^{2} \\&+\f{1}{2}\| u _{t}^{(r+1)}\|_{L^{2}}^{2}+\f{1}{2}\|B _{t}^{(r+1)}\|_{L^{2}}^{2} \\\leq&  C(\| D  u_{t}^{(r )}  \|_{L^{2}}^{2}+\| D  B_{t}^{(r )}  \|_{L^{2}}^{2})^{\f{ (4r +3 )}{ 4r  +1}} +C\sum^{r-1}_{j=0}(\| D   u_{t}^{(j)}\|_{L^{2}}^{2}+\|  D   B_{t}^{(j)}\|_{L^{2}}^{2})^{ \f{4 r +3 }{4j +1}}
\\&+8\eta\sum_{j>0}(\| D  u_{t}^{(r-j)} \|^{2}_{L^{2}}+\| D  B_{t}^{(r-j)} \|^{2}_{L^{2}})^{\f{4  j}{  4(r-j)+1}}(
\| D ^{2}u_{t}^{(r-j)} \|^{2}_{L^{2}}+
\| D ^{2}B_{t}^{(r-j)} \|^{2}_{L^{2}}).
\ea
\ee
\end{lemma}
\begin{proof}
(1) We derive from the    MHD equations
\eqref{MHD} and Leibniz's formular that
\be\ba\label{3.26}
 &u^{(r+1) }_{t} -\Delta u^{(r)}_{t}+\sum_{j=0}^{r}\binom{r}{j}u_{t}^{(j)}\cdot\nabla u^{(r-j)}_{t}-\sum_{j=0}^{r}\binom{r}{j}B_{t}^{(j)}\cdot\nabla B^{(r-j)}_{t} +\nabla p^{(r)}_{t}=0,\\
&B^{(r+1) }_{t} -\Delta B^{(r)}_{t}+\sum_{j=0}^{r}\binom{r}{j}u_{t}^{(j)}\cdot\nabla B^{(r-j)}_{t}-\sum_{j=0}^{r}\binom{r}{j}B_{t}^{(j)}\cdot\nabla u^{(r-j)}_{t}=0.
\ea
\ee
We take the $L^{2}$  inner product of \eqref{3.26} with $u _{t}^{(r)}$ and $B _{t}^{(r)}$, respectively, and sum to get
\be\ba\label{3.27}
&\f{1}{2}\f{d}{dt}(\| u _{t}^{(r)}\|^{2}_{L^{2}}+\| B _{t}^{(r)}\|^{2}_{L^{2}})+\|  \nabla  u _{t}^{(r)}\|_{L^{2}}^{2}+\|  \nabla  B _{t}^{(r)}\|_{L^{2}}^{2}\\
=&-\sum_{j=0}^{r}\binom{r}{j}\int_{\mathbb{R}^{3}}u_{t}^{(j)}\cdot\nabla u^{(r-j)}_{t}  u _{t}^{(r)}dx+\sum_{j=0}^{r}\binom{r}{j}\int_{\mathbb{R}^{3}} B_{t}^{(j)}\cdot\nabla B^{(r-j)}_{t}  u _{t}^{(r)}dx\\
&-\sum_{j=0}^{r}\binom{r}{j}\int_{\mathbb{R}^{3}}u_{t}^{(j)}\cdot\nabla B^{(r-j)}_{t}  B _{t}^{(r)}dx+\sum_{j=0}^{r}\binom{r}{j}\int_{\mathbb{R}^{3}} B_{t}^{(j)}\cdot\nabla u^{(r-j)}_{t}  B_{t}^{(r)}dx.
\ea
\ee
According to the divergence free condition and integration by parts, one has
$$\ba&\int_{\mathbb{R}^{3}}u \cdot\nabla u_{t}^{(r )} \cdot u _{t}^{(r)}dx=0,\\
&\int_{\mathbb{R}^{3}}B \cdot\nabla B_{t}^{(r )} \cdot u _{t}^{(r)}dx-\int_{\mathbb{R}^{3}}B \cdot\nabla u_{t}^{(r )}\cdot B _{t}^{(r)}dx=0,\\
&\int_{\mathbb{R}^{3}}u \cdot\nabla B_{t}^{(r )} \cdot B _{t}^{(r)}dx=0.\ea$$
As a consquence, we reformulate \eqref{3.27} as
\be\ba\label{3.33}
&\f{1}{2}\f{d}{dt}(\| u _{t}^{(r)}\|^{2}_{L^{2}}+\| B _{t}^{(r)}\|^{2}_{L^{2}})+\|  \nabla  u _{t}^{(r)}\|_{L^{2}}^{2}+\|  \nabla  B _{t}^{(r)}\|_{L^{2}}^{2}\\=&-\sum_{j=1}^{r}\binom{r}{j}\int_{\mathbb{R}^{3}} u_{t}^{(j)}\cdot\nabla u_{t}^{(r-j)}  u _{t}^{(r)}dx+\sum_{j=1}^{r}\binom{r}{j}\int_{\mathbb{R}^{3}} B_{t}^{(j)}\cdot\nabla B_{t}^{(r-j)}  u _{t}^{(r)}dx\\
&-\sum_{j=1}^{r}\binom{r}{j}\int_{\mathbb{R}^{3}}u_{t}^{(j)}\cdot\nabla B_{t}^{(r-j)}  B _{t}^{(r)}dx+\sum_{j=1}^{r}\binom{r}{j}\int_{\mathbb{R}^{3}} B_{t}^{(j)}\cdot\nabla u_{t}^{(r-j)}  B _{t}^{(r)}dx.
\ea
\ee
As the same manner of \eqref{3.17}, we derive from  the Young inequality that
\be\ba\label{3.36}
\B|\int_{\mathbb{R}^{3}}u_{t}^{(j)}\cdot\nabla u_{t}^{(r-j)}\cdot  u _{t}^{(r)}dx\B|
\leq& \|u_{t}^{(j)}\|_{L^{2}} \|\nabla u^{(r-j)}_{t}\|_{L^{3}}\| u _{t}^{(r)}\|_{L^{6}}\\
\leq&\|u_{t}^{(j)}\|_{L^{2}}\| D  u_{t}^{(r-j)} \|_{L^{2}}^{\f{1}{ 2}}
\| D ^{2}u_{t}^{(r-j)} \|^{\f{1}{ 2}}_{L^{2}}    \| D   u _{t}^{(r)}\|_{L^{2}} \\
\leq&C\|u_{t}^{(j)}\|^{2}_{L^{2}}\| D  u_{t}^{(r-j)} \|_{L^{2}}^{\f{1}{  \alpha}}
\| D ^{2}u_{t}^{(r-j)} \| _{L^{2}}  +\varepsilon  \| D   u _{t}^{(r)}\|_{L^{2}}^{2}\\
\leq&C \|u_{t}^{(j)}\|
^{\f{2(4r +1)}{4j -1}}
_{L^{2}}
+C \| D  u_{t}^{(r-j)} \|_{L^{2}}^{\f{2(4r +1)}{4(r-j)+1}}
\\&+\eta \| D  u_{t}^{(r-j)} \|_{L^{2}}^{\f{4 (2j-1)}{  4(r-j)+1}}
\| D ^{2}u_{t}^{(r-j)} \|^{2}_{L^{2}} +\varepsilon  \| D   u _{t}^{(r)}\|_{L^{2}}^{2},\ea
\ee
where $\eta$ and $\varepsilon$ will be fixed later. \\
Repeating the above deduction, we infer that
\be\ba
\B|\int_{\mathbb{R}^{3}}B_{t}^{(j)}\cdot\nabla B_{t}^{(r-j)}\cdot   u _{t}^{(r)}dx\B|
\leq&C \|B_{t}^{(j)}\|
^{\f{2(4r +1)}{4j -1}}
_{L^{2}}
+C \| D  B_{t}^{(r-j)} \|_{L^{2}}^{\f{2(4r +1)}{4(r-j)+1}}
\\&+\eta \| D  B_{t}^{(r-j)} \|_{L^{2}}^{\f{4 (2j-1)}{  4(r-j)+1}}
\| D ^{2}B_{t}^{(r-j)} \|^{2}_{L^{2}} +\varepsilon  \| D   u _{t}^{(r)}\|_{L^{2}}^{2},\ea
\ee
\be\ba
\B|\int_{\mathbb{R}^{3}}u_{t}^{(j)}\cdot\nabla B_{t}^{(r-j)} \cdot  B _{t}^{(r)}dx\B|
\leq&C \|u_{t}^{(j)}\|
^{\f{2(4r +1)}{4j -1}}
_{L^{2}}
+C \| D  B_{t}^{(r-j)} \|_{L^{2}}^{\f{2(4r +1)}{4(r-j)+1}}
\\&+\eta \| D  B_{t}^{(r-j)} \|_{L^{2}}^{\f{4 (2j-1)}{  4(r-j)+1}}
\| D ^{2}B_{t}^{(r-j)} \|^{2}_{L^{2}} +\varepsilon  \| D   B_{t}^{(r)}\|_{L^{2}}^{2},\ea
\ee
and
 \be\ba\label{3.37}
\B|\int_{\mathbb{R}^{3}}B_{t}^{(j)}\cdot\nabla u_{t}^{(r-j)} \cdot  B _{t}^{(r)}dx\B|
\leq&C \|B_{t}^{(j)}\|
^{\f{2(4r +1)}{4j -1}}
_{L^{2}}
+C \| D  u_{t}^{(r-j)} \|_{L^{2}}^{\f{2(4r +1)}{4(r-j)+1}}
\\&+\eta \| D  u_{t}^{(r-j)} \|_{L^{2}}^{\f{4 (2j-1)}{  4(r-j)+1}}
\| D ^{2}u_{t}^{(r-j)} \|^{2}_{L^{2}} +\varepsilon  \| D   B_{t}^{(r)}\|_{L^{2}}^{2}.\ea
\ee
Substituting \eqref{3.36}-\eqref{3.37} into \eqref{3.33}, we conclude by
  choosing $\varepsilon$  sufficiently small  that
 \be\ba\label{3.31}
&\f{1}{2}\f{d}{dt}(\| u _{t}^{(r)}\|^{2}_{L^{2}}+\| B _{t}^{(r)}\|^{2}_{L^{2}})+\f34\|  \nabla  u _{t}^{(r)}\|_{L^{2}}^{2}+\f34\|  \nabla  B _{t}^{(r)}\|_{L^{2}}^{2} \\
\leq&C\sum_{j>0}(\|u_{t}^{(j)}\|_{L^{2}}^{2}+\|B_{t}^{(j)}\|_{L^{2}}^{2})
^{\f{ (4r +1)}{4j -1}}
+C\sum_{j>0}(\| D  u_{t}^{(r-j)} \|_{L^{2}}^{2}+\| D  u_{t}^{(r-j)} \|_{L^{2}}^{2})^{\f{(4r +1)}{4(r-j)+1}}
\\&+4\eta\sum_{j>0}(\| D  u_{t}^{(r-j)} \|_{L^{2}}+\| D  B_{t}^{(r-j)} \|_{L^{2}})^{\f{2 (2j-1)}{  4(r-j)+1}}
(\| D ^{2}u_{t}^{(r-j)} \|^{2}_{L^{2}}+\| D ^{2}B_{t}^{(r-j)} \|^{2}_{L^{2}}).
\ea
\ee
(2) Taking scalar product of $\eqref{3.26}_{1}$ by
$ D ^{2}u _{t}^{(r)}$ and $\eqref{3.26}_{2}$ by
$ D ^{2}B _{t}^{(r)}$, we integrate over $\mathbb{R}^{3}$ and sum to deduce that
\be\ba\label{3.39v2}
&\f{1}{2}\f{d}{dt}(\|  D  u _{t}^{(r)}\|^{2}_{L^{2}}+\|  D  B _{t}^{(r)}\|^{2}_{L^{2}})+\| D \nabla u _{t}^{(r)}\|_{L^{2}}^{2}+\| D \nabla B _{t}^{(r)}\|_{L^{2}}^{2}\\=&-\sum_{j=0}^{r}\binom{r}{j}\int_{\mathbb{R}^{3}} u_{t}^{(j)}\cdot\nabla u^{(r-j)}_{t}\cdot  D ^{2}u _{t}^{(r)}dx+\sum_{j=0}^{r}\binom{r}{j}\int_{\mathbb{R}^{3}} B_{t}^{(j)}\cdot\nabla B^{(r-j)}_{t}\cdot  D ^{2}u _{t}^{(r)}dx\\&-\sum_{j=0}^{r}\binom{r}{j}\int_{\mathbb{R}^{3}} u_{t}^{(j)}\cdot\nabla B^{(r-j)}_{t}\cdot  D ^{2}B _{t}^{(r)}dx+\sum_{j=0}^{r}\binom{r}{j}\int_{\mathbb{R}^{3}} B_{t}^{(j)}\cdot\nabla u^{(r-j)}_{t}\cdot  D ^{2}B _{t}^{(r)}dx.
\ea
\ee
A slight modification of derivation of \eqref{3.17}, one gets
\be\ba\label{3.40v2}
&\B|\sum_{j=0}\int_{\mathbb{R}^{3}}u_{t}^{(j)}\cdot\nabla u^{(r-j)}_{t}\cdot  D ^{2}u _{t}^{(r)}dx\B|\\
\leq& \sum_{j=0}\| u_{t}^{(j)}\|_{L^{6}}\|\nabla u^{(r-j)}_{t}\|_{L^{3}}\| D ^{2}u _{t}^{(r)}\|_{L^{2}}\\
\leq& \sum_{j=0}\|  D   u_{t}^{(j)}\|_{L^{2}}\| D  u_{t}^{(r-j)} \|_{L^{2}}^{\f{1}{ 2}}
\| D ^{2}u_{t}^{(r-j)} \|^{\f{1}{ 2}}_{L^{2}}\| D ^{2}u _{t}^{(r)}\|_{L^{2}}\\
\leq& \sum_{j>0}\|  D   u_{t}^{(j)}\|_{L^{2}}^{2}\| D  u_{t}^{(r-j)} \|_{L^{2}}
\| D ^{2 }u_{t}^{(r-j)} \| _{L^{2}}+\f{\varepsilon}{2}\| D ^{2}u _{t}^{(r)}\|_{L^{2}}^{2}\\&+\|  D   u \|_{L^{2}}\| D  u_{t}^{(r )} \|_{L^{2}}^{\f{1}{ 2}}
\| D ^{2}u_{t}^{(r )} \|^{\f{1}{ 2}}_{L^{2}}\| D ^{2}u _{t}^{(r)}\|_{L^{2}}\\
\leq& C\sum_{j>0}\|  D   u_{t}^{(j)}\|_{L^{2}}^{\f{2 (4 r +3 )}{4j +1}}
+C\sum_{j>0}\| D  u_{t}^{(r-j)} \|_{L^{2}}^{\f{2(4 r +3 )}{4(r-j)+1}}
\\&+\eta\sum_{j>0}\| D  u_{t}^{(r-j)} \|_{L^{2}}^{\f{8  j}{  4(r-j)+1}}
\| D ^{2}u_{t}^{(r-j)} \|^{2}_{L^{2}}\\&+C\| D  u  \|_{L^{2}}^{ 2(4r +3 ) }+C\| D  u_{t}^{(r )}  \|_{L^{2}}^{\f{2(4r +3 )}{ 4r  +1}}+\varepsilon\| D ^{2}u _{t}^{(r)}\|_{L^{2}}^{2},
 \ea
\ee
where $\varepsilon$ and $\eta$ will be specified below.\\
 In a similar way we can prove that
\be\ba
&\B|\sum_{j=0}^{r}\binom{r}{j}\int_{\mathbb{R}^{3}}B_{t}^{(j)}\cdot\nabla B^{(r-j)}_{t}\cdot D ^{2}u _{t}^{(r)}dx\B|\\
 \leq&  C\| D  B_{t}^{(r )}  \|_{L^{2}}^{\f{2(4r +3 )}{ 4r  +1}}+C\sum_{j=0}^{r-1}\|  D   B_{t}^{(j)}\|_{L^{2}}^{2\f{4 r +3 }{4j +1}}
\\&+C\sum_{j>0}\| D  B_{t}^{(r-j)} \|_{L^{2}}^{\f{2( 4 r +3 )}{4(r-j)+1}}
+\eta\sum_{j>0}\| D  B_{t}^{(r-j)} \|_{L^{2}}^{\f{8  j}{  4(r-j)+1}}
\| D ^{2}B_{t}^{(r-j)} \|^{2}_{L^{2}}+\varepsilon\| D ^{2}u _{t}^{(r)}\|_{L^{2}}^{2},
\ea
\ee
\be\ba
&\B|-\sum_{j=0}^{r}\binom{r}{j}\int_{\mathbb{R}^{3}}u_{t}^{(j)}\cdot\nabla B^{(r-j)}_{t}\cdot D ^{2}B _{t}^{(r)}dx\B|\\
 \leq&  C\| D  u_{t}^{(r )}  \|_{L^{2}}^{\f{2(4r +3 )}{ 4r  +1}}+C\sum_{j=0}^{r-1}\|  D   u_{t}^{(j)}\|_{L^{2}}^{2\f{4 r +3 }{4j +1}}
\\&+C\sum_{j>0}\| D  B_{t}^{(r-j)} \|_{L^{2}}^{\f{2( 4 r +3 )}{4(r-j)+1}}
+\eta\sum_{j>0}\| D  B_{t}^{(r-j)} \|_{L^{2}}^{\f{8  j}{  4(r-j)+1}}
\| D ^{2}B_{t}^{(r-j)} \|^{2}_{L^{2}}+\varepsilon\| D ^{2}B _{t}^{(r)}\|_{L^{2}}^{2},
\ea
\ee
and
\be\ba\label{3.43}
&\B|\sum_{j=0}^{r}\binom{r}{j}\int_{\mathbb{R}^{3}}B_{t}^{(j)}\cdot\nabla u^{(r-j)}_{t}\cdot D ^{2}B _{t}^{(r)}dx\B|
\\
 \leq& C\| D  B\|_{L^{2}}^{ 2(4r +3 ) }+C\| D  B_{t}^{(r )}  \|_{L^{2}}^{\f{2(4r +3 )}{ 4r  +1}}+C\sum_{j>0}\|  D   B_{t}^{(j)}\|_{L^{2}}^{2\f{4 r +3 }{4j +1}}
\\&+C\sum_{j>0}\| D  u_{t}^{(r-j)} \|_{L^{2}}^{\f{2( 4 r +3 )}{4(r-j)+1}}
+\eta\sum_{j>0}\| D  u_{t}^{(r-j)} \|_{L^{2}}^{\f{8  j}{  4(r-j)+1}}
\| D ^{2}u_{t}^{(r-j)} \|^{2}_{L^{2}}+\varepsilon\| D ^{2}B _{t}^{(r)}\|_{L^{2}}^{2}.
\ea
\ee
Plugging    \eqref{3.40v2}-\eqref{3.43}   into  \eqref{3.39v2} and using
  the arbitrariness of $\varepsilon$,
one has
\be\ba\label{3.35}
&\f{1}{2}\f{d}{dt}(\|  D  u _{t}^{(r)}\|^{2}_{L^{2}}+\|  D  B _{t}^{(r)}\|^{2}_{L^{2}})+\f{3}{4}\| D  \nabla u _{t}^{(r)}\|_{L^{2}}^{2}+\f{3}{4}\| D  \nabla B _{t}^{(r)}\|_{L^{2}}^{2}\\\leq&  C(\| D  u_{t}^{(r )}  \|_{L^{2}}^{2}+\| D  u_{t}^{(r )}  \|_{L^{2}}^{2})^{\f{ (4r +3 )}{ 4r  +1}}\\&+C\sum^{r-1}_{j=0}(\|  D   u_{t}^{(j)}\|^{2}_{L^{2}}+\| D   u_{t}^{(j)}\|_{L^{2}}^{2})^{ \f{4 r +3 }{4j +1}}
 +C\sum_{j>0}(\| D  u_{t}^{(r-j)} \|_{L^{2}}^{2}+\| D  B_{t}^{(r-j)} \|_{L^{2}}^{2})^{\f{ ( 4 r +3 )}{4(r-j)+1}}
\\&+\eta\sum_{j>0}(\| D  u_{t}^{(r-j)} \|_{L^{2}}^{2}+\| D  B_{t}^{(r-j)} \|_{L^{2}}^{2})^{\f{4 j}{  4(r-j)+1}}
(\| D ^{2}u_{t}^{(r-j)} \|^{2}_{L^{2}}+\| D ^{2}B_{t}^{(r-j)} \|^{2}_{L^{2}})\\\leq&  C(\| D  u_{t}^{(r )}  \|_{L^{2}}^{2}+\| D  u_{t}^{(r )}  \|_{L^{2}}^{2})^{\f{ (4r +3 )}{ 4r  +1}} +C\sum^{r-1}_{j=0}(\|  D   u_{t}^{(j)}\|^{2}_{L^{2}}+\|  D   u_{t}^{(j)}\|_{L^{2}}^{2})^{ \f{4 r +3 }{4j +1}}
  \\&+4\eta\sum_{j>0}(\| D  u_{t}^{(r-j)} \|_{L^{2}}^{2}+\| D  B_{t}^{(r-j)} \|_{L^{2}}^{2})^{\f{4 j}{  4(r-j)+1}}
(\| D ^{2}u_{t}^{(r-j)} \|^{2}_{L^{2}}+\| D ^{2}B_{t}^{(r-j)} \|^{2}_{L^{2}}).
\ea
\ee
We multiply the velocity equations $\eqref{3.26}_{1}$ by  $u _{t}^{(r+1)}$
and the velocity equations $\eqref{3.26}_{2}$ by  $B _{t}^{(r+1)}$, integrate over the spatial variable
and sum to obtain
\be\ba\label{3.44}
& \f{1}{2}\f{d}{dt}(\|   \nabla  u _{t}^{(r)}\|^{2}_{L^{2}}+\|  \nabla   B _{t}^{(r)}\|^{2}_{L^{2}})+\| u _{t}^{(r+1)}\|_{L^{2}}^{2}+\|B _{t}^{(r+1)}\|_{L^{2}}^{2}\\
=&-\sum_{j=0}^{r}\binom{r}{j}\int_{\mathbb{R}^{3}}u_{t}^{(j)}\cdot\nabla u^{(r-j)}_{t}\cdot u _{t}^{(r+1)}dx+\sum_{j=0}^{r}\binom{r}{j}\int_{\mathbb{R}^{3}} B_{t}^{(j)}\cdot\nabla B^{(r-j)}_{t}\cdot u _{t}^{(r+1)}dx\\&-\sum_{j=0}^{r}\binom{r}{j}\int_{\mathbb{R}^{3}} u_{t}^{(j)}\cdot\nabla B^{(r-j)}_{t}\cdot B _{t}^{(r+1)}dx+\sum_{j=0}^{r}\binom{r}{j}\int_{\mathbb{R}^{3}} B_{t}^{(j)}\cdot\nabla u^{(r-j)}_{t}\cdot B _{t}^{(r+1)}dx\\
\leq& C\sum_{j=0}\| u_{t}^{(j)}\|_{L^{6}}\|\nabla u^{(r-j)}_{t}\|_{L^{3}}\| u _{t}^{(r+1)}\|_{L^{2}}\\
\leq&C\sum_{j=0} \|  D   u_{t}^{(j)}\|_{L^{2}}\| D  u_{t}^{(r-j)} \|_{L^{2}}^{\f{1}{ 2}}
\| D ^{2}u_{t}^{(r-j)} \|^{\f{1}{ 2}}_{L^{2}} \| u _{t}^{(r+1)}\|_{L^{2}}.
\ea
\ee
As derivation of \eqref{3.17}, we infer that
\be\ba\label{3.45v1}&\B|-\sum_{j=0}^{r}\binom{r}{j}\int_{\mathbb{R}^{3}} u_{t}^{(j)}\cdot\nabla u^{(r-j)}_{t}\cdot u _{t}^{(r+1)}dx\B|\\
\leq& C\sum_{j=0}\| u_{t}^{(j)}\|_{L^{6}}\|\nabla u^{(r-j)}_{t}\|_{L^{3}}\| u _{t}^{(r+1)}\|_{L^{2}}\\
\leq& \sum_{j>0}\|  D   u_{t}^{(j)}\|_{L^{2}}^{2}\| D  u_{t}^{(r-j)} \|_{L^{2}}^{\f{1}{ 2}}
\| D ^{ 2}u_{t}^{(r-j)} \| _{L^{2}}\\&+\|  D   u \|_{L^{2}}^{2}\| D  u_{t}^{(r )} \|_{L^{2}}
\| D ^{2}u_{t}^{(r )} \|_{L^{2}}+\varepsilon\| u _{t}^{(r+1)}\|_{L^{2}}^{2} \\
\leq& \sum_{j>0}\|  D   u_{t}^{(j)}\|_{L^{2}}^{2}\| D  u_{t}^{(r-j)} \|_{L^{2}}^{\f{1}{ 2}}
\| D ^{ 2}u_{t}^{(r-j)} \|_{L^{2}}\\&+\varepsilon\| D ^{2}u _{t}^{(r)}\|_{L^{2}}^{2}+\delta\| u _{t}^{(r+1)}\|_{L^{2}}^{2}+C\| D  u  \|_{L^{2}}^{ 2(4r +3 ) }+C\| D  u_{t}^{(r )}  \|_{L^{2}}^{\f{2(4r +3 )}{ 4r  +1}}\\
\leq& C\sum_{j>0}\|  D   u_{t}^{(j)}\|_{L^{2}}^{2\f{4 r +3 }{4j +1}}
 +C\sum_{j>0}\| D  u_{t}^{(r-j)} \|_{L^{2}}^{\f{2( 4 r +3 )}{4(r-j)+1}}
+\eta\sum_{j>0}\| D  u_{t}^{(r-j)} \|_{L^{2}}^{\f{8  j}{  4(r-j)+1}}
\| D ^{2}u_{t}^{(r-j)} \|^{2}_{L^{2}}\\&+\varepsilon\| D ^{2}u _{t}^{(r)}\|_{L^{2}}^{2}+C\| D  u  \|_{L^{2}}^{ 2(4r +3 ) }+C\| D  u_{t}^{(r )}  \|_{L^{2}}^{\f{2(4r +3 )}{ 4r  +1}}+\delta\| u _{t}^{(r+1)}\|_{L^{2}}^{2},
\ea
\ee
where $\varepsilon$, $\delta$ and $\eta$   are to be chosen later.

Arguing as above, we deduce that
\be\ba
&\B|\sum_{j=0}^{r}\binom{r}{j}\int_{\mathbb{R}^{3}}B_{t}^{(j)}\cdot\nabla B^{(r-j)}_{t}\cdot u _{t}^{(r+1)}dx\B|
\\
\leq& C\sum_{j>0}\|  D   B_{t}^{(j)}\|_{L^{2}}^{2\f{4 r +3 }{4j +1}}
 +C\sum_{j>0}\| D  B_{t}^{(r-j)} \|_{L^{2}}^{\f{2( 4 r +3 )}{4(r-j)+1}}
+\eta\sum_{j>0}\| D  B_{t}^{(r-j)} \|_{L^{2}}^{\f{8  j}{  4(r-j)+1}}
\| D ^{2}B_{t}^{(r-j)} \|^{2}_{L^{2}}\\&+\varepsilon\| D ^{2}B _{t}^{(r)}\|_{L^{2}}^{2}+C\| D  B  \|_{L^{2}}^{ 2(4r +3 ) }+C\| D  B_{t}^{(r )}  \|_{L^{2}}^{\f{2(4r +3 )}{ 4r  +1}}+\delta\| u _{t}^{(r+1)}\|_{L^{2}}^{2},
\ea
\ee
\be\ba
&\B|-\sum_{j=0}^{r}\binom{r}{j}\int_{\mathbb{R}^{3}}u_{t}^{(j)}\cdot\nabla B^{(r-j)}_{t}\cdot B _{t}^{(r+1)}dx\B|
\\
\leq& C\sum_{j>0}\|  D   u_{t}^{(j)}\|_{L^{2}}^{2\f{4 r +3 }{4j +1}}
 +C\sum_{j>0}\| D  B_{t}^{(r-j)} \|_{L^{2}}^{\f{2( 4 r +3 )}{4(r-j)+1}}
+\eta\sum_{j>0}\| D  B_{t}^{(r-j)} \|_{L^{2}}^{\f{8  j}{  4(r-j)+1}}
\| D ^{2}B_{t}^{(r-j)} \|^{2}_{L^{2}}\\&+\varepsilon\| D ^{2}B _{t}^{(r)}\|_{L^{2}}^{2}+C\| D  u  \|_{L^{2}}^{ 2(4r +3 ) }+C\| D  B_{t}^{(r )}  \|_{L^{2}}^{\f{2(4r +3 )}{ 4r  +1}}+\delta\| B _{t}^{(r+1)}\|_{L^{2}}^{2},
\ea
\ee
and
\be\ba\label{3.48v1}
&\B|\sum_{j=0}^{r}\binom{r}{j}\int_{\mathbb{R}^{3}}B_{t}^{(j)}\cdot\nabla u^{(r-j)}_{t}\cdot B _{t}^{(r+1)}dx\B|
\\
\leq& C\sum_{j>0}\|  D   B_{t}^{(j)}\|_{L^{2}}^{2\f{4 r +3 }{4j +1}}
 +C\sum_{j>0}\| D  u_{t}^{(r-j)} \|_{L^{2}}^{\f{2( 4 r +3 )}{4(r-j)+1}}
+\eta\sum_{j>0}\| D  u_{t}^{(r-j)} \|_{L^{2}}^{\f{8  j}{  4(r-j)+1}}
\| D ^{2}u_{t}^{(r-j)} \|^{2}_{L^{2}}\\&+\varepsilon\| D ^{2}u _{t}^{(r)}\|_{L^{2}}^{2}+C\| D  B  \|_{L^{2}}^{ 2(4r +3 ) }+C\| D  B_{t}^{(r )}  \|_{L^{2}}^{\f{2(4r +3 )}{ 4r  +1}}+\delta\| B _{t}^{(r+1)}\|_{L^{2}}^{2}.
\ea
\ee
Inserting   \eqref{3.45v1}-\eqref{3.48v1}  into \eqref{3.44}
and taking $\delta$ small, we thus obtain
\be\ba
& \f12\f{d}{dt}(\|   \nabla   u _{t}^{(r)}\|^{2}_{L^{2}}+\|   \nabla   B _{t}^{(r)}\|^{2}_{L^{2}})+\f12\| u _{t}^{(r+1)}\|_{L^{2}}^{2}+\f12\| B _{t}^{(r+1)}\|_{L^{2}}^{2}\\\leq& C(\| D  u  \|^{2}_{L^{2}}+\| D  B  \|^{2}_{L^{2}})^{  (4r +3 ) }+C(\| D  u_{t}^{(r )}  \|_{L^{2}}^{2}+\| D  B_{t}^{(r )}  \|_{L^{2}}^{2})^{\f{ (4r +3 )}{ 4r  +1}}\\&+C\sum_{j>0}(\| D   u_{t}^{(j)}\|_{L^{2}}^{2}+\|  D   B_{t}^{(j)}\|_{L^{2}}^{2})^{ \f{4 r +3 }{4j +1}}
+C\sum_{j>0}(\| D  u_{t}^{(r-j)} \|_{L^{2}}^{2}+\| D  B_{t}^{(r-j)} \|_{L^{2}}^{2})^{\f{ ( 4 r +3 )}{4(r-j)+1}}
\\&+4\eta\sum_{j>0}(\| D  u_{t}^{(r-j)} \|^{2}_{L^{2}}+\| D  B_{t}^{(r-j)} \|^{2}_{L^{2}})^{\f{4  j}{  4(r-j)+1}}(
\| D ^{2}u_{t}^{(r-j)} \|^{2}_{L^{2}}+
\| D ^{2}B_{t}^{(r-j)} \|^{2}_{L^{2}})\\&+\varepsilon(\| D ^{2}u _{t}^{(r)}\|_{L^{2}}^{2}+
\| D ^{2}B_{t}^{(r)} \|^{2}_{L^{2}}).
\ea
\ee
We derive from this and
\eqref{3.35} that
\be\ba\label{3.38}
&\f{d}{dt}(\|   \nabla   u _{t}^{(r)}\|^{2}_{L^{2}}+\|   \nabla   B _{t}^{(r)}\|^{2}_{L^{2}})+\f{1}{2}\| D  \nabla u _{t}^{(r)}\|_{L^{2}}^{2}+\f{1}{2}\| D  \nabla B _{t}^{(r)}\|_{L^{2}}^{2} +\f{1}{2}\| u _{t}^{(r+1)}\|_{L^{2}}^{2}+\f{1}{2}\|B _{t}^{(r+1)}\|_{L^{2}}^{2}\\\leq& C(\| D  u  \|^{2}_{L^{2}}+\| D  B  \|^{2}_{L^{2}})^{  (4r +3 ) }+C(\| D  u_{t}^{(r )}  \|_{L^{2}}^{2}+\| D  B_{t}^{(r )}  \|_{L^{2}}^{2})^{\f{ (4r +3 )}{ 4r  +1}}\\&+C\sum_{j>0}(\| D   u_{t}^{(j)}\|_{L^{2}}^{2}+\|  D   B_{t}^{(j)}\|_{L^{2}}^{2})^{ \f{4 r +3 }{4j +1}}
+C\sum_{j>0}(\| D  u_{t}^{(r-j)} \|_{L^{2}}^{2}+\| D  B_{t}^{(r-j)} \|_{L^{2}}^{2})^{\f{ ( 4 r +3 )}{4(r-j)+1}}
\\&+8\eta\sum_{j>0}(\| D  u_{t}^{(r-j)} \|^{2}_{L^{2}}+\| D  B_{t}^{(r-j)} \|^{2}_{L^{2}})^{\f{4  j}{  4(r-j)+1}}(
\| D ^{2}u_{t}^{(r-j)} \|^{2}_{L^{2}}+
\| D ^{2}B_{t}^{(r-j)} \|^{2}_{L^{2}})\\\leq&  C(\| D  u_{t}^{(r )}  \|_{L^{2}}^{2}+\| D  B_{t}^{(r )}  \|_{L^{2}}^{2})^{\f{ (4r +3 )}{ 4r  +1}} +C\sum^{r-1}_{j=0}(\|  D   u_{t}^{(j)}\|_{L^{2}}^{2}+\|  D   B_{t}^{(j)}\|_{L^{2}}^{2})^{ \f{4 r +3 }{4j +1}}
\\&+8\eta\sum_{j>0}(\| D  u_{t}^{(r-j)} \|^{2}_{L^{2}}+\| D  B_{t}^{(r-j)} \|^{2}_{L^{2}})^{\f{4  j}{  4(r-j)+1}}(
\| D ^{2}u_{t}^{(r-j)} \|^{2}_{L^{2}}+
\| D ^{2}B_{t}^{(r-j)} \|^{2}_{L^{2}}),
\ea
\ee
where we have chosen $\varepsilon$ sufficiently small.
\end{proof}
With this in hand,
we are now turning to the proof of Theorem  \ref{the1.2}.
\begin{proof}[Proof of Theorem \ref{the1.2}]
We will establish the following fact  by induction on $k\in \mathbb{N}$,
\be\label{3.40}
\left\{\ba
&F_{r,0}\in L^{\f{1}{4r -1}},\\
&\f{d}{dt}F_{r,0}+G_{r,0}\leq CF^{\f{4r +1}{4r -1}}_{r,0},\\
&G_{r,0}\in L^{\f{1}{4r +1}}, \| D  u^{( r) }_{t}\|_{L^{2}}^{2},\| D  B^{( r) }_{t}\|_{L^{2}}^{2}\in L^{\f{1}{4r +1}},\\
&F_{r,1}\in L^{\f{1}{4r +1}},\\
&\f{d}{dt}F_{r,1}+G_{r,1}\leq CF^{\f{4r +2}{4r +1}}_{r,1},\\
&G_{r,1}\in L^{\f{1}{4r +3 }}, \| D ^{2}u^{( r) }_{t}\|_{L^{2}}^{2},\| u^{( r+1) }_{t}\|_{L^{2}}^{2}\in L^{\f{1}{4r -1}},  \| D ^{2}B^{( r) }_{t}\|_{L^{2}}^{2},\| B^{( r+1) }_{t}\|_{L^{2}}^{2}\in L^{\f{1}{4r -1}},\ea\right.\ee
where
\be\ba\label{3.52}
&F_{r,0}=\|u_{t}^{(r)}\|_{L^{2}}^{2}+\|B_{t}^{(r)}\|_{L^{2}}^{2}+\sum_{j=1}^{r-1}\f{4j -1}{4r -1}F_{j,0}^{\f{4r -1}{4j -1}}+\sum_{j=0}^{r-1}\f{4j +1}{4r -1}F_{j,1}^{\f{4r -1}{4j +1}},\\
&G_{r,0}=\f12\| D  u_{t}^{(r)}\|_{L^{2}}^{2}+\f12\| D  B_{t}^{(r)}\|_{L^{2}}^{2}+\f12
\sum_{j=1}^{r-1}F_{j,1}^{\f{4(r-j)-2}{4j +1}}G_{j,1},\\
&F_{r,1}=\| D  u_{t}^{(r)}\|_{L^{2}}^{2}+\| D  B_{t}^{(r)}\|_{L^{2}}^{2}+\sum_{j=0}^{r-1}\f{4j +1}{4r +1}F_{j,1}^{\f{4r +1}{4j +1}},\\
&G_{r,1}=\f12\| D ^{2}u_{t}^{(r)}\|_{L^{2}}^{2}
+\f12\| D ^{2}B_{t}^{(r)}\|_{L^{2}}^{2}+\f12\| u_{t}^{(r+1)}\|_{L^{2}}^{2}+\f12\| B_{t}^{(r+1)}\|_{L^{2}}^{2}+\f12
\sum_{j=0}^{r-1}F_{j,1}^{\f{4(r-j)-2}{4j +1}}G_{j,1}.
\ea\ee
The above definitions of $F_{r,0}, G_{r,0}, F_{r,1}, G_{r,1}$ rest on the energy inequalities  \eqref{3.31v1} and \eqref{3.38v1}.

Hence, for the case $k=0$, we notice that
\be\ba\label{3.39} &F_{0,1}=\| D  u  \|_{L^{2}}^{2}+\| D  B  \|_{L^{2}}^{2}, \\&G_{0,1}=\f12\| D ^{2}u  \|_{L^{2}}^{2}+\f12\| u_{t}  \|_{L^{2}}^{2} +\f12\| D ^{2}B  \|_{L^{2}}^{2}+\f12\| B_{t}  \|_{L^{2}}^{2},\\
&F_{1,0}=(\| u_{t}  \|_{L^{2}}^{2}+\| B_{t}  \|_{L^{2}}^{2})+\f{1}{ 3  }(\| D  u  \|_{L^{2}}^{2}+\| D  B  \|_{L^{2}}^{2})^{3}=(\| u_{t}  \|_{L^{2}}^{2}+\| B_{t}  \|_{L^{2}}^{2})+\f{1}{ 3  }F_{0,1}^{3},\\
&G_{1,0}=\f12\|  D   u_{t}  \|_{L^{2}}^{2}+\f12\|  D   B_{t}  \|_{L^{2}}^{2}+\f12F_{0,1}^{2}G_{0,1}, \\
&F_{1,1}=\|  D   u_{t}  \|_{L^{2}}^{2}+\|  D   B_{t}  \|_{L^{2}}^{2}+\f{1}{5}F_{0,1}^{5},\\
&G_{1,1}=\f12\|  D ^{2} u_{t}  \|_{L^{2}}^{2}+\f12\| u_{tt}  \|_{L^{2}}^{2}+\f12\|  D ^{2} B_{t}  \|_{L^{2}}^{2}+\f12\| B_{tt}  \|_{L^{2}}^{2}+\f12F_{0,1}^{4}G_{0,1} +\f12F_{1,0}^{\f{4  }{3 }}G_{1,0}.\ea
\ee
  From \eqref{3.6v1}, we know that
 \be\label{3.43v1}\f{d}{dt}F_{0,1}+G_{0,1}\leq CF_{0,1}^{3}.
\ee
The   natural energy estimates   of the MHD equations leads to
  \be\label{3.42}
F_{0,1}=\| D  u  \|_{L^{2}}^{2}+\| D  B  \|_{L^{2}}^{2}\in L^1.\ee
Taking advantage of
Lemma \ref{hofflemma}, we conclude by \eqref{3.43v1} and \eqref{3.42} that
$G_{0,1}\in L^{\f{1}{3 }}$. Hence, we further have
$$
F_{ 1,0}=\| u_{t}  \|_{L^{2}}^{2}+\| B_{t}  \|_{L^{2}}^{2}+\f13(\| D  u  \|_{L^{2}}^{2}+\| D  B \|_{L^{2}}^{2})^{3}\in   L^{\f{ 1 }{ 3 }}.
$$
To deal with the last term on the right hand side of \eqref{3.10v1}, we deduce from  \eqref{3.2v1} that
 \be\ba\label{3.13}
&\f13\f{d}{dt}(\| D  u (t)\|^{2}_{L^{2}}+\| D  B (t)\|^{2}_{L^{2}})^{3}
 + \f32(\| D ^{2}u (t)\|_{L^{2}}^{2}+ \| D ^{2}B (t)\|_{L^{2}}^{2})(\| D  B \|_{L^{2}}^{2}+\| D  u\|_{L^{2}}^{2})^{2}\\
\leq& C(\| D  u (t)\|_{L^{2}}^{2}+\| D  B (t)\|_{L^{2}}^{2})^{5}.
\ea\ee
By combining this and \eqref{3.13v2} and choosing $\eta$ sufficiently small, we discover that
\be\ba\label{3.13v1}
&\f{d}{dt}\B[\| u _{t}\|^{2}_{L^{2}}+\|B _{t}\|^{2}_{L^{2}}+\f13(\| D  u (t)\|_{L^{2}}^{2}+\| D  B (t)\|_{L^{2}}^{2})^{3}\B]\\&+\f12\| D  u _{t}\|_{L^{2}}^{2}+\f12\| D  B _{t}\|_{L^{2}}^{2}
+\f12( \| D  u \|^{2}+\| D  B \|^{2})^{2}(\| D ^{2}u (t)\|_{L^{2}}^{2}+\| D ^{2}B (t)\|_{L^{2}}^{2})\\
\leq &C
(\| D  u (t)\|_{L^{2}}^{2}+\| D  B (t)\|_{L^{2}}^{2})^{5}+C\|u _{t}\|^{\f{10}{3 }}_{L^{2}}+C\|B _{t}\|^{\f{10}{3 }}_{L^{2}}\\
\leq &C\B[\| u _{t}\|^{2}_{L^{2}}+\|B _{t}\|^{2}_{L^{2}}+\f13(\| D  u (t)\|_{L^{2}}^{2}+\| D  B (t)\|_{L^{2}}^{2})^{3}\B]^{\f{  5 }{3 }}.
\ea\ee
We derive from  \eqref{3.13v1}  that
\be\label{3.45}
\f{d}{dt}F_{ 1,0}+G_{1,0}\leq CF_{ 1,0}^{\f{  5 }{3 }}.
\ee
As a consequence, one employs  Lemma \ref{hofflemma} to verify that
$$\ba
G_{1,0}\in L^{\f{1}{ 5}},
\ea
$$
which implies that
 \be\ba\label{3.48}
\| D  u  _{t}\|_{L^{2}}^{2}, \| D  B  _{t}\|_{L^{2}}^{2}\in L^{\f{1}{5}}.
\ea
\ee
According to facts
\eqref{3.42}  and \eqref{3.48}, we get
\be\label{3.49}
F_{1,1}\in L^{\f{1}{5}}.
\ee
By  straightforward calculations, we deduce from
\eqref{3.43v1}   that
\be\ba
\f{1}{5}\f{d}{dt}F_{0,1}^{5}+F_{0,1}^{4}G_{0,1} \leq CF_{0,1}^{7}.
\ea\ee
It follows from the latter inequality and \eqref{3.24v1}  that
\be\ba \label{3.50v1}
&\f{d}{dt}\B[\| D   u _{t}\|^{2}_{L^{2}}+\| D   B _{t}\|^{2}_{L^{2}}+\f{1}{5}F_{0,1}^{5}\B]\\&+\f12\|  D ^{2} u_{t}  \|_{L^{2}}^{2}+\f12\| u_{tt}  \|_{L^{2}}^{2}+\f{1}{2}\| D ^{2 }u _{t}\|_{L^{2}}^{2}+\f{1}{2}\| D ^{2 }B _{t}\|_{L^{2}}^{2}
+ F_{0,1}^{4}G_{0,1} \\
\leq& C( \| D  u_{t}\|_{L^{2}}^{2}+\| D  B_{t}\|_{L^{2}}^{2})^{\f{7}{5}}
+C(\| D  u \|_{L^{2}}^{2}+\| D  B \|_{L^{2}}^{2})^{7}
\\&+8\eta(\| D  u \|^{2}_{L^{2}}+\| D  B \|^{2}_{L^{2}} )^{4}(\| D ^{2 }u \|_{L^{2}}^{2}+\| D ^{2 }B \|_{L^{2}}^{2})+F_{0,1}^{7}.
\ea
\ee
By virtue of the definition of $F_{1,1}$ in \eqref{3.39}, one infers that
$$\ba
&F_{0,1}^{7}\leq CF_{1,1}^{\f{ 7}{5}},\\
&(\| D  u \|^{2}_{L^{2}}+\| D  B \|^{2}_{L^{2}} )^{4}(\| D ^{2 }u \|_{L^{2}}^{2}+\| D ^{2 }B \|_{L^{2}}^{2})\leq  CF_{1,0}^{ 4    }G_{1,0},
\ea$$
which helps us to  rewrite  \eqref{3.50v1} as
\be\ba\label{3.50}
 \f{d}{dt}F_{1,1}+G_{1,1}
\leq C F_{1,1}^{\f{ 7}{5}},
\ea
\ee
where  we have chosen $\eta$   sufficiently small.

Making full use  of
\eqref{3.49} and Lemma \ref{hofflemma}, we deduce from \eqref{3.50} that
$$
G_{1,1} \in L^{\f17},
$$
which yields
$$
\|  D ^{2} u_{t}  \|_{L^{2}}^{2},\| u_{tt}  \|_{L^{2}}^{2},  \|  D ^{2} B_{t}  \|_{L^{2}}^{2},\| B_{tt}  \|_{L^{2}}^{2}\in L^{\f17}.
$$
Hence, we prove \eqref{3.40} for $r=1$.

Next, we assume that \eqref{3.40} is valid for $r\leq k-1$. It is enough to show  \eqref{3.40} for $r=k$. To this end, inductive hypothesis yields
\be\label{3.56}
\left\{\ba
&F_{j,0}\in L^\f{1}{4j   -1}, j\leq k-1,\\
&F_{j,1}\in L^\f{1}{4j + 1}, j\leq k-1,\\
&\| u^{( k) }_{t}\|_{L^{2}}^{2}+\| B^{( k) }_{t}\|_{L^{2}}^{2}\in L^{\f{1}{4k -1}}. \ea\right.\ee
As a result, one has
\be\label{3.57} F_{k,0}=\|u_{t}^{(k)}\|_{L^{2}}^{2}+\|B_{t}^{(k)}\|_{L^{2}}^{2}+\sum_{j=1}^{k-1}\f{4j -1}{4k -1}F_{j,0}^{\f{4k -1}{4j -1}}+\sum_{j=0}^{k-1}\f{4j +1}{4k -1}F_{j,1}^{\f{4k -1}{4j +1}}\in L^{\f{1}{4k -1}}.
 \ee
Inductive hypothesis  \eqref{3.40} for $j\leq k-1$ gives
$$
\f{d}{dt}F_{j,1}+G_{j,1}\leq CF^{\f{4j +2}{4j +1}}_{j,1},\f{d}{dt}F_{j,0}+G_{j,0}\leq CF^{\f{4j +1}{4j -1}}_{r,0}.
$$
Hence, we compute
\be\ba
\f{4j +1}{4k -1}\f{d}{dt}F_{j,1}^{\f{4k -1}{4j +1}}+F_{j,1}^{\f{4(k-j)-2}{4j +1}}G_{j,1}\leq CF^{\f{4k +1}{4j +1}}_{j,1},\\
\f{4j -1}{4k -1}\f{d}{dt}F_{j,0}^{\f{4k -1}{4j -1}}+F_{j,1}^{\f{4(k-j) }{4j -1}}G_{j,0}\leq CF^{\f{4k +1}{4j -1}}_{j,0},
\ea\ee
which together with \eqref{3.31v1} leads to
\be\ba\label{3.57v1}
& \f{d}{dt}\B[\| u _{t}^{(k)}\|^{2}_{L^{2}} +\|B _{t}^{(k)}\|^{2}_{L^{2}}+\sum_{j=1}^{k-1}\f{4j +1}{4k -1}F_{j,1}^{\f{4k -1}{4j +1}}+\sum_{j=1}^{k-1}\f{4j -1}{4k -1}F_{j,0}^{\f{4k -1}{4j -1}}\B]\\&+\f12\| D  u _{t}^{(k)}\|_{L^{2}}^{2} +\f12\| D  B _{t}^{(k)}\|_{L^{2}}^{2} +\sum_{j=1}^{k-1}F_{j,1}^{\f{4(r-j)-2}{4j +1}}G_{j,1} +\sum_{j=1}^{k-1}F_{j,1}^{\f{4(r-j) }{4j -1}}G_{j,0}\\
\leq& C(\|u_{t}^{(k)}\|^{2}+\|B_{t}^{(k)}\|^{2})
^{\f{  4k +1 }{4k -1}}
_{L^{2}}
+C\sum_{j=1}^{k-1}(\|u_{t}^{(j)}\|_{L^{2}}^{2}+\|B_{t}^{(j)}\|_{L^{2}}^{2})
^{\f{ 4k +1 }{4j -1}}
\\&+8\eta\sum_{j=1}^{k-1}(\| D  u_{t}^{( j)} \|_{L^{2}}^{2}+\| D  B_{t}^{( j)} \|_{L^{2}}^{2})^{\f{2 (2(k-j)-1)}{  4j +1}}
(\| D ^{2}u_{t}^{(j)} \|^{2}_{L^{2}}+\| D ^{2}B_{t}^{(j)} \|^{2}_{L^{2}})\\&+C\sum_{j=0}^{k-1}(\| D  u_{t}^{( j)} \|_{L^{2}}^{2}+\| D  B_{t}^{( j)} \|_{L^{2}}^{2})^{\f{ 4k +1 }{4j +1}}+\sum_{j=1}^{k-1}F^{\f{4k +1}{4j +1}}_{j,1}+\sum_{j=1}^{k-1}F^{\f{4k +1}{4j -1}}_{j,0}.
\ea
\ee
In view  of definition of $F_{j,0}$ and  $F_{j,1}$ in \eqref{3.39}, one has
$$\ba
&\| D  u_{t}^{( j)} \|_{L^{2}}+ \| D  B_{t}^{( j)} \|_{L^{2}}\leq CF_{j,1}\\
&\|  u_{t}^{( j)} \|_{L^{2}}+ \|  B_{t}^{( j)} \|_{L^{2}}\leq CF_{j,0}.\ea$$ In addition, we derive from \eqref{3.52} that
$$(\| D  u_{t}^{( j)} \|_{L^{2}}^{2}+\| D  B_{t}^{( j)} \|_{L^{2}}^{2})^{\f{2 (2(k-j)-1)}{  4j +1}}(\| D ^{2}u_{t}^{(j)} \|^{2}_{L^{2}}+\| D ^{2}B_{t}^{(j)} \|^{2}_{L^{2}})
\leq CF_{j,1}^{\f{4(k-j)-2}{4j +1}}G_{j,1}.$$
Consequencely, taking $\eta$ small, we reformulate \eqref{3.57v1}
as
\be\ba\label{3.70}
& \f{d}{dt}\B[\| u _{t}^{(k)}\|^{2}_{L^{2}}+\| B _{t}^{(k)}\|^{2}_{L^{2}}+\sum_{j=1}^{k-1}\f{4j +1}{4k -1}F_{j,1}^{\f{4k -1}{4j +1}}+\sum_{j=1}^{k-1}\f{4j -1}{4k -1}F_{j,0}^{\f{4k -1}{4j -1}}\B]\\&+\f12\| D  u _{t}^{(k)}\|_{L^{2}}^{2} +\f12\| D  B _{t}^{(k)}\|_{L^{2}}^{2} +\f12F_{j,1}^{\f{4(r-j)-2}{4j +1}}G_{j,1} +\f12F_{j,1}^{\f{4(r-j) }{4j -1}}G_{j,0}\\
\leq& C(\|u_{t}^{(k)}\|^{2}+\|B_{t}^{(k)}\|^{2})
^{\f{ (4k +1)}{4k -1}}
_{L^{2}}
  +C\sum_{j=0}^{k-1}F^{\f{4k +1}{4j +1}}_{j,1}+C\sum_{j=1}^{k-1}F^{\f{4k +1}{4j -1}}_{j,0},
\ea
\ee
On account of   \eqref{3.57}, we write
$$
C\sum_{j=0}^{k-1}F^{\f{4k +1}{4j +1}}_{j,1}\leq CF_{k,0}^{\f{  4k +1 }{4k -1}},~~
C\sum_{j=1}^{k-1}F^{\f{4k +1}{4j -1}}_{j,0}\leq CF_{k,0}^{\f{  4k +1 }{4k -1}}.
 $$
 As a consequence, we further deduce from \eqref{3.70} that
$$
\f{d}{dt}F_{k,0}+G_{k,0}\leq CF_{k,0}^{\f{  4k +1 }{4k -1}}.
$$
With the help of  Lemma \ref{hofflemma}, we combine  this and \eqref{3.57}  to get
$$
G_{k,0}\in L^{\f{1}{4k +6\alpha  -5}},
$$
which means
\be\label{3.71}
\| D  u _{t}^{(k)}\|_{L^{2}}^{2}, \| D  B _{t}^{(k)}\|_{L^{2}}^{2} \in L^{\f{1}{4k +6\alpha  -5}}.\ee
This verifies $\eqref{3.40}_{1}$-$\eqref{3.40}_{3}$ for $r=k $. We turn our attentions to proving  $\eqref{3.40}_{4}$-$\eqref{3.40}_{6}$ for $r=k $.
To do this, we deduce from \eqref{3.71} and \eqref{3.56} that
\be
F_{k,1}=\| D  u_{t}^{(k)}\|_{L^{2}}^{2}+\| D  B_{t}^{(k)}\|_{L^{2}}^{2}+\sum_{j=0}^{k-1}\f{4j +1}{4k +1}F_{j,1}^{\f{4k +1}{4j +1}}\in L^{\f{1}{4k +6\alpha  -5}}.\ee
Since
\eqref{3.40} is valid for $r\leq k-1$, one arrives at
$$
\f{d}{dt}F_{j,1}+G_{j,1}\leq CF^{\f{4j +2}{4j +1}}_{j,1},\f{d}{dt}F_{r,0}+G_{r,0}\leq CF^{\f{4r +1}{4r -1}}_{r,0},
$$
which implies
\be\ba\label{3.63}
\f{4j +1}{4k +1}\f{d}{dt}F_{j,1}^{\f{4k +1}{4j +1}}+F_{j,1}^{\f{4(k-j) }{4j +1}}G_{j,1}\leq CF^{\f{4k +3 }{4j +1}}_{j,1}.
\ea\ee
It is apparent from  \eqref{3.52} that
$$\ba
&\sum_{j>0}(\| D  u_{t}^{(r-j)} \|^{2}_{L^{2}}+\| D  B_{t}^{(r-j)} \|^{2}_{L^{2}})^{\f{4  j}{  4(r-j)+1}}(
\| D ^{2}u_{t}^{(r-j)} \|^{2}_{L^{2}}+
\| D ^{2}B_{t}^{(r-j)} \|^{2}_{L^{2}})\leq C\sum_{j=0}^{k-1}F_{j,1}^{\f{4(k-j) }{4j +1}}G_{j,1}.
\ea$$
It follows from \eqref{3.38v1}  and \eqref{3.63}   that
\be\ba
& \f{d}{dt}\B[\|  D  u _{t}^{(k)}\|^{2}_{L^{2}}+\|  D  B _{t}^{(k)}\|^{2}_{L^{2}}+\sum_{j=0}^{k-1}\f{4j +1}{4k +1}F_{j,1}^{\f{4k +1}{4j +1}}\B]\\&+\f12\| D ^{2 }u _{t}^{(k)}\|_{L^{2}}^{2} +\f12\| D ^{2 }B _{t}^{(k)}\|_{L^{2}}^{2} +\f12\| u _{t}^{(k+1)}\|_{L^{2}}^{2}+\f12\| B _{t}^{(k+1)}\|_{L^{2}}^{2}+\f12\sum_{j=0}^{k-1}F_{j,1}^{\f{4(k-j) }{4j +1}}G_{j,1}\\ \leq& C(\| D  u_{t}^{(k )}  \|_{L^{2}}^{2}+\| D  B_{t}^{(k )}  \|_{L^{2}}^{2})^{\f{ 4k +3  }{ 4k  +1}}+C\sum_{j=0}^{k-1}F_{j,1}^{ \f{4 k +3 }{4j +1}},
\ea
\ee
where we have taken $\eta$ sufficiently small in \eqref{3.38v1}.\\
 One derives from  \eqref{3.52} that
  $$\sum_{j=0}^{k-1}F_{j,1}^{ \f{4 k +3 }{4j +1}}
\leq CF_{k,1}^{\f{  4k +3  }{4k +1}}.$$
Finally, we arrive at
$$
\f{d}{dt}F_{k,1}+G_{k,1}\leq CF_{k,1}^{\f{  4k +3  }{4k +1}}.
$$
In view of $
F_{k,1} \in L^{\f{1}{4k +6\alpha  -5}} $ and Lemma \ref{hofflemma}, we find $
G_{k,1} \in L^{\f{1}{4k +3 }}. $ We end up with $\| u _{t}^{(k+1)}\|_{L^{2}}^{2}, \| u _{t}^{(k+1)}\|_{L^{2}}^{2}\in L^{\f{1}{4k +3 }}$. Hence, we prove $\eqref{3.40}$ for $r=k$. At this stage, the proof of Theorem \ref{the1.2} is completed.
\end{proof}
\section{Space-time derivative estimates for the
magnetohydrodynamic equations}
\begin{lemma}\label{lemma4.1}Let the pair $(u, B)$  be a classical solution to the MHD system \eqref{MHD}. Then for each $t >0$ and $\eta >0$,  there holds
\be\ba\label{4.13v3}
& \f{d}{dt} (\|  D ^{s}u _{t}^{(r)}\|^{2}_{L^{2}}+\|  D ^{s}B _{t}^{(r)}\|^{2}_{L^{2}})+\|  D ^{s+1}u _{t}^{(r)}\|^{2}_{L^{2}}+\| D ^{s+1}B _{t}^{(r)}\|^{2}_{L^{2}}\\
\leq &  C\sum_{j }^{r-1}(\| D ^{ s }B_{t}^{(j)}\|_{L^{2}}^{2}+\| D ^{ s }u_{t}^{(j)}\|_{L^{2}}^{2})
^{\f{ 4r +2(s+1)-1 }{4j +2s -1}}+  C(\| D ^{ s  }u_{t}^{(r)}\|_{L^{2}}+\| D ^{ s  }B_{t}^{(r)}\|_{L^{2}})
^{\f{  4r +2(s+1)-1 }
{4r +2s-1}}
 \\&+C\sum_{j=0 }^{r}
(\| D   u_{t}^{(j)}\|_{L^{2}}^{2}+\| D   B_{t}^{(j)}\|_{L^{2}}^{2})^{\f{ 4r +2(s+1)-1 )}{4j  +1}}
 \\&+4\eta\sum_{j>0 }^{r}(\| D ^{s}u^{(r-j )}_{t}\|_{L^{2}}^{2}+\| D ^{s}B^{(r-j )}_{t}\|_{L^{2}}^{2})^{\f{4  j}{ 4(r-j)+2 s -1}}(\| D ^{s+1} u^{(r-j )}_{t}\|_{L^{2}}^{2}+\| D ^{s+1} B^{(r-j )}_{t}\|_{L^{2}}^{2} ),
\ea\ee
and
\be\ba\label{4.15v1}
& \f{1}{2}\f{d}{dt} (\|  D ^{s}u \|^{2}_{L^{2}}+\|  D ^{s}B \|^{2}_{L^{2}})+\f12\|  D ^{s+1}u \|^{2}_{L^{2}}+\f12\|  D ^{s+1}B \|^{2}_{L^{2}}\\
 \leq & C(\| D ^{ s }u \|_{L^{2}}^{2}+\| D ^{ s }B \|_{L^{2}}^{2})
^{\f{   2(s+1)-1 }
{ 2s -1}} +C(
\| D  u \|_{L^{2}}^{2}+
\| D  B \|_{L^{2}}^{2})^{   2(s+1)-1  }.
\ea
\ee
\end{lemma}
\begin{proof}
(1)
Applying $ D ^{s}\partial^{r}_{t} $ to the MHD equations \eqref{MHD}, we conclude by Leibniz rule that
\be\ba\label{4.1}  D ^{s} u^{(r+1) }_{t}-\Delta D ^{s}u^{(r)}_{t}+\sum_{j=0}^{r}\binom{r}{j} D ^{s}(u_{t}^{(j)}\cdot\nabla u^{(r-j)}_{t})- \sum_{j=0}^{r}\binom{r}{j} D ^{s}(B_{t}^{(j)}\cdot\nabla B^{(r-j)}_{t}) +\nabla  D ^{s}p^{(r)}_{t}=0,\\
 D ^{s} B^{(r+1) }_{t}-\Delta D ^{s}B^{(r)}_{t}+\sum_{j=0}^{r}\binom{r}{j} D ^{s}(u_{t}^{(j)}\cdot\nabla B^{(r-j)}_{t})- \sum_{j=0}^{r}\binom{r}{j} D ^{s}(B_{t}^{(j)}\cdot\nabla u^{(r-j)}_{t})  =0.
\ea
\ee
We multiply the above equations by  $ D ^{s}u _{t}^{(r)}$ and $ D ^{s}B _{t}^{(r)}$, sum and integrate to
obtain
\be\ba\label{4.2}
&\f{1}{2}\f{d}{dt} (\|  D ^{s}u _{t}^{(r)}\|^{2}_{L^{2}}+\|  D ^{s}B _{t}^{(r)}\|^{2}_{L^{2}})+\|  D ^{s+1}u _{t}^{(r)}\|^{2}_{L^{2}}+\| D ^{s+1}B _{t}^{(r)}\|^{2}_{L^{2}}\\
=&- \sum_{j=0}^{r}\binom{r}{j}\int_{\mathbb{R}^{3}}D ^{s}(u_{t}^{(j)}\cdot\nabla u^{(r-j)}_{t})\cdot  D ^{s}u _{t}^{(r)}dx+\sum_{j=0}^{r}\binom{r}{j}\int_{\mathbb{R}^{3}}D ^{s}(B_{t}^{(j)}\cdot\nabla B^{(r-j)}_{t})\cdot  D ^{s}u _{t}^{(r)}dx\\&- \sum_{j=0}^{r}\binom{r}{j}\int_{\mathbb{R}^{3}}D ^{s}(u_{t}^{(j)}\cdot\nabla B^{(r-j)}_{t})\cdot  D ^{s}B _{t}^{(r)}dx+\sum_{j=0}^{r}\binom{r}{j}\int_{\mathbb{R}^{3}}D ^{s}(B_{t}^{(j)}\cdot\nabla u^{(r-j)}_{t})\cdot  D ^{s}B_{t}^{(r)}dx\\
=&I+II+III+IV.
\ea
\ee
In the light of the incompressible condition, we conclude by integration by parts that
\be\ba \int_{\mathbb{R}^{3}}D ^{s}(u \cdot\nabla u^{(r )}_{t})\cdot D ^{s}u _{t}^{(r)}dx=\int_{\mathbb{R}^{3}}[ D ^{s}(u \cdot\nabla u^{(r )}_{t})-u \cdot D ^{s}\nabla u^{(r )}_{t} ]\cdot D ^{s}u _{t}^{(r)}dx,
\ea
\ee
which means
\be\ba
I=&-\int_{\mathbb{R}^{3}}[ D ^{s}(u \cdot\nabla u^{(r )}_{t})-u \cdot\nabla  D ^{s}u^{(r )}_{t}]\cdot  D ^{s}u _{t}^{(r)}dx\\&-\sum_{j>0}\binom{r}{j}\int_{\mathbb{R}^{3}}D ^{s}(u_{t}^{(j)}\cdot\nabla u^{(r-j)}_{t})\cdot  D ^{s}u _{t}^{(r)}dx.
\ea
\ee
Following the same path, we can show that
\be\ba
III= &-\int_{\mathbb{R}^{3}}D ^{s}(u \cdot\nabla B^{(r )}_{t})\cdot  D ^{s}B _{t}^{(r)}dx-\sum_{j>0}\binom{r}{j}\int_{\mathbb{R}^{3}} D ^{s}(u_{t}^{(j)}\cdot\nabla B^{(r-j)}_{t})\cdot  D ^{s}B _{t}^{(r)}dx\\
=&-\int_{\mathbb{R}^{3}}[ D ^{s}(u \cdot\nabla B^{(r )}_{t})-u \cdot\nabla  D ^{s}B^{(r )}_{t}]\cdot  D ^{s}B _{t}^{(r)}dx\\&-\sum_{j>0}\binom{r}{j}\int_{\mathbb{R}^{3}}D ^{s}(u_{t}^{(j)}\cdot\nabla B^{(r-j)}_{t})\cdot  D ^{s}B _{t}^{(r)}dx.
\ea
\ee
Integrating by parts, we infer that
$$\int_{\mathbb{R}^{3}}B \cdot\nabla  D ^{s}B^{(r )}_{t} \cdot  D ^{s}u _{t}^{(r)}dx+\int_{\mathbb{R}^{3}}B \cdot\nabla  D ^{s}u^{(r )}_{t} \cdot  D ^{s}B _{t}^{(r)}dx=0,$$
from which it follows that
$$\ba
&II+IV\\
=&
\int_{\mathbb{R}^{3}}[ D ^{s}(B \cdot\nabla B^{(r )}_{t})-B \cdot\nabla  D ^{s}B^{(r )}_{t}]\cdot  D ^{s}u _{t}^{(r)}dx+\sum_{j>0}\binom{r}{j}\int_{\mathbb{R}^{3}}D ^{s}(B_{t}^{(j)}\cdot\nabla B^{(r-j)}_{t})\cdot  D ^{s}u _{t}^{(r)}dx\\&
+\int_{\mathbb{R}^{3}}[ D ^{s}(B \cdot\nabla u^{(r )}_{t})-B \cdot\nabla  D ^{s}u^{(r )}_{t}]\cdot  D ^{s}B _{t}^{(r)}dx+\sum_{j>0}\binom{r}{j}\int_{\mathbb{R}^{3}}D ^{s}(B_{t}^{(j)}\cdot\nabla u^{(r-j)}_{t})\cdot  D ^{s}B _{t}^{(r)}dx.
\ea
$$
By means of the H\"older inequality and Kato-Ponce type     commutator estimate \eqref{katoponce}, one arrives at
\be\ba\label{3.5} & \B|\int_{\mathbb{R}^{3}}[ D ^{s}(u \cdot\nabla u^{(r )}_{t})-u \cdot\nabla D ^{s}u^{(r )}_{t}]\cdot D ^{s}u _{t}^{(r)}dx\B|\\
\leq& (\| D ^{s}u\|_{L^{2}}\|\nabla u^{(r )}_{t}\|_{L^{3}}+\|\nabla u\|_{L^{3}}\| D ^{s} u^{(r )}_{t}\|_{L^{2}})\| D ^{s}u _{t}^{(r)}\|_{L^{6}}.
\ea
\ee
Interpolation inequality \eqref{keyinequality} ensures that
$$\ba
&\|\nabla u^{(r )}_{t}\|_{L^{3}}\leq C\| D  u^{(r )}_{t}\|^{\f{2(s+1) -5 }{2(s-1)}}_{L^{2}}\| D ^{s} u^{(r )}_{t}\|_{L^{2}}^{\f{1}{2(s-1)}},\\
&\|\nabla u \|_{L^{3}}\leq C\| D  u \|^{\f{2(s+1) -5 }{2(s-1)}}_{L^{2}}\| D ^{s} u \|_{L^{2}}^{\f{1}{2(s-1)}}.
\ea
$$
Plugging  the previous inequalities into \eqref{3.5}, we derive from Young inequality  that
\be\ba\label{4.7}
&\B|\int_{\mathbb{R}^{3}}[ D ^{s}(u \cdot\nabla u^{(r )}_{t})-u \cdot\nabla  D ^{s}u^{(r )}_{t}]\cdot D ^{s}u _{t}^{(r)}dx\B|\\
\leq& C\| D ^{s}u\|_{L^{2}}\| D  u^{(r )}_{t}\|^{\f{2(s+1) -5 }{2(s-1)}}_{L^{2}}\| D ^{s} u^{(r )}_{t}\|_{L^{2}}^{\f{1}{2(s-1)}}\| D ^{s+1}u _{t}^{(r)}\|_{L^{2}}\\&+C\| D  u \|^{\f{2(s+1) -5 }{2(s-1)}}_{L^{2}}\| D ^{s} u \|_{L^{2}}^{\f{1}{2(s-1)}}\| D ^{s} u^{(r )}_{t}\|_{L^{2}} \| D ^{s+1}u _{t}^{(r)}\|_{L^{2}}\\
\leq& C\| D ^{s}u\|_{L^{2}}^{2}\| D  u^{(r )}_{t}\|^{\f{2(s+1) -5 }{ (s-1)}}_{L^{2}}\| D ^{s} u^{(r )}_{t}\|_{L^{2}}^{\f{1}{ (s-1)}} \\&+C\| D ^{s} u^{(r )}_{t}\|^{2}_{L^{2}}\| D  u \|^{\f{2(s+1) -5 }{ (s-1)}}_{L^{2}}\| D ^{s} u \|_{L^{2}}^{\f{1}{ (s-1)}} +\varepsilon\| D ^{s+1}u _{t}^{(r)}\|^{2}_{L^{2}},
\ea
\ee
where $\varepsilon$ will be fixed below.\\
Using the H\"older inequality, Kato-Ponce type    Leibniz rules \eqref{fraleibnilaw} and the Young  inequality, one infers that
\be\ba\label{4.8}
 &\B|\sum_{j>0}\binom{r}{j}\int_{\mathbb{R}^{3}}D ^{s}(u_{t}^{(j)}\cdot\nabla u^{(r-j)}_{t})\cdot D ^{s}u _{t}^{(r)}dx\B|\\
  = & \B|\binom{r}{j}\sum_{j>0}\int_{\mathbb{R}^{3}}D ^{(s-1)}(u_{t}^{(j)}\cdot\nabla u^{(r-j)}_{t})\cdot D ^{s+1}u _{t}^{(r)}dx\B|\\
 \leq & C \sum_{j>0}\| D ^{(s-1)}u_{t}^{(j)}\|_{L^{6}}\|\nabla u^{(r-j)}_{t}\|_{L^{3}}\| D ^{s+1}u _{t}^{(r)}\|_{L^{2}}
\\& +C\sum_{j>0}\|u_{t}^{(j)}\|_{L^{6}}\| D ^{(s-1)}\nabla u^{(r-j)}_{t}\|_{L^{3}}\| D ^{s+1}u _{t}^{(r)}\|_{L^{2}}
\\
 \leq & C\sum_{j>0}\| D ^{ s }u_{t}^{(j)}\|_{L^{2}}\| D  u^{(r-j)}_{t} \|^{\f{2(s+1) -5 }{2(s-1)}}_{L^{2}}\| D ^{s} u^{(r-j)}_{t} \|_{L^{2}}^{\f{1}{2(s-1)}}\| D ^{s+1}u _{t}^{(r)}\|_{L^{2}}
 \\&+C\sum_{j>0}\| D   u_{t}^{(j)}\|_{L^{2}}\| D ^{s}u^{(r-j )}_{t}\|^{\f{2(s+1) -5 }{2(s-1)}}_{L^{2}}\| D ^{s+1} u^{(r-j )}_{t}\|_{L^{2}}^{\f{1}{2(s-1)}} \| D ^{s+1}u _{t}^{(r)}\|_{L^{2}}\\
 \leq & C\sum_{j>0}\| D ^{ s }u_{t}^{(j)}\|_{L^{2}}^{2}\| D  u^{(r-j)}_{t} \|^{\f{2(s+1) -5 }{ (s-1)}}_{L^{2}}\| D ^{s} u^{(r-j)}_{t} \|_{L^{2}}^{\f{1}{ (s-1)}}
 \\&+C\sum_{j>0}\| D   u_{t}^{(j)}\|^{2}_{L^{2}}\| D ^{s}u^{(r-j )}_{t}\|^{\f{2(s+1) -5 }{ (s-1)}}_{L^{2}}\| D ^{s+1} u^{(r-j )}_{t}\|_{L^{2}}^{\f{1}{ (s-1)}} +\varepsilon\| D ^{s+1}u _{t}^{(r)}\|_{L^{2}}^{2}.
 \ea\ee
A combination of \eqref{4.7} and \eqref{4.8} ensures that
 \be\ba\label{4.11v2}
&\B|-\sum_{j}^{r}\binom{r}{j}\int_{\mathbb{R}^{3}}D ^{s}(u_{t}^{(j)}\cdot\nabla u^{(r-j)}_{t})\cdot D ^{s}u _{t}^{(r)}dx\B|
\\
\leq& C\sum_{j }^{r}\| D ^{ s }u_{t}^{(j)}\|_{L^{2}}^{2}\| D  u^{(r-j)}_{t} \|^{\f{2(s+1) -5 }{ (s-1)}}_{L^{2}}\| D ^{s} u^{(r-j)}_{t} \|_{L^{2}}^{\f{1}{ (s-1)}}
 \\&+C\sum_{j>0 }^{r}\| D   u_{t}^{(j)}\|^{2}_{L^{2}}\| D ^{s}u^{(r-j )}_{t}\|^{\f{2(s+1) -5 }{ (s-1)}}_{L^{2}}\| D ^{s+1} u^{(r-j )}_{t}\|_{L^{2}}^{\f{1}{ (s-1)}} +\varepsilon\| D ^{s+1}u _{t}^{(r)}\|_{L^{2}}^{2}.
\ea\ee
According to the Young inequality, we obtain
\be\ba\label{4.12v2}
&\sum_{j }^{r}\| D ^{ s }u_{t}^{(j)}\|_{L^{2}}^{2}\| D  u^{(r-j)}_{t} \|^{\f{2(s+1) -5 }{ (s-1)}}_{L^{2}}\| D ^{s} u^{(r-j)}_{t} \|_{L^{2}}^{\f{1}{ (s-1)}}\\
\leq& C\| D ^{ s }u_{t}^{(j)}\|_{L^{2}}
^{\f{2(4r +2(s+1)-1)}{4j +2s -1}}+C
\| D  u^{(r-j)}_{t} \|^{\f{2(4r +2(s+1)-1)}{4(r-j)+2 -1}}_{L^{2}}\\&+C
\| D ^{s}u^{(r-j)}_{t} \|^{\f{2(4r +2(s+1)-1)}{4(r-j)+2 s -1}}_{L^{2}},
\ea\ee
and
\be\ba\label{4.13v2}
&\sum_{j>0 }^{r}\| D   u_{t}^{(j)}\|^{2}_{L^{2}}\| D ^{s}u^{(r-j )}_{t}\|^{\f{2(s+1) -5 }{ (s-1)}}_{L^{2}}\| D ^{s+1} u^{(r-j )}_{t}\|_{L^{2}
}^{\f{1}{ (s-1)}}\\
\leq&C\sum_{j>0 }^{r}
\| D   u_{t}^{(j)}\|^{\f{2(4r +2(s+1)-1)}{4j  +1}}_{L^{2}}
+C\sum_{j>0 }^{r}\| D ^{s}u^{(r-j )}_{t}\|^{\f{2(4r +2(s+1)-1)}{ 4(r-j)+2 s -1}}_{L^{2}}\\&+\f\eta4\sum_{j>0 }^{r}\| D ^{s}u^{(r-j )}_{t}\|^{\f{8  j}{ 4(r-j)+2 s -1}}_{L^{2}}\| D ^{s+1} u^{(r-j )}_{t}\|_{L^{2}}^{2},
\ea\ee
where $\eta$ will be determined later.\\
Thus, plugging \eqref{4.12v2} and \eqref{4.13v2} into   \eqref{4.11v2}, one infers that
\be\ba\label{4.14v2}
 \B|I\B|
 \leq&C\sum_{j }^{r-1} \| D ^{ s \alpha}u_{t}^{(j)}\|_{L^{2}}
^{\f{2(4r +2(s+1)-1)}
{4j +2s-1}}+C\| D ^{ s  }u_{t}^{(r)}\|_{L^{2}}
^{\f{2(4r +2(s+1)-1)}
{4r +2s-1}} +C\sum_{j }^{r}
\| D  u^{(r-j)}_{t} \|^{\f{2(4r +2(s+1)-1)}{4(r-j) +1}}_{L^{2}}\\&+ \eta\sum_{j>0 }^{r}\| D ^{s }u^{(r-j )}_{t}\|^{\f{8  j}{ 4(r-j) +2 s -1}}_{L^{2}}\| D ^{(s+1) } u^{(r-j )}_{t}\|_{L^{2}}^{2}+\varepsilon\| D ^{s+1}u _{t}^{(r)}\|_{L^{2}}^{2}.
\ea\ee
Along the same line of above derivation, one obtains
\be\ba\label{4.15v2}
& \B|II\B|+\B|III\B|+\B|IV\B|\\
\leq &  C\sum_{j }^{r}\| D ^{ s }B_{t}^{(j)}\|_{L^{2}}
^{\f{2(4r +2(s+1)-1)}{4j +2s -1}}+C\sum_{j }^{r}
\| D  u^{(r-j)}_{t} \|^{\f{2(4r +2(s+1)-1)}{4(r-j)+2 -1}}_{L^{2}} +C
\sum_{j }^{r}\| D ^{s}u^{(r-j)}_{t} \|^{\f{2(4r +2(s+1)-1)}{4(r-j)+2 s -1}}_{L^{2}}
 \\&+C\sum_{j>0 }^{r}
\| D   B_{t}^{(j)}\|^{\f{2(4r +2(s+1)-1)}{4j  +1}}_{L^{2}}
  + \eta \sum_{j>0 }^{r}\| D ^{s}u^{(r-j )}_{t}\|^{\f{8  j}{ 4(r-j)+2 s -1}}_{L^{2}}\| D ^{s+1} u^{(r-j )}_{t}\|_{L^{2}}^{2}\\&+ \eta \sum_{j>0 }^{r}\| D ^{s}B^{(r-j )}_{t}\|^{\f{8  j}{ 4(r-j)+2 s -1}}_{L^{2}}\| D ^{s+1} B^{(r-j )}_{t}\|_{L^{2}}^{2}+\f\varepsilon4\| D ^{s+1}B _{t}^{(r)}\|_{L^{2}}^{2} +\f\varepsilon4\| D ^{s+1}u _{t}^{(r)}\|_{L^{2}}^{2}.
   \ea\ee
Inserting \eqref{4.14v2}-\eqref{4.15v2} into \eqref{4.2} and choosing $\varepsilon$ sufficiently small, we observe that
\be\ba\label{4.16v1}
&\f{1}{2}\f{d}{dt} (\|  D ^{s}u _{t}^{(r)}\|^{2}_{L^{2}}+\|  D ^{s}B _{t}^{(r)}\|^{2}_{L^{2}})+\f{1}{2}\|  D ^{s+1}u _{t}^{(r)}\|^{2}_{L^{2}}+\f{1}{2}\|  D ^{s+1}B _{t}^{(r)}\|^{2}_{L^{2}}\\
\leq &  C\sum_{j }^{r-1}(\| D ^{ s }B_{t}^{(j)}\|_{L^{2}}^{2}+\| D ^{ s }u_{t}^{(j)}\|_{L^{2}}^{2})
^{\f{ 4r +2(s+1)-1 }{4j +2s -1}}+  C(\| D ^{ s  }u_{t}^{(r)}\|_{L^{2}}+\| D ^{ s  }B_{t}^{(r)}\|_{L^{2}})
^{\f{ (4r +2(s+1)-1)}
{4r +2s-1}}
 \\&+C\sum_{j=0 }^{r}
(\| D   u_{t}^{(j)}\|_{L^{2}}^{2}+\| D   B_{t}^{(j)}\|_{L^{2}}^{2})^{\f{ 4r +2(s+1)-1 )}{4j  +1}}
 \\&+2\eta\sum_{j>0 }^{r}(\| D ^{s}u^{(r-j )}_{t}\|_{L^{2}}^{2}+\| D ^{s}B^{(r-j )}_{t}\|_{L^{2}}^{2})^{\f{4  j}{ 4(r-j)+2 s -1}}(\| D ^{s+1} u^{(r-j )}_{t}\|_{L^{2}}^{2}+\| D ^{s+1} B^{(r-j )}_{t}\|_{L^{2}}^{2} ),
\ea\ee
which turns out  \eqref{4.13v3}.\\
(2)
It follows from
\eqref{4.2}-\eqref{4.8} with $r=0$ that
\be\ba\label{4.13v1}
&\f{1}{2}\f{d}{dt} (\|  D ^{s}u \|^{2}_{L^{2}}+\|  D ^{s}B \|^{2}_{L^{2}})+(\| D ^{s+1}u \|^{2}_{L^{2}}+\|  D ^{s+1}B \|^{2}_{L^{2}})\\
=&- \int_{\mathbb{R}^{3}}[ D ^{s}(u \cdot\nabla u )-u \cdot\nabla  D ^{s} u]\cdot D ^{s}u dx+ \int_{\mathbb{R}^{3}}[ D ^{s}(B\cdot\nabla B )-B\cdot\nabla  D ^{s}B ]\cdot D ^{s}u dx\\&- \int_{\mathbb{R}^{3}}[ D ^{s}(u \cdot\nabla B )-u \cdot\nabla  D ^{s}B]\cdot D ^{s}B dx+ \int_{\mathbb{R}^{3}}[ D ^{s}(B\cdot\nabla u )-B\cdot\nabla  D ^{s}u]\cdot D ^{s}B dx.
\ea
\ee
A slight modification the proof of \eqref{4.7}, we have
\be\ba\label{4.19}
   & \B|\int_{\mathbb{R}^{3}}[ D ^{s}(u \cdot\nabla u )- u \cdot\nabla D ^{s} u  ] \cdot D ^{s}u dx\B|\\
\leq& C\| D ^{s} u\|_{L^{2}}\|\nabla u\|_{L^{3}}\| D ^{s}u\|_{L^{6}}\\\leq& C\| D ^{s} u\|_{L^{2}} \| D  u \|^{\f{2(s+1) -5 }{2(s-1)}}_{L^{2}}\| D ^{s} u \|_{L^{2}}^{\f{1}{2(s-1)}}\|  D ^{s}u\|_{L^{6}}\\
\leq& C \| D  u \|^{\f{2(s+1) -5 }{ (s-1)}}_{L^{2}}\| D ^{s} u \|_{L^{2}}^{2+\f{1}{ (s-1)}}   +\varepsilon\| D ^{s+1}u \|^{2}_{L^{2}}\\\leq& C\| D ^{ s }u \|_{L^{2}}
^{\f{2( 2(s+1)-1)}
{ 2s -1}} +C
\| D  u \|^{ {2( 2(s+1)-1)} }_{L^{2}} +\varepsilon\| D ^{s+1}u \|^{2}_{L^{2}}.
\ea
\ee
Likewise,
\be\ba
&\B|- \int_{\mathbb{R}^{3}}[ D ^{s}(u \cdot\nabla B )-u \cdot\nabla D ^{s}B]\cdot D ^{s}B dx\B|\\
\leq& C\| D ^{s} u\|_{L^{2}}\|\nabla B\|_{L^{3}}\|  D ^{s}B\|_{L^{6}}+\| \nabla u\|_{L^{3}}\| D ^{s} B\|_{L^{2}}\|  D ^{s}B\|_{L^{6}}\\\leq &C\| D ^{s} u\|_{L^{2}} \| D  B \|^{\f{2(s+1) -5 }{2(s-1)}}_{L^{2}}\| D ^{s} B \|_{L^{2}}^{\f{1}{2(s-1)}}\|  D ^{s+1}B\|_{L^{2}}\\&+C\| D ^{s} B\|_{L^{2}} \| D  u \|^{\f{2(s+1) -5 }{2(s-1)}}_{L^{2}}\| D ^{s} u \|_{L^{2}}^{\f{1}{2(s-1)}}\|  D ^{s+1}B\|_{L^{2}}\\ \leq& C(\| D ^{ s }u \|^{2}_{L^{2}}+\| D ^{ s }B \|^{2}_{L^{2}})
^{\f{ ( 2(s+1)-1)}
{ 2s -1}} +C
(\| D  u \|_{L^{2}} ^{2}+\| D  B \|_{L^{2}} ^{2})^{\f{ ( 2(s+1)-1)}{  2 -1}} \\&+\varepsilon(\| D ^{s+1}u \|^{2}_{L^{2}}+\| D ^{s+1}B \|^{2}_{L^{2}}),
\ea
\ee
and
\be\ba\label{4.21v2} &\B|\int_{\mathbb{R}^{3}}[ D ^{s}(B\cdot\nabla B )-B\cdot\nabla D ^{s}B ]\cdot D ^{s}u dx+ \int_{\mathbb{R}^{3}}[ D ^{s}(B\cdot\nabla u )-B\cdot\nabla D ^{s}u]\cdot D ^{s}B dx\B|
\\\leq& C(\| D ^{ s }u \|^{2}_{L^{2}}+\| D ^{ s }B \|^{2}_{L^{2}})
^{\f{ ( 2(s+1)-1)}
{ 2s -1}} +C
(\| D  u \|_{L^{2}} ^{2}+\| D  B \|_{L^{2}} ^{2})^{  2(s+1)-1  }\\& +\varepsilon(\| D ^{s+1}u \|^{2}_{L^{2}}+\| D ^{s+1}B \|^{2}_{L^{2}}).
\ea
\ee
Inserting   \eqref{4.19}-\eqref{4.21v2} into \eqref{4.13v1}, we infer that
\be\ba\label{4.15}
& \f{1}{2}\f{d}{dt} (\|  D ^{s}u \|^{2}_{L^{2}}+\|  D ^{s}B \|^{2}_{L^{2}})+\f12\|  D ^{s+1}u \|^{2}_{L^{2}}+\f12\|  D ^{s+1}B \|^{2}_{L^{2}}\\
 \leq & C(\| D ^{ s }u \|_{L^{2}}^{2}+\| D ^{ s }B \|_{L^{2}}^{2})
^{\f{   2(s+1)-1 }
{ 2s -1}} +C(
\| D  u \|_{L^{2}}^{2}+
\| D  B \|_{L^{2}}^{2})^{   2(s+1)-1  },
\ea
\ee
where we have taken
  $\varepsilon$ sufficiently small.\\
This completes the proof of this lemma.
\end{proof}
It is a position   to show   Theorem \ref{the1.1}.
\begin{proof}[Proof of Theorem \ref{the1.1}]
To proceed further, we set
\be\ba\label{4.23}
&F_{p,q}= \| D ^{q }u_{t}^{(p)}\|_{L^{2}}^{2}+\| D ^{q }B_{t}^{(p)}\|_{L^{2}}^{2}+\sum_{\ell=0}^{p-1}
\f{4\ell +2q -1}{4p +2q -1}F_{\ell,q}^{}+\sum_{\ell=0}^{p }
\f{4\ell +1}{4p +2q -1}F_{\ell,1}^{\f{4p +2q -1}{4\ell +1}},\\
&F_{0,q}= \| D ^{q }u \|_{L^{2}}^{2}+\| D ^{q }B \|_{L^{2}}^{2}+\f{1}{ 2q -1}
F_{0,1}^{  2q -1 },\\
&
G_{p,q}= \f12\| D ^{(q+1)}u_{t}^{(p)}\|_{L^{2}}^{2}+\f12\| D ^{(q+1)}B_{t}^{(p)}\|_{L^{2}}^{2}+\f12\sum_{\ell=0}^{p-1}
F_{\ell,q}^{\f{4(p-\ell) }{4\ell +2q -1}}G_{\ell,q},\\
&G_{0,q}= \f12\| D ^{(q+1)}u \|_{L^{2}}^{2}+\f12\| D ^{(q+1)}B \|_{L^{2}}^{2}+\f12
F_{0,1}^{ 2q -2  }G_{0,1},\\
&G_{p,0}= \f12\| D  u_{t}^{(p)}\|_{L^{2}}^{2}+ \f12\| D  B_{t}^{(p)}\|_{L^{2}}^{2}+\f12\sum_{\ell=0}^{p-1}
F_{\ell,1}^{\f{4(p-\ell) -2}{4\ell  +1}}G_{\ell,1}.
\ea\ee
In particular, we have
\be\ba\label{4.16}
&F_{0,1}= \| D  u \|_{L^{2}}^{2}+\| D  B \|_{L^{2}}^{2}, \\&F_{1,0} =(\| u_{t}  \|_{L^{2}}^{2}+\| B_{t}  \|_{L^{2}}^{2})+\f{1}{ 3  }F_{0,1}^{3},\\
&F_{0,2}= \| D ^{2}u \|_{L^{2}}^{2}+\| D ^{2}B \|_{L^{2}}^{2}+\f{1}{3 } F_{0,1}^{3},\\
&F_{1,1}=  \| D  u_{t} \|_{L^{2}}^{2}+\| D  B_{t} \|_{L^{2}}^{2}+\f{ 1}{5}
F_{0,1}^{5}, \\
&F_{2,0}= \| u_{t}^{(2)}\|_{L^{2}}^{2}+\| B_{t}^{(2)}\|_{L^{2}}^{2}+
\f{3 }{ 7}F_{1,0}^{\f{ 7}{3 }}+\f{1}{ 7}
F_{0,1}^{7}+\f{ 5}{ 7}F_{1,1}^{\f{ 7}{5}},\\
&G_{0,1}= \f12\| D ^{2}u \|_{L^{2}}^{2}+ \f12\| D ^{2}B \|_{L^{2}}^{2}, \\
&G_{1,1}= \f12\| D ^{ 2 }u_{t} \|_{L^{2}}^{2}+\f12\| D ^{ 2 }B_{t} \|_{L^{2}}^{2}+\f12
F_{0,1}^{4}G_{0,1},\\
&G_{0,2}= \f12\| D ^{3 }u \|_{L^{2}}^{2} +\f12\| D ^{3 }B \|_{L^{2}}^{2} +\f12F^{2}_{0,1}G_{0,1},\\
&G_{2,0}= \f12\| D   u_{t}^{(2)}\|_{L^{2}}^{2}+ \f12\| D   B_{t}^{(2)}\|_{L^{2}}^{2}+\f12
F_{0,1}^{6}G_{0,1}+\f12F_{1,1}^{\f{2  }{ 5}}G_{1,1}.
\ea\ee
The introduction of above notations is motivated by energy inequality \eqref{4.13v3}.

We will apply inductive method to confirm that
 \be\label{4.13}
\left\{\ba
&F_{p,q}\in L^{\f{1}{4p +2q -1}},\\
&\f{d}{dt}F_{p,q}+G_{p,q}\leq CF^{\f{4p +2q +1}{4p +2q -1}}_{p,q},\\
&G_{p,q}\in L^{\f{1}{4p +2q +1}}, \| D ^{( q+1)}u^{(p ) }_{t}\|_{L^{2}}^{2}, \| D ^{( q+1)}B^{(p ) }_{t}\|_{L^{2}}^{2}\in L^{\f{1}{4p +2q +1}},\ea\right.\ee for $ k= p+q$, where $p, q\in \mathbb{N}$. \\
It is worth pointing out that we have proved    \eqref{4.13} for $p+q=1$ in last section.
Indeed,   from \eqref{3.42}-\eqref{3.48}, we infer
that
\be\label{4.21}
\left\{\ba
&F_{0,1} \in L^1,F_{ 1,0} \in   L^{\f{ 1 }{ 3 }},\\
&\f{d}{dt}F_{0,1}+G_{0,1}\leq CF_{0,1}^{3},\f{d}{dt}F_{ 1,0}+G_{1,0}\leq CF_{ 1,0}^{\f{ 5 }{3 }},\\
&G_{0,1}\in L^{\f{1}{  3  }},  \|  D ^{2}u   \|_{L^{2}}^{2},  \|  D ^{2}B   \|_{L^{2}}^{2}\in L^{\f{1}{  3  }},G_{1,0}\in L^{\f{1}{ 5}},\| D  u  _{t}\|_{L^{2}}^{2},\| D  B  _{t}\|_{L^{2}}^{2}\in L^{\f{1}{5}}.\ea\right.\ee
This yields \eqref{4.13} for $p+q=1$. Next we shall address the case $p+q=2$.  We will present the derivation of $F_{0,2}, F_{1,1}, F_{2,0}$ in turn.\\
 It follows  from
\eqref{4.16}-\eqref{4.21} that
\be\label{4.26v2}
F_{0,2} \in L^{\f{ 1}{3 }}.
\ee
\\
With the help of \eqref{4.21}, we know that
\be\ba\label{4.21v1}
\f{d}{dt}F_{0,1}+G_{0,1}\leq CF_{0,1}^{3}.
\ea
\ee
A routine computation gives rise to
\be\ba\label{4.24}
\f13 \f{d}{dt}F_{0,1}^{3}+F^{2}_{0,1}G_{0,1}\leq CF_{0,1}^{5}.
\ea
\ee
By the definition of $F_{0,1}$ in \eqref{4.16}, we derive from \eqref{4.15v1} and \eqref{4.24}  that
$$\ba
&\f{1}{2}\f{d}{dt}\B[\|  D ^{2}u \|^{2}_{L^{2}}+\|  D ^{2}B \|^{2}_{L^{2}}+\f13 F_{0,1}^{3}\B]+\f12\|  D ^{3 }u \|^{2}_{L^{2}}+\f12\| D ^{3 }B \|^{2}_{L^{2}}+F^{2}_{0,1}G_{0,1} \\
\leq& C\| D ^{ 2}u \|_{L^{2}}
^{\f{10}
{3 }} +C
\| D  u \|^{10}_{L^{2}}
+CF_{0,1}^{5}\\
\leq&C \| D ^{ 2}u \|_{L^{2}}
^{\f{10}
{3 }}  +CF_{0,1}^{5}.
\ea
$$
We conclude by the latter  inequality that
\be\ba
\f{d}{dt}F_{0,2}+G_{0,2}\leq CF_{0,2}^{\f{  5 }{3 }}.
\ea
\ee
Using Lemma \ref{hofflemma}, $F_{0,2} \in L^{\f{ 1}{3 }}$ in \eqref{4.26v2} and  the latter inequality, one has
$$
G_{0,2}\in L^{\f{ 1}{5}},
$$
which means that
\be\|  D ^{3 }u \|^{2}_{L^{2}},\|  D ^{3 }B \|^{2}_{L^{2}}\in L^{\f{ 1}{5}}.
\ee
We drive from \eqref{4.21} that
 $\| D  u  _{t}\|_{L^{2}}^{2}\in L^{\f{1}{5}}$ and $F_{0,1} \in L^1$, which leads to
 \be\label{4.29}
 F_{1,1} \in L^{\f{1}{5}}.
 \ee
Following the same path of \eqref{4.24}, one gets
\be\ba\label{4.26}
\f{1}{  5 }\f{d}{dt}F_{0,1}^{5}+F^{4}_{0,1}G_{0,1}\leq CF_{0,1}^{7}.
\ea
\ee
Combining  \eqref{3.24v1} and \eqref{4.26}, we infer that
\be\ba
&\f{d}{dt}\B[\| D   u _{t}\|^{2}_{L^{2}}+\| D  B _{t}\|^{2}_{L^{2}}+\f{1}{5}F_{0,1}^{5}\B]+\f12\| D ^{2 }u _{t}\|_{L^{2}}^{2}+\f12\| D ^{2 }B _{t}\|_{L^{2}}^{2}+F^{4}_{0,1}G_{0,1}\\
\leq & C(\| D  u_{t}\|_{L^{2}}^{2}+\| D  B_{t}\|_{L^{2}}^{2})^{\f{7}{5}}
+C(\| D  u \|_{L^{2}}^{2}+\| D  B \|_{L^{2}}^{2})^{ 7 }
\\&+16\eta(\| D ^{2}u (t)\|_{L^{2}}^{2}+ \| D ^{2}B (t)\|_{L^{2}}^{2})(\| D  B \|_{L^{2}}^{2}+\| D  u\|_{L^{2}}^{2})^{4}+C F_{0,1}^{  7 }\\
\leq & C(\| D  u_{t}\|^{2}_{L^{2}}+\| D  B_{t}\|^{2}_{L^{2}})^{\f{7}{5}}
 +16\eta   F^{4}_{0,1}G_{0,1}+CF_{0,1}^{ 7 }.
\ea
\ee
where we have used
$$F_{0,1}^{7}\leq CF_{1,1}^{\f{ 7}{5}}.$$
Thus, we conclude by the arbitrariness of  $\eta$ that
\be\ba
&\f{d}{dt}\B[\| D   u _{t}\|^{2}_{L^{2}}+\| D   B _{t}\|^{2}_{L^{2}}+\f{1}{5}F_{0,1}^{5}\B]+\f12\| D ^{2 }u _{t}\|_{L^{2}}^{2}+\f12\| D ^{2 }B _{t}\|_{L^{2}}^{2}+\f12F^{ 4  }_{0,1}G_{0,1} \\
\leq & C(\| D  u_{t}\|^{2}_{L^{2}}+\| D  B_{t}\|^{2}_{L^{2}})^{\f{7}{5}}
 +CF_{1,1}^{\f{ 7}{5}},
\ea
\ee
which means
\be\label{4.33}
\f{d}{dt}F_{1,1}+G_{1,1}\leq
CF_{1,1}^{\f{ 7}{5}}.
\ee
This together with Lemma \ref{hofflemma} implies that
\be
G_{1,1}\in L^{\f{1}{7}}, \| D ^{2 }u _{t}\|_{L^{2}}^{2}\in L^{\f{1}{7}}.
\ee
where \eqref{4.29} was used.\\
We deduce from
\eqref{3.31v1} with $r=2$ that
\be\ba\label{4.38v1}
&\f{d}{dt}(\| u _{t}^{(2)}\|^{2}_{L^{2}}+\| B _{t}^{(2)}\|^{2}_{L^{2}})+\f{1}{2}\| D  u _{t}^{(2)}\|_{L^{2}}^{2}+\f{1}{2}\| D  B _{t}^{(2)}\|_{L^{2}}^{2}  \\
\leq&C(\|u_{t} \|_{L^{2}}^{2}+\|B_{t} \|_{L^{2}}^{2})
^{\f{ 9}{3 }}
+(\|u_{t}^{(2)} \|_{L^{2}}^{2}+\|u_{t}^{(2)} \|_{L^{2}}^{2})
^{\f{ 9}7}
\\&+C(\| D  u_{t}  \|_{L^{2}}^{2}+\| D  B_{t}  \|_{L^{2}}^{2})^{\f{ 9}{5}}+C(\| D  u  \|^{2}_{L^{2}}+\| D  B  \|^{2}_{L^{2}})^{ 9}\\&+16\eta(\| D  u_{t} \|_{L^{2}}^{2}+\| D  B_{t} \|_{L^{2}}^{2})^{\f{2  }{5}}
(\| D ^{2}u_{t} \|^{2}_{L^{2}}+\| D ^{2}B_{t} \|^{2}_{L^{2}})\\&+16\eta(\| D  u  \|_{L^{2}}^{2}+\| D  B  \|_{L^{2}}^{2})^{6}
(\| D ^{2}u   \|^{2}_{L^{2}}+\| D ^{2}B   \|^{2}_{L^{2}}).
\ea
\ee
 We deduce from
\eqref{4.13} with $k=1$ that
$$\ba
\f17\f{d}{dt}F_{0,1}^{7}+F_{0,1}^{6}G_{0,1}\leq CF_{0,1}^{  9 },\\
\f57\f{d}{dt}F_{1,1}^{\f{ 7}{5}}+F_{1,1}^{\f{2  }{ 5}}G_{1,1}\leq
CF_{1,1}^{\f{ 9}{5}},\\
\f37\f{d}{dt}F_{1,0}^{\f{ 7}{3 }}+F_{ 1,0}^{\f{ 4  }{3 }}G_{1,0}\leq CF_{ 1,0}^{3}.\ea$$
Combining this and
\eqref{4.38v1}, we conclude by the definition   \eqref{4.16} that
\be\ba
&\f{d}{dt}\B[\| u _{t}^{(2)}\|^{2}_{L^{2}}+\| B _{t}^{(2)}\|^{2}_{L^{2}}+\f37F_{1,0}^{\f{ 7}{3 }}+
\f17F_{0,1}^{7}+\f57F_{1,1}^{\f{ 7}{5}}\B]\\&+\f{1}{2}\| D  u _{t}^{(2)}\|_{L^{2}}^{2}+\f{1}{2}\| D  B _{t}^{(2)}\|_{L^{2}}^{2} +F_{0,1}^{6}G_{0,1}+F_{1,1}^{\f{2  }{ 5}}G_{1,1}\\
 \leq&C(\|u_{t} \|^{2}_{L^{2}}+\|B_{t} \|_{L^{2}}^{2})
^{\f{ 9}{3 }}
 +C(\|u_{t}^{(2)} \|_{L^{2}}^{2}+\|u_{t}^{(2)} \|_{L^{2}}^{2})
^{\f{ 9}7}
_{L^{2}}\\&+C(\| D  u_{t}  \|_{L^{2}}^{2}+\| D  B_{t}  \|_{L^{2}}^{2})^{\f{ 9}{5}}+C(\| D  u  \|^{2}_{L^{2}}+\| D  B  \|^{2}_{L^{2}})^{ 9}\\&+16\eta(\| D  u_{t} \|_{L^{2}}^{2}+\| D  B_{t} \|_{L^{2}}^{2})^{\f{2  }{5}}
(\| D ^{2}u_{t} \|^{2}_{L^{2}}+\| D ^{2}B_{t} \|^{2}_{L^{2}})\\&+16\eta(\| D  u  \|_{L^{2}}^{2}+\| D  B  \|_{L^{2}}^{2})^{6}
(\| D ^{2}u   \|^{2}_{L^{2}}+\| D ^{2}B   \|^{2}_{L^{2}})+ CF_{0,1}^{  9 }+CF_{1,1}^{\f{ 9}{5}}+CF_{ 1,0}^{3}\\
 \leq&C (\|u_{t}^{(2)} \|_{L^{2}}^{2}+\|u_{t}^{(2)} \|_{L^{2}}^{2})
^{\f{ 9}7}
_{L^{2}}+CF_{0,1}^{  9 }+CF_{1,1}^{\f{ 9}{5}}+CF_{ 1,0}^{3}\\&  +16\eta(\| D  u_{t} \|_{L^{2}}^{2}+\| D  B_{t} \|_{L^{2}}^{2})^{\f{2  }{5}}
(\| D ^{2}u_{t} \|^{2}_{L^{2}}+\| D ^{2}B_{t} \|^{2}_{L^{2}})\\&+16\eta(\| D  u  \|_{L^{2}}^{2}+\| D  B  \|_{L^{2}}^{2})^{6}
(\| D ^{2}u   \|^{2}_{L^{2}}+\| D ^{2}B   \|^{2}_{L^{2}}).
\ea
\ee
As a consequence, we derive for any sufficiently
small $\eta>0$ that
\be\label{4.38v3}
\f{d}{dt}F_{2,0}+G_{2,0}\leq CF_{2,0}^{\f{ 9}{7}}.
\ee
With the help of Theorem \ref{the1.2}, we know that $\| u _{t}^{(2)}\|^{2}_{L^{2}}\in  L^{\f{1}{ 7}}.$ Moreover, we deduce from
\eqref{4.21} and \eqref{4.29} that
\be
F_{0,1} \in L^1,F_{ 1,0} \in   L^{\f{ 1 }{ 3 }}, F_{1,1} \in L^{\f{1}{5}},
\ee
which implies  that
\be
F_{2,0}\in L^{\f{1}{ 7}}.
\ee
As a result, we employ  Lemma  \ref{hofflemma} to estimate \eqref{4.38v3} to get
$$
\| D  u _{t}^{(2)}\|_{L^{2}}^{2}, \| D  B_{t}^{(2)}\|_{L^{2}}^{2}\in L^{\f{1}{ 9}}.
$$
Hence, we have shown  \eqref{4.13} for $p+q=1$ and $p+q=2$.
Now, we assume that \eqref{4.13} is valid for \be\label{4.41v1} p+q\leq  k-2= m+n -2.\ee
It suffices  to prove \eqref{4.13} for $p+q=k-1= m+n -1$.
To this end,
we handle the  quantity   $F_{i, m+n-i-1}$ which is divided into   four cases:  $$\ba
&i=0, \\&1\leq i\leq m+n -3,\\&i= m+n -2,\\ &i= m+n -1.
\ea$$

$Case ~1$. It follows from \eqref{4.23} that
\be\ba\label{4.41}
&F_{0, m+n -1}=\| D ^{(m+n -1)}u \|_{L^{2}}^{2}+\| D ^{(m+n -1)}B\|_{L^{2}}^{2}+\f{1}{ 2(m+n -1)-1}
F_{0,1}^{ 2(m+n -1)-1 },\\
&G_{0,m+n -1}=\f12 \| D ^{(m+n  )}u \|_{L^{2}}^{2}+\f12 \| D ^{(m+n  )}B \|_{L^{2}}^{2}+\f12
F_{0,1}^{ 2(m+n -1)-2  }G_{0,1}.
\ea\ee
Thanks to induction hypotheses \eqref{4.41v1} and $\eqref{4.13}_{3}$, we get $\| D ^{(m+n -1)}u \|_{L^{2}}^{2}\in  L^{\f{1}{ 2(m+n -1)-1}}.$ Hence, making use of  $F_{0,1} \in L^1 $, we obtain $F_{0, m+n -1}\in  L^{\f{1}{ 2(m+n -1)-1}}.$
In the light of estimate \eqref{4.21v1}, we arrive at
$$\ba \f{d}{dt}F_{0,1}+G_{0,1}\leq CF_{0,1}^{3}.
\ea$$
Further more, we compute
\be\ba\label{4.38} \f{d}{dt}F_{0,1}^{  2(m+n -1)-1 }+F_{0,1}^{ 2(m+n -1)-2 }G_{0,1}\leq CF_{0,1}^{    2(m+n -1)+1  }.
\ea\ee
It follows from the   relationship \eqref{4.15v1} and \eqref{4.38} that
$$\ba
& \f{d}{dt} \B[\|  D ^{(m+n -1)}u \|^{2}_{L^{2}}+\|  D ^{(m+n -1)}B \|^{2}_{L^{2}}+\f{1}{ 2(m+n -1)-1}F_{0,1}^{  2(m+n -1)-1 }\B]\\&+\f12\| D ^{ (m+n )  }u \|^{2}_{L^{2}}+\f12\|  D ^{ (m+n )  }B \|^{2}_{L^{2}}+F_{0,1}^{ 2(m+n -1)-2  }G_{0,1} \\
\leq &C(\| D ^{ (m+n -1)  }u \|^{2}_{L^{2}}+\| D ^{ (m+n -1)  }B\|^{2}_{L^{2}})
^{\f{  2(m+n)-1 }
{ 2(m+n -1)-1}} +C
F_{0,1}^{   2(m+n)-1  }.
\ea
$$
According  to the definition of $F_{0, m+n -1}$ in \eqref{4.41}, we discover that
$$
F_{0,1}^{   2(m+n)-1  }\leq
CF_{0,m+n-1}^{\f{  2(m+n)-1 }{  2(m+n -1)-1}}.
$$
As a consequence, we observe that
$$\ba \f{d}{dt}F_{0,m+n-1}+G_{0,m+n-1}\leq F_{0,m+n-1}^{\f{  2(m+n)-1 }
{ 2(m+n -1)-1}}.
\ea$$
We derive from Lemma \ref{hofflemma} and $F_{0, m+n -1}\in  L^{\f{1}{ 2(m+n -1)-1}} $ that
$$G_{0,m+n-1}\in L^{\f{1}{ 2(m+n)-1}}, \|  D ^{ (m+n ) }u \|^{2}_{L^{2}}, \| D ^{ (m+n ) }B \|^{2}_{L^{2}}\in L^{\f{1}{ 2(m+n)-1}}.$$
The first case is asserted.

$Case$ 2.  $1\leq i<m+n-2$.\\
We deduce from \eqref{4.23} that
\be\ba\label{4.46v3}
F_{i,m+n-i-1}= & \| D ^{(m+n-i-1)}u_{t}^{(i)}\|_{L^{2}}^{2}+ \| D ^{(m+n-i-1)}B_{t}^{(i)}\|_{L^{2}}^{2}\\&+ \sum_{\ell=0}^{i-1}\f{4\ell +2(m+n-i-1)-1}{4i +2(m+n-i-1)-1}
F_{\ell,m+n-i-1}^{\f{4i +2(m+n-i-1)-1}{4\ell +2(m+n-i-1)-1}}\\&+\sum_{\ell=0}^{i }\f{4\ell +1}{4i +2(m+n-i-1)-1}
F_{\ell,1}^{\f{4i +2(m+n-i-1)-1}{4\ell +1}},\\
G_{i,m+n-i-1}= &\f12\| D ^{(m+n-i )}u_{t}^{(i)}\|_{L^{2}}^{2}+\f12\| D ^{(m+n-i )}B_{t}^{(i)}\|_{L^{2}}^{2}\\&+\f12\sum_{\ell=0}^{i-1}
F_{\ell,m+n-i-1}^{\f{4(i-\ell) }{4\ell +2(m+n-i-1)-1}}G_{\ell,m+n-i-1}.
\ea\ee
In view of
inductive hypothesis \eqref{4.41v1} and $\eqref{4.13}_{3}$,  we get $\| D ^{(m+n-i-1)}u_{t}^{(i)}\|_{L^{2}}^{2},
\| D ^{(m+n-i-1)}B_{t}^{(i)}\|_{L^{2}}^{2}\in L^{\f{1}{4i +2(m+n-i-1)-1}}$. Moreover, since  $\ell+m+n-i-1\leq m+n -2$ and $\ell+1\leq i+1\leq m+n-2$, by   the induction hypotheses  \eqref{4.41v1},  we infer that $F_{\ell,m+n-i-1}\in L^{\f{1}{4\ell +2(m+n-i-1)-1}}$     and   $F_{\ell,1}\in L^{\f{1}{4\ell  +1}}$. Summarily, one has
\be\label{4.48}
F_{i,m+n-i-1}\in L^{\f{1}{4i +2(m+n-i-1)-1}}.
\ee
According to   the induction hypotheses \eqref{4.41v1}, we arrive at
\be\ba
&\f{d}{dt}F_{\ell,m+n-i-1}+G_{\ell,m+n-i-1}\leq F_{\ell,m+n-i-1}^{\f{4\ell +2(m+n-i-1)+1}{4\ell +2(m+n-i-1)-1}},\ell \leq i-1,\\
&\f{d}{dt}F_{\ell,1}+G_{\ell,1}\leq F_{\ell,1}^{\f{4\ell + 3 }{4\ell +1}},\ell \leq i,
\ea\ee
which turns out that
\be\ba\label{4.44}
 &c_{\ell,m+n-i-1}\f{d}{dt}F_{\ell,m+n-i-1}^{\f{4i +2(m+n-i-1)-1}{4\ell +2(m+n-i-1)-1}}+
F_{\ell,m+n-i-1}^{\f{4(i-\ell) }{4\ell +2(m+n-i-1)-1}}G_{\ell,m+n-i-1}\leq C F_{\ell,m+n-i-1}^{\f{4i +2(m+n-i-1)+1}{4\ell +2(m+n-i-1)-1}},\\
& c_{\ell,1}\f{d}{dt}
F_{\ell,1}^{\f{4i +2(m+n-i-1)-1}{4\ell +1}}
+
F_{\ell,1}^{\f{4(i-\ell)+2(m+n-i-1)-2 }{4\ell +1}}G_{\ell,1}\leq C F_{\ell,1}^{\f{4i +2(m+n-i-1)+1}{4\ell +1}},
\ea\ee
where
$$\ba
c_{\ell,m+n-i-1}=\f{4\ell +2(m+n-i-1)-1}{4i +2(m+n-i-1)-1},\\
c_{\ell,1}=\f{4\ell +1}{4i +2(m+n-i-1)-1}.
\ea$$
From \eqref{4.13v3} and
\eqref{4.13}, we conclude by the definition of  $F_{j,1}, F_{j,m+n-i-1}, G_{j,m+n-i-1}$ in \eqref{4.41}  that
\be\ba\label{4.45}
&\f{d}{dt}\B[ \|  D ^{(m+n-i-1)}u _{t}^{(i)}\|^{2}_{L^{2}}+\| D ^{(m+n-i-1)}B _{t}^{(i)}\|^{2}_{L^{2}}\B]\\&+\f{1}{2}\|  D ^{(m+n-i )}u _{t}^{(r)}\|^{2}_{L^{2}}+\f{1}{2}\|  D ^{(m+n-i )}B _{t}^{(r)}\|^{2}_{L^{2}}\\\leq&C\sum_{j }^{i-1} (\| D ^{ (m+n-i-1) \alpha}u_{t}^{(j)}\|_{L^{2}}^{2}+\| D ^{ (m+n-i-1) \alpha}B_{t}^{(j)}\|_{L^{2}}^{2})
^{\f{ (4i +2(m+n-i)-1)}
{4j +2(m+n-i-1)-1}}\\&+C(\| D ^{ (m+n-i-1) \alpha}u_{t}^{(i)}\|_{L^{2}}^{2}+\| D ^{ (m+n-i-1) \alpha}B_{t}^{(i)}\|_{L^{2}}^{2})
^{\f{ (4i +2(m+n-i)-1)}
{4i +2(m+n-i-1)-1}}\\&+C\sum_{j }^{i}
(\| D  u^{(i-j)}_{t} \|_{L^{2}}^{2}+\| D  B^{(i-j)}_{t} \|_{L^{2}}^{2})^{\f{ (4i +2(m+n-i)-1)}{4(i-j)+2 -1}}\\&+4\eta\sum_{j>0 }^{i}(\| D ^{(m+n-i-1)}u^{(i-j )}_{t}\|_{L^{2}}^{2}+\| D ^{(m+n-i-1)}B^{(i-j )}_{t}\|_{L^{2}}^{2})^{\f{4  j}{ 4(r-j)+2 (m+n-i-1)-1}}\\&\times(\| D ^{(m+n-i)} u^{(i-j )}_{t}\|_{L^{2}}^{2}+\| D ^{(m+n-i)} B^{(i-j )}_{t}\|_{L^{2}}^{2})\\\leq&C\sum_{j }^{i-1} F_{j,m+n-i-1}
^{\f{ (4i +2(m+n-i)-1)}
{4j +2(m+n-i-1)-1}}+C(\| D ^{ (m+n-i-1) \alpha}u_{t}^{(i)}\|_{L^{2}}^{2}+\| D ^{ (m+n-i-1) \alpha}B_{t}^{(i)}\|_{L^{2}}^{2})
^{\f{ (4i +2(m+n-i)-1)}
{4i +2(m+n-i-1)-1}}\\&+C\sum_{j=0 }^{i}
F_{j,1}^{\f{ (4i +2(m+n-i)-1)}{4j +2 -1}} +4\eta\sum_{j=0 }^{i-1}F_{j,m+n-i-1}^{\f{4  (i-j)}{ 4j +2 (m+n-i-1)-1}}G_{j,m+n-i-1}.
\ea\ee
Combining  \eqref{4.44} and    \eqref{4.45}, we get
\be\ba\label{4.46}
&\f{d}{dt} \B[\|  D ^{(m+n-i-1)}u _{t}^{(i)}\|^{2}_{L^{2}}+\| D ^{(m+n-i-1)}B _{t}^{(i)}\|^{2}_{L^{2}}\B]\\&+\f{d}{dt} \B[\sum_{\ell=0}^{i-1 }c_{\ell,m+n-i-1}F_{\ell,m+n-i-1}^{\f{4i +2(m+n-i-1)-1}{4\ell +2(m+n-i-1)-1}}+\sum_{\ell=0}^{i } c_{\ell,1}F_{\ell,1}^{\f{4i +2(m+n-i-1)-1}{4\ell +1}}
\B]\\&+\f{1}{2}\|  D ^{(m+n-i )}u _{t}^{(r)}\|^{2}_{L^{2}}+\f{1}{2}\| D ^{(m+n-i )}B _{t}^{(r)}\|^{2}_{L^{2}}+F_{\ell,m+n-i-1}^{\f{4(i-\ell) }{4\ell +2(m+n-i-1)-1}}G_{\ell,m+n-i-1}  \\\leq&\sum_{j }^{i-1} F_{j,m+n-i-1}
^{\f{ (4i +2(m+n-i)-1)}
{4j +2(m+n-i-1)-1}}+(\| D ^{ (m+n-i-1) \alpha}u_{t}^{(i)}\|_{L^{2}}^{2}+\| D ^{ (m+n-i-1) \alpha}B_{t}^{(i)}\|_{L^{2}}^{2})
^{\f{ (4i +2(m+n-i)-1)}
{4i +2(m+n-i-1)-1}}\\&+\sum_{j=0 }^{i}
F_{j,1}^{\f{ (4i +2(m+n-i)-1)}{4j +2 -1}} +\eta\sum_{j=0 }^{i-1}F_{j,m+n-i-1}^{\f{4  (i-j)}{ 4j +2 (m+n-i-1)-1}}G_{j,m+n-i-1}.
\ea\ee
In view of  of definition \eqref{4.41}, we have
 $$\ba&\sum_{j }^{i-1} F_{j,m+n-i-1}
^{\f{ (4i +2(m+n-i)-1)}
{4j +2(m+n-i-1)-1}}\leq \sum_{j }^{i-1} F_{i,m+n-i-1}^{\f{4i +2(m+n-i)-1}{4i +2(m+n-i-1)-1}}\leq C F_{i,m+n-i-1}^{\f{4i +2(m+n-i)-1}{4i +2(m+n-i-1)-1}},\\&\sum_{j=0 }^{i}
F_{j,1}^{\f{  4i +2(m+n-i)-1 }{4j +2 -1}}
\leq\sum_{j=0 }^{i}
F_{i,m+n-i-1}^{\f{4i +2(m+n-i)-1}{4i +2(m+n-i-1)-1}}\leq C
F_{i,m+n-i-1}^{\f{4i +2(m+n-i)-1}{4i +2(m+n-i-1)-1}}.\ea$$
Plugging this into \eqref{4.46} and choosing $\eta$ small, we remark that
$$
\f{d}{dt}F_{i,m+n-i-1}+G_{i,m+n-i-1}\leq
CF_{i,m+n-i-1}^{\f{4i +2(m+n-i)-1}{4i +2(m+n-i-1)-1}}.
$$
With this in hand,  we invoke
\eqref{4.48} and Lemma \ref{hofflemma} to get
\be
G_{i,m+n-i-1}\in L^{\f{1}{4i +2(m+n-i)-1}}, \|  D ^{(m+n-i )}u _{t}^{(i)}\|^{2}_{L^{2}}, \|  D ^{(m+n-i )}B _{t}^{(i)}\|^{2}_{L^{2}}\in L^{\f{1}{4i +2(m+n-i)-1}}.
\ee
Case 2 is proved.\\

$Case$ 3. $i=m+n-2.$\\
In this case, we set
\be\ba\label{4.49}
&F_{ m+n-2,1}= \| D  u_{t}^{(m+n-2)}\|_{L^{2}}^{2}+ \| D  B_{t}^{(m+n-2)}\|_{L^{2}}^{2}+
\f{4\ell +1}{4(m+n-2)+1}\sum_{\ell=0}^{m+n-3}
F_{\ell,1}^{\f{4(m+n-2)+1}{4\ell +1}}, \\
&G_{m+n-2,  1}= \f12 \| D ^{2}u_{t}^{(m+n-2)}\|_{L^{2}}^{2}+\f12 \| D ^{2}B_{t}^{(m+n-2)}\|_{L^{2}}^{2}+\f12\sum_{\ell=0}^{m+n-3}
F_{\ell,1}^{\f{4(i-\ell) }{4\ell +2 -1}}G_{\ell,1}.
\ea\ee
According to induction hypotheses \eqref{4.41v1} and $\eqref{4.13}_{3}$, we have $\| D  u_{t}^{(m+n-2)}\|_{L^{2}}^{2}, \| D  B_{t}^{(m+n-2)}\|_{L^{2}}^{2}\in L^{\f{1}{4(m+n-2) +1}}$.
In addition,     the induction hypotheses \eqref{4.41v1} implies  that   $F_{\ell,1}\in L^{\f{1}{4\ell  +1}}$, where we have used $\ell+1\leq m+n-2$.
As a consequence, we get
\be\label{4.56}
F_{ m+n-2,1}\in   L^{\f{1}{4(m+n-2) +1}}.
\ee
It follows from
\eqref{4.13v3} that
\be\ba
&\f{d}{dt}\B[ \|  D  u _{t}^{(m+n-2)}\|^{2}_{L^{2}}+\|  D  B _{t}^{(m+n-2)}\|^{2}_{L^{2}}\B]+\f{1}{2}\|  D ^{2}u _{t}^{(m+n-2)}\|^{2}_{L^{2}}+\f{1}{2}\|  D ^{2}B _{t}^{(m+n-2)}\|^{2}_{L^{2}}\\\leq&C\sum_{j }^{m+n-3} (\| D  u_{t}^{(j)}\|_{L^{2}}^{2}+\| D  u_{t}^{(j)}\|_{L^{2}}^{2})
^{\f{ (4(m+n-2)+3 )}
{4j +1}}\\&+C(\| D ^{   \alpha}u_{t}^{(m+n-2)}\|_{L^{2}}^{2}+\| D ^{   \alpha}B_{t}^{(m+n-2)}\|_{L^{2}}^{2})
^{\f{ 4(m+n-2)+3  }
{4(m+n-2)+1}}\\&+C\sum_{j }^{m+n-2}
[\| D  u^{(m+n-2-j)}_{t} \|_{L^{2}}^{2}+\| D  B^{(m+n-2-j)}_{t} \|_{L^{2}}^{2}]^{\f{ (4(m+n-2)+3 )}{4(r-j)+2 -1}}\\&+4\eta\sum_{j>0 }^{m+n-2}
(\| D  u^{(m+n-2-j )}_{t}\|_{L^{2}}^{2}+\| D  B^{(m+n-2-j )}_{t}\|_{L^{2}}^{2})^{\f{4 j}{ 4(m+n-2-j)+6-5}}\\&\times(\| D ^{2} u^{(m+n-2-j )}_{t}\|_{L^{2}}^{2}+\| D ^{2} B^{(m+n-2-j )}_{t}\|_{L^{2}}^{2})\\\leq&C\sum_{j }^{m+n-3} (\| D  u_{t}^{(j)}\|_{L^{2}}^{2}+\| D  B_{t}^{(j)}\|_{L^{2}}^{2})
^{\f{ (4(m+n-2)+3 )}
{4j +1}}\\&+C(\| D ^{   \alpha}u_{t}^{(m+n-2)}\|_{L^{2}}^{2}+\| D ^{   \alpha}B_{t}^{(m+n-2)}\|_{L^{2}}^{2})
^{\f{ (4(m+n-2)+3 )}
{4(m+n-2)+1}} \\&+4\eta\sum_{j>0 }^{m+n-2}
(\| D  u^{(m+n-2-j )}_{t}\|_{L^{2}}^{2}+\| D  B^{(m+n-2-j )}_{t}\|_{L^{2}}^{2})^{\f{4  j}{ 4(m+n-2-j)+1}}\\&\times(\| D ^{2} u^{(m+n-2-j )}_{t}\|_{L^{2}}^{2}+\| D ^{2} B^{(m+n-2-j )}_{t}\|_{L^{2}}^{2})\\\leq&
C\sum_{j }^{m+n-3} F_{j,1}
^{\f{ 4(m+n-2)+3  }
{4j +1}}+C(\| D ^{   \alpha}u_{t}^{(m+n-2)}\|_{L^{2}}^{2}+\| D ^{   \alpha}B_{t}^{(m+n-2)}\|_{L^{2}}^{2})
^{\f{ (4(m+n-2)+3 )}
{4(m+n-2)+1}} \\&+4\eta\sum_{j=0 }^{m+n-3} F_{j,1}^{\f{4  (m+n-2-j ) }{ 4j +1}}G_{j,1}.
\ea\ee
In the light of definition \eqref{4.49}, we arrive at
$$\sum_{j }^{m+n-3}F_{j,1}
^{\f{ 4(m+n-2)+3  }
{4j +1}}\leq CF_{ m+n-2,1}^{\f{ 4(m+n-2)+3  }
{4(m+n-2)+1}},$$
which means that
\be\ba\label{4.52}
&\f{d}{dt} \B[\|  D  u _{t}^{(m+n-2)}\|^{2}_{L^{2}}+\|  D  B _{t}^{(m+n-2)}\|^{2}_{L^{2}}\B]+\f{1}{2}\|  D ^{2}u _{t}^{(m+n-2)}\|^{2}_{L^{2}}+\f{1}{2}\|  D ^{2}B _{t}^{(m+n-2)}\|^{2}_{L^{2}}
\\\leq&
CF_{ m+n-2,1}^{\f{ 4(m+n-2)+3  }
{4(m+n-2)+1}}+(\| D ^{   \alpha}u_{t}^{(m+n-2)}\|_{L^{2}}^{2}+\| D ^{   \alpha}B_{t}^{(m+n-2)}\|_{L^{2}}^{2})
^{\f{ (4(m+n-2)+3 )}
{4(m+n-2)+1}} \\&+4\eta\sum_{j=0 }^{m+n-3} F_{j,1}^{\f{4  (m+n-2-j ) }{ 4j +1}}G_{j,1}.
\ea\ee
By means of     the induction hypotheses \eqref{4.41v1}, one has
$$
\f{d}{dt}F_{j,1}+G_{j,1}\leq CF_{j,1}^{\f{4j + 3 }{4j +1}},
$$
from which it follows that
\be\label{4.53}
\f{d}{dt}
F_{\ell,1}^{\f{4(m+n-2)+1}{4j +1}}+ F_{j,1}^{\f{4  (m+n-2-j ) }{ 4j +1}}G_{j,1}\leq CF_{j,1}^{\f{4  (m+n-2)+ 3 }{4j +1}},
\ee
where we have used $j+1\leq m+n-2$.\\
Combining \eqref{4.52} and \eqref{4.53}, we arrive at
\be\ba
& \f{d}{dt}\B[ \|  D  u _{t}^{(m+n-2)}\|^{2}_{L^{2}}+\|  D  u _{t}^{(m+n-2)}\|^{2}_{L^{2}}+\f{4j +1}{4(m+n-2)+1} \sum_{j=0 }^{m+n-3}F_{\ell,1}^{\f{4(m+n-2)+1}{4j +1}}\B]
\\&+\f12\|  D ^{2}u _{t}^{(m+n-2)}\|^{2}_{L^{2}}+\f12\|  D ^{2}B _{t}^{(m+n-2)}\|^{2}_{L^{2}}+ \sum_{j=0 }^{m+n-3}F_{j,1}^{\f{4  (m+n-2-j ) }{ 4j +1}}G_{j,1}
\\\leq&
CF_{ m+n-2,1}^{\f{ 4(m+n-2)+3  }
{4(m+n-2)+1}}+(\| D ^{   \alpha}u_{t}^{(m+n-2)}\|_{L^{2}}^{2}+\| D ^{   \alpha}B_{t}^{(m+n-2)}\|_{L^{2}}^{2})
^{\f{ (4(m+n-2)+3 )}
{4(m+n-2)+1}}  \\&+4\eta\sum_{j=0 }^{m+n-3} F_{j,1}^{\f{4  (m+n-2-j ) }{ 4j +1}}G_{j,1},
\ea\ee
which leads to that
\be\ba \label{4.60}
&\f{d}{dt}F_{ m+n-2,1}+G_{ m+n-2,1}\leq CF_{ m+n-2,1}^{\f{ 4(m+n-2)+3  }
{4(m+n-2)+1}}.
\ea\ee
Keeping this in mind, we apply
Lemma \ref{hofflemma} and  \eqref{4.56} to write
$$G_{ m+n-2,1}\in L^{\f{1}{4(m+n-2)+3  }}, \| D ^{ 2}u _{t}^{(m+n-2)}\|_{L^{2}}^{2}, \| D ^{ 2}B _{t}^{(m+n-2)}\|_{L^{2}}^{2}\in L^{\f{1}{4(m+n-2)+3  }}.$$
Hence the statements for Case 3 are proved.

$Case$ 4. $i=m+n-1.$\\
In view of \eqref{4.23}, we write
\be\ba\label{4.61v1}
&F_{ m+n-1,0}= \| u_{t}^{(m+n-1)}\|_{L^{2}}^{2}+\| B_{t}^{(m+n-1)}\|_{L^{2}}^{2}+\sum_{\ell=0}^{m+n-2}
c_{\ell,0}F_{\ell,0}^{\f{4(m+n-1)-1}{4\ell -1}}+\sum_{\ell=0}^{m+n-2}
F_{\ell,1}^{\f{4(m+n-1)-1}{4\ell +1}}, \\
&G_{m+n-1,  0}= \f12\| D  u_{t}^{(m+n-1)}\|_{L^{2}}^{2}+\f12\| D  B_{t}^{(m+n-1)}\|_{L^{2}}^{2}+\f12\sum_{\ell=0}^{m+n-2}
c_{\ell,1}F_{\ell,1}^{\f{4(m+n-1-\ell)-2  }{4\ell +2 -1}}G_{\ell,1},
\ea\ee
where
$$\ba
c_{\ell,0}=\f{4\ell -1}{4(m+n-1)-1},
c_{\ell,1}=\f{4\ell +1}{4(m+n-1)-1}.
\ea$$
From Theorem \ref{the1.2}, we know that
$$\| u_{t}^{(m+n-1)}\|_{L^{2}}^{2}, \| B_{t}^{(m+n-1)}\|_{L^{2}}^{2}\in L^{\f{1}{4(m+n-1)-1}}.$$
Owing to the induction hypotheses \eqref{4.41v1}, one gets
$F_{\ell,0}\in L^{\f{1}{4\ell  -1}}$ and $F_{\ell,1}\in L^{\f{1}{4\ell  +1}}$, where $\ell+0\leq m+m-2$ was used. \\
Before proceeding further, we write
$$
\sum_{\ell=0}^{m+n-2}
F_{\ell,1}^{\f{4(m+n-1)-1}{4\ell +1}}
=\sum_{\ell=0}^{m+n-3}
F_{\ell,1}^{\f{4(m+n-1)-1}{4\ell +1}}+F_{m+n-2,1}^{\f{4(m+n-1)-1}{4(m+n-2)+1}}.
$$
Thanks to  \eqref{4.56} in case 3, we know that $F_{ m+n-2,1}\in   L^{\f{1}{4(m+n-2) +1}}$.
Consequently, we assert
\be\label{4.61}
F_{ m+n-1,0}\in L^{\f{1}{4(m+n-1)-1}}.\ee
Combining
\eqref{3.31v1} and \eqref{4.61}, we find
 \be\ba\label{4.62}
&\f{d}{dt}(\| u _{t}^{(m+n-1)}\|^{2}_{L^{2}}+\| B _{t}^{(m+n-1)}\|^{2}_{L^{2}})+\f{1}{2}\| D  u _{t}^{(m+n-1)}\|_{L^{2}}^{2}+\f{1}{2}\| D  B _{t}^{(m+n-1)}\|_{L^{2}}^{2} \\
\leq&C\sum_{j>0}(\|u_{t}^{(j)}\|
_{L^{2}}^{2}+\|B_{t}^{(j)}\|
_{L^{2}}^{2})^{\f{ (4(m+n-1)+1)}{4j -1}}
\\&+C\sum_{j>0}(\| D  u_{t}^{(m+n-1-j)} \|_{L^{2}}^{2}+\| D  B_{t}^{(m+n-1-j)} \|_{L^{2}}^{2})^{\f{ (4(m+n-1)+1)}{4(m+n-1-j)+1}}
\\&+8\eta\sum_{j>0}(\| D  u_{t}^{(m+n-1-j)} \|_{L^{2}}^{2}+\| D  u_{t}^{(m+n-1-j)} \|_{L^{2}}^{2})^{\f{2 (2j-1)}{  4(m+n-1-j)+1}}
\\&\times(\| D ^{2}u_{t}^{(m+n-1-j)} \|^{2}_{L^{2}}+\| D ^{2}B_{t}^{(m+n-1-j)} \|^{2}_{L^{2}})\\
\leq&C\sum_{j=1}^{m+n-2}F_{j,0}
^{\f{ 4(m+n-1)+1 }{4j -1}}
+C(\| u _{t}^{(m+n-1)}\| _{L^{2}}^{2}+\| B _{t}^{(m+n-1)}\| _{L^{2}}^{2})^{\f{ [4(m+n-1)+1] }{4(m+n-1)-1}}
\\&+C\sum_{j=0}^{m+n-2}F_{j,1} ^{\f{ 4(m+n-1)+1 }{4j +1}}
+8\eta\sum_{j=0}^{m+n-2}F_{j,1}^{\f{  4 (m+n-1-j)-2 }{  4 j +1}}G_{j,1}.
\ea
\ee
In view of definition of $F_{ m+n-1,0}$  in \eqref{4.61}, one arrives at
\be\ba
\sum_{j=1}^{m+n-2}F_{j,0}
^{\f{ 4(m+n-1)+1 }{4j -1}}\leq C F_{ m+n-1,0}^{\f{4(m+n-1)+1}{4(m+n-1)-1}},\\
\sum_{j=1}^{m+n-3}F_{j,1}
^{\f{ 4(m+n-1)+1 }{4j +1}}\leq C F_{ m+n-1,0}^{\f{4(m+n-1)+1}{4(m+n-1)-1}},\\
F_{ m+n-2,1}^{\f{ 4(m+n-2)+3  +2}
{4(m+n-2)+1}}\leq C F_{ m+n-1,0}^{\f{4(m+n-1)+1}{4(m+n-1)-1}}.
 \ea\ee
Substituting  this into \eqref{4.62}, we infer that
 \be\ba\label{4.64}
& \f{d}{dt}\B[\| u _{t}^{(m+n-1)}\|^{2}_{L^{2}}+\| B _{t}^{(m+n-1)}\|^{2}_{L^{2}}\B]+\f12\| D  u _{t}^{(m+n-1)}\|_{L^{2}}^{2} +\f12\| D  B _{t}^{(m+n-1)}\|_{L^{2}}^{2}\\
\leq& C F_{ m+n-1,0}^{\f{4(m+n-1)+1}{4(m+n-1)-1}}
+C\| u _{t}^{(m+n-1)}\| _{L^{2}}^{\f{ 2[4(m+n-1)+1] }{4(m+n-1)-1}}
 +8\eta\sum_{j=0}^{m+n-2}F_{j,1}^{\f{  4 (m+n-1-j)-2 }{  4 j +1}}G_{j,1}.
\ea
\ee
For $j+1\leq m+n-2$,   by  the induction hypotheses \eqref{4.41v1}, we observe that
$$\ba
&\f{d}{dt}F_{j,1}+G_{j,1}\leq CF_{j,1}^{\f{4j + 3 }{4j +1}},
\ea$$
from which it follows that
\be\label{4.65}
c_{j,1}\f{d}{dt}
F_{j,1}^{\f{4(m+n-1)-1}{4j +1}}+F_{j,1}^{\f{4(m+n-1-j)-2  }{4j +1}}G_{j,1}\leq CF_{j,1}^{\f{4(m+n-1) + 1}{4j +1}}\leq C F_{ m+n-1,0}^{\f{4(m+n-1)+1}{4(m+n-1)-1}}.
\ee
For $j= m+n-2$, we employ inequality \eqref{4.60} proved in case 3 to obtain
\be\ba
&\f{d}{dt}F_{ m+n-2,1}+G_{ m+n-2,1}\leq CF_{ m+n-2,1}^{\f{ 4(m+n-2)+3  }
{4(m+n-2)+1}},
\ea\ee
which leads to
\be\ba\label{4.66}
&F_{m+n-2,1}^{\f{   2 }{  4 (m+n-2)+1}}\f{d}{dt}F_{ m+n-2,1}+F_{m+n-2,1}^{\f{   2 }{  4 (m+n-2)+1}}G_{ m+n-2,1}\leq CF_{ m+n-2,1}^{\f{ 4(m+n-2)+3  +2}
{4(m+n-2)+1}}\leq C F_{ m+n-1,0}^{\f{4(m+n-1)+1}{4(m+n-1)-1}}.
\ea\ee
For $j+0\leq m+n-2$,  in light of the induction hypotheses \eqref{4.41v1}  and \eqref{4.13}, one notices that
$$\ba
&\f{d}{dt}F_{j,0}+G_{j,0}\leq CF_{j,0}^{\f{4j + 1}{4j -1}},
\ea$$
which means that
\be\ba\label{4.63}
&c_{j,0}\f{d}{dt}F_{j,0}^{\f{4(m+n-1)-1}{4j -1}}+F_{j,0}^{\f{4(m+n-1-j) }{4j -1}}G_{j,0}\leq CF_{j,0}^{\f{4(m+n-1 )+ 1}{4j -1}}\leq C F_{ m+n-1,0}^{\f{4(m+n-1)+1}{4(m+n-1)-1}}.
\ea\ee
Collecting the estimates  \eqref{4.65}, \eqref{4.66}, \eqref{4.63} and \eqref{4.64}, we end up with
$$\ba
& \f{d}{dt}\B[\| u _{t}^{(m+n-1)}\|^{2}_{L^{2}} +\| B _{t}^{(m+n-1)}\|^{2}_{L^{2}} +\sum_{\ell=0}^{m+n-2}
c_{\ell,0}F_{\ell,0}^{\f{4(m+n-1)-1}{4\ell -1}}+\sum_{\ell=0}^{m+n-2}
c_{\ell,1}F_{\ell,1}^{\f{4(m+n-1)-1}{4\ell +1}} \B]\\&+\sum_{j=0}^{m+n-3}F_{j,1}^{\f{4(m+n-1-j)-2  }{4j +1}}G_{j,1}+F_{m+n-2,1}^{\f{   2 }{  4 (m+n-2)+1}}G_{ m+n-2,1} \\&+\f12\| D  u _{t}^{(m+n-1)}\|_{L^{2}}^{2} +\f12\| D  B _{t}^{(m+n-1)}\|_{L^{2}}^{2}\\
\leq& C F_{ m+n-1,0}^{\f{4(m+n-1)+1}{4(m+n-1)-1}}
+C\| u _{t}^{(m+n-1)}\| _{L^{2}}^{\f{ 2[4(m+n-1)+1] }{4(m+n-1)-1}}
 +8\eta\sum_{j=0}^{m+n-2}F_{j,1}^{\f{  4 (m+n-1-j)-2 }{  4 j +1}}G_{j,1}.
\ea
$$
Consequently, we choose $\eta$ small to obtain
$$\ba
\f{d}{dt}F_{ m+n-1,0}+G_{ m+n-1,0}\leq C F_{ m+n-1,0}^{\f{4(m+n-1)+1}{4(m+n-1)-1}}.
\ea
$$
At this stage,  we apply
Lemma \ref{hofflemma} and  \eqref{4.61} to get
\be\ba
G_{ m+n-1,0}\in L^{\f{1}{4(m+n-1)+1}},
\| D  u _{t}^{(m+n-1)}\|_{L^{2}}^{2}, \| D  B _{t}^{(m+n-1)}\|_{L^{2}}^{2}\in L^{\f{1}{4(m+n-1)+1}}.
\ea
\ee
Finally, no mater in which case, one always get
\be\ba
\| D ^{ (m+n-i)}u _{t}^{(i)}\|_{L^{2}}^{2}, \| D ^{ (m+n-i)}B _{t}^{(i)}\|_{L^{2}}^{2}\in L^{\f{1}{4i+2(m+n-i)-1}}.
\ea
\ee
We achieve the proof of Case 4.
Hence, the proof of this theorem is completed.
\end{proof}
Finally, we present the proof of Corollary \ref{coro}.
\begin{proof}[Proof of Corollary \ref{coro}]
First, we assert that, for $k\in \mathbb{N}$,
\be\label{while}
 D ^{k}u,  D ^{k}B\in L^{\f{1}{k +1}}(0,T;L^{\infty}(\mathbb{R}^{3})). \ee
Indeed, we deduce from the Gagliardo-Nirenberg inequality and Sobolev embedding theorem that
$$\| D ^{k}u\|
^{ }_{L^{\infty}}\leq C\| u\|
^{\f{1}{2(k+1)} }_{L^{6}} \| D ^{k+2}u\|
^{ \f{2k+1}{2(k+1)} }_{L^{2}}\leq  C\| \nabla u\|
^{\f{1}{2(k+1)} }_{L^{2}} \| D ^{k+2}u\|
^{ \f{2k+1}{2(k+1)} }_{L^{2}}.$$
By integrating it with respect to time, we observe that
$$\ba
\int_{0}^{T}\| D ^{k}u\|
^{\f{1}{k+1} }_{L^{\infty}}dt\leq&   C\int_{0}^{T}\| \nabla u\|
^{\f{1}{2(k+1)^{2}} }_{L^{2}} \| D ^{k+2}u\|
^{ \f{2k+1}{2(k+1)^{2}} }_{L^{2}}dt\\
\leq&   C\B(\int_{0}^{T}\| \nabla u\|
^{2 }_{L^{2}}dt\B)^{\f{1}{4(k+1)^{2}}} \B(\int_{0}^{T}\| D ^{k+2}u\|
^{ \f{2 }{2 k+3} }_{L^{2}}dt\B)^{1-\f{1}{4(k+1)^{2}}}.
\ea$$
Hence, we have $ D ^{k}u \in L^{\f{1}{k +1}}(0,T;L^{\infty}(\mathbb{R}^{3})).$
Likewise,  $ D ^{k}B \in L^{\f{1}{k +1}}(0,T;L^{\infty}(\mathbb{R}^{3})).$
\\
(1) According to the  interpolation inequality and Sobolev embedding theorem, we remark that
$$\ba
\| u\|_{L^{q}}\leq \| u\|_{L^{6}}^{\f6q}\| u\|^{1-\f6q}_{L^{\infty}}\leq C\| \nabla u\|_{L^{2}}^{\f6q}\| u\|^{1-\f6q}_{L^{\infty}}, 6\leq q \leq\infty.
\ea$$
Integrating the above inequality with respect to time variable and using the H\"older inequality, we get
 \be\ba
\int_{0}^{T}\| u\|_{L^{q}}^{\f{q}{q-3}}dt \leq C \B(\int_{0}^{T}\|\nabla u\|_{L^{2}}^{2}dt\B)^{\f{3}{q-3}}\B( \int_{0}^{T} \|u\|_{L^{\infty}}dt\B)^{1-\f{3}{q-3}}.
\ea\ee
We derive from \eqref{while} that $  u   \in L^{\f{q}{q-3}}(0,T;L^{q} (\mathbb{R}^{3})),  6\leq q\leq\infty.$ Repeating the same deduction, we find $  B   \in L^{\f{q}{q-3}}(0,T;L^{q} (\mathbb{R}^{3})),  6\leq q\leq\infty.$ \\
(2)
We derive from the interpolation inequality that
$$\ba
\| D ^{k}u\|_{L^{q}}\leq \| D ^{k}u\|_{L^{2}}^{\f2q}\| D ^{k}u\|^{1-\f2q}_{L^{\infty}}, 2\leq q \leq\infty.
\ea$$
Integrating the above inequality in time and using the H\"older inequality, we verify that
\be\ba
\int_{0}^{T}\| D ^{k}u\|_{L^{q}}^{\f{q}{kq+q-3}}dt\leq & \int_{0}^{T}\| D ^{k}u\|_{L^{2}}^{\f{2}{kq+q-3}}\| D ^{k}u\|^{\f{q-2}{kq+q-3}}_{L^{\infty}}dt
\\
\leq& \B(\int_{0}^{T}\| D ^{k}u\|_{L^{2}}^{\f{2}{2k -1}}dt\B)^{\f{2k -1}{kq+q-3}}\B(\int_{0}^{T}\| D ^{k}u\|
^{\f{1}{k +1}}_{L^{\infty}}dt
\B) ^{1-\f{2k -1}{kq+q-3}}.
\ea\ee
Consequently, we find
$ D ^{k}u   \in L^{\f{q}{k(q+1)-3}}(0,T;L^{q} (\mathbb{R}^{3})), 2\leq q\leq\infty.$ Along the same line, we have $ D ^{k}B   \in L^{\f{q}{k(q+1)-3}}(0,T;L^{q} (\mathbb{R}^{3})), 2\leq q\leq\infty.$ The proof of this corollary is completed.
\end{proof}

\end{document}